\documentclass[a4paper,10pt]{amsart}
\usepackage[all]{xy}
\usepackage{latexsym}
\usepackage{amssymb}
\usepackage{amscd}
\usepackage{amsthm}
\usepackage{amsmath}
\usepackage{hyperref}

\newcommand{\dis}{\displaystyle}
\newcommand{\lb}{\linebreak}

\newcommand{\dual}{^{\ast}}

\newcommand{\dotcup}{\stackrel{\cdot}{\cup}}

\newcommand{\bbN}{\mathbb{N}}

\newcommand{\m}{\mathfrak{m}}

\newcommand{\Max}{\operatorname{Max}}    
\newcommand{\Zg}{\operatorname{Zg}}    
\newcommand{\MaxZg}{\operatorname{MaxZg}}    
\newcommand{\Mod}{\operatorname{Mod}}    

\newcommand{\add}{\operatorname{add}}
\newcommand{\Add}{\operatorname{Add}}
\newcommand{\ann}{\operatorname{ann}}
\newcommand{\Hom}{\operatorname{Hom}}

\newcommand{\Dsum}{\operatorname{Sum}}
\newcommand{\End}{\operatorname{End}}

\newcommand{\Ext}{\operatorname{Ext}}
\newcommand{\Tor}{\operatorname{Tor}}
\newcommand{\Tr}{\operatorname{Tr}}
\newcommand{\Cl}{\operatorname{Cl}}
\newcommand{\FC}{\operatorname{FC}}
\newcommand{\Cogen}{\operatorname{Cogen}}
\newcommand{\Gen}{\operatorname{Gen}}
\newcommand{\Pres}{\operatorname{Pres}}
\newcommand{\Ker}{\operatorname{Ker}}

\newcommand{\Img}{\operatorname{Im}}
\newcommand{\Coker}{\operatorname{Coker}}

\newcommand{\li}{\varinjlim}

\newcommand{\HomJT}[3]{\mbox{\rm{Hom}}_{#1}(#2,#3)}
\newcommand{\ExtJT}[4]{\mbox{\rm{Ext}}^{#1}_{#2}(#3,#4)}

\newcommand{\A}{\mathcal{A}}

\newcommand{\C}{\mathcal{C}}
\newcommand{\D}{\mathcal{D}}
\newcommand{\E}{\mathcal{E}}
\newcommand{\F}{\mathcal{F}}
\newcommand{\G}{\mathcal{G}}
\newcommand{\clH}{\mathcal{H}}

\newcommand{\T}{\mathcal{T}}

\newcommand{\X}{\mathcal{X}}

\newcommand{\Modr}[1]{\mathrm{Mod}\textrm{-}{#1}}
\newcommand{\Modl}[1]{{#1}\textrm{-}\mathrm{Mod}}

\newcommand{\modr}[1]{\mathrm{mod}\textrm{-}{#1}}
\newcommand{\modl}[1]{{#1}\textrm{-}\mathrm{mod}}

\newcommand{\ModR}{\mathrm{Mod}\textrm{-}R}

\newcommand{\SInj}{S\textrm{-}\mathrm{Inj}}

\newcommand{\rmod}[1]{\mbox{\rm{Mod}--}{#1}}
\newcommand{\simpR}{\text{simp-}R}
\newcommand{\SPInj}{S\textrm{-}\mathrm{PInj}}
\newcommand{\SPProj}{S\textrm{-}\mathrm{PProj}}
\newcommand{\Sproj}{S\textrm{-}\mathrm{proj}}
\newcommand{\SProj}{S\textrm{-}\mathrm{Proj}}
\newcommand{\SFlat}{S\textrm{-}\mathrm{Flat}}
\newcommand{\SAbs}{S\textrm{-}\mathrm{Abs}}

\newcommand{\Ab}{\mathrm{Ab}}

\newcommand{\fp}{\operatorname{fp}}
\newcommand{\soc}{\operatorname{soc}}

\newcommand{\op}{^\textrm{op}}

\newcommand{\Ann}{\mathrm{Ann}}
\newcommand{\Prod}{\mathrm{Prod}} 
\newcommand{\pd}{\mathrm{pd}} 
\newcommand{\id}{\mathrm{id}} 
\newcommand{\PE}{\mathrm{PE}}

\newcommand{\Lang}{\mathcal{L}}
\newcommand{\Th}{\mathrm{Th}}

\theoremstyle{plain}
\newtheorem{thm}{Theorem}[section]
\newtheorem{lem}[thm]{Lemma}
\newtheorem{prop}[thm]{Proposition}
\newtheorem{cor}[thm]{Corollary}
\newtheorem{ques}[thm]{Question}
\newtheorem{conj}[thm]{Conjecture}
\newtheorem{fact}[thm]{Fact}

\theoremstyle{definition}
\newtheorem{defn}[thm]{Definition}

\theoremstyle{remark}
\newtheorem{rem}[thm]{Remark}

\newtheorem{expl}[thm]{Example}

\begin{document}

\footskip30pt

\date{}

\title{Pure projective tilting modules}

\author[S.\ Bazzoni]{Silvana Bazzoni}
\address[Silvana Bazzoni]{%
Dipartimento di Matematica \\
Universit\`a di Padova \\
Via Trieste 63, 35121 Padova (Italy)}
\email{bazzoni@math.unipd.it}

\author[I.\ Herzog]{Ivo Herzog}
\address[Ivo Herzog]{%
The Ohio State University at Lima\\
4240 Campus Drive \\
Lima, OH 45804 (USA)}
\email{herzog.23@osu.edu}

\author[P.\ Prihoda]{Pavel P\v{r}\'{i}hoda}
\address[Pavel P\v{r}\'{i}hoda]{%
Faculty of Mathematics and Physics\\
Department of Algebra\\
Sokolovsk\'{a} 83\\
186 75 Praha 8 (Czech Republic)}
\email{prihoda@karlin.mff.cuni.cz}

\author[J.\ \v{S}aroch]{Jan \v{S}aroch}
\address[Jan \v{S}aroch]{%
Faculty of Mathematics and Physics\\
Department of Algebra\\
Sokolovsk\'{a} 83\\
186 75 Praha 8 (Czech Republic)}
\email{saroch@karlin.mff.cuni.cz}

\author[J.\ Trlifaj]{Jan Trlifaj}
\address[Jan Trlifaj]{%
Faculty of Mathematics and Physics\\
Department of Algebra\\
Sokolovsk\'{a} 83\\
186 75 Praha 8 (Czech Republic)}
\email{trlifaj@karlin.mff.cuni.cz}

\thanks{The first named author is supported by Progetto di Eccellenza Fondazione Cariparo ''Algebraic Structures and their applications;'' the second by a Fulbright Distinguished Chair Grant and NSF Grant DMS 12-01523; the third, fourth and fifth named authors were supported by the grant GA\v{C}R 14-15479S.}



\begin{abstract}
Let $T$ be a $1$-tilting module whose tilting torsion pair $(\T, \F)$ has the property that the heart $\clH_t$ of the induced $t$-structure (in the derived category $\D (\Modr R)$) is Grothendieck. It is proved that such tilting torsion pairs are characterized in several ways: (1) the $1$-tilting module $T$ is pure projective; (2) $\T$ is a definable subcategory of $\Modr R$ with enough pure projectives, and (3) both classes $\T$ and $\F$ are finitely axiomatizable. 

This study addresses the question of Saor\'{i}n that asks whether the heart is equivalent to a module category, i.e., whether the pure projective $1$-tilting module is tilting equivalent to a finitely presented module. The answer is positive for a Krull-Schmidt ring and for a commutative ring, every pure projective $1$-tilting module is projective. A criterion is found that yields a negative answer to Saor\'{i}n's Question for a left and right noetherian ring. A negative answer is also obtained for a Dubrovin-Puninski ring, whose theory is covered in the Appendix. Dubrovin-Puninski rings also provide examples of (1) a pure projective $2$-tilting module that is not classical; (2) a finendo quasi-tilting module that is not silting; and (3) a noninjective module $A$ for which there exists a left almost split morphism $m: A \to B,$ but no almost split sequence beginning with $A.$ 
\end{abstract}

\maketitle

\setcounter{tocdepth}{1}
\tableofcontents

\pagestyle{plain}

\section*{Introduction}

Beilinson, Bernstein and Deligne~\cite{BBD} showed how a triangulated category $\D$ may be equipped with additional structure - a $t$-structure - that gives rise to a subcategory $\clH_t$ of $\D$ - the heart of the $t$-structure - on which the triangulated structure of $\D$ induces the structure of an abelian category. If $\A$ is an abelian category Happel, Reiten and Smal\o $\;$ (HRS) found a general method~\cite[\S I.2]{HRS} that associates to a torsion pair $(\T, \F)$ of $\A$ a $t$-structure on the derived category $\D (\A).$ The torsion pair $(\T, \F)$ in $\A$ induces a torsion pair $(\F [1], \T [0])$ in the heart $\clH_t$ of this $t$-structure with the property that the torsion class $\F [1] \cong \F$ is equivalent to the given torsion free class in $\A,$ while the torsion free class $\T [0] \cong \T$ is equivalent to the given torsion class. 

In the case when the ambient abelian category is a Grothendieck category $\A = \G,$ and so has direct limits, Parra and Saor\'{i}n (\cite{PS} and~\cite[Thm.\ 1.2]{PS16}) showed that the heart $\clH_t$ of the $t$-structure induced by $(\T, \F)$ on $\D (\G)$ is itself a Grothendieck category if and only if the torsion free class $\F$ is closed under direct limits in $\G.$ We consider the condition that $\clH_t$ be equivalent to a Grothendieck category as interesting and we pursue the closely related question of Saor\'{i}n who asked whether the heart $\clH_t$ is, in that case, equivalent to a module category, that is, to a Grothendieck category with a finitely generated projective generator.

The HRS derived equivalence is best understood for a tilting torsion pair. If $R$ is a ring and $T$ a $1$-tilting right $R$-module, then the associated tilting class $\T = \Gen T = T^{\perp} \subseteq \Modr R$ constitutes a torsion class~\cite[Lemma 14.2]{GT12} in $\Modr R.$ The associated tilting torsion pair $(\T, \F)$ induces, by means of the HRS theory, a $t$-structure in the derived category $\D (\Modr R).$ Our article is motivated by the following question posed by M.\ Saor\'{i}n.  

\begin{ques} \label{Q:Saorin}
{\rm ~\cite[Question 5.5]{PS}} Let $R$ be a ring and $T$ a $1$-tilting right $R$-module. Suppose that the heart $\clH_t$ of the $t$-structure on $\D (\Modr R)$ induced by $(\Gen T, \F)$ is a Grothendieck category. Does it follow that the heart $\clH_t$ is equivalent to the module category $\Modr S$ for some ring $S?$ 
\end{ques}

The question of whether a torsion pair $(\T, \F)$ in $\Modr R$ (resp., $\modr R$) induces a heart in $\D (\Modr R)$
(resp., $\modr R$) equivalent to a module category has been explored by several authors (for example, \cite{HRS, CMT}). One of the first tasks of this article is to formulate Saor\'{i}n's Question in terms of module theory.  We show (Corollary~\ref{C:pure-proj}) that if $T$ is a $1$-tilting module, then the heart $\clH_t \subseteq \D (\Modr R)$ of the $t$-structure associated to the tilting torsion pair $(\Gen T, \F)$ in $\Modr R$ is a Grothendieck category if and only if $T$ is pure projective. In that case, Corollary~\ref{C:classical} implies that the heart $\clH_t$ is equivalent to a module category $\Modr S$ if and only if $T$ is a {\em classical} $1$-tilting module, that is, a $1$-tilting module tilting equivalent to a finitely presented module. This suggests the following version of Saor\'{i}n's Question.

\begin{ques} \label{Q:Saorin-pure-proj}
{\rm (Saor\'{i}n, pure projective version)} For which rings $R$ is every pure projective $1$-tilting module classical?
\end{ques}

Saor\'{i}n's Question comes about from considerations that arise in the derived Morita theory of rings, but the notion of a pure projective $1$-tilting module is a natural one to study in any case, because it represents the situation dual to the result of the first author~\cite{B03} (generalized by \v{S}\'{t}ov\'{i}\v{c}ek~\cite{St06} to all cotilting modules), that every $1$-cotilting module is pure injective. On the other hand, there are many examples of $1$-tilting modules that are not pure projective. 

The pure projective version of Saor\'{i}n's Question is related to a more general question considered by the second author and Ph.\ Rothmaler in their work on definable subcategories with enough pure projective modules. It is well known that if $\C \subseteq \Modr R$ is a definable subcategory and $M \in \C,$ then there exists a pure monomorphism $m : M \to U$ with $U$ a pure injective module in $\C.$ The articles~\cite{HR, HeRo} are devoted to the investigation of those definable subcategories $\C \subseteq \Modr R$ with {\em enough pure projective modules,} in the sense that for every module $M \in \C,$ there exists a pure epimorphism $e :V \to M$ with $V$ a pure projective module in $\C.$ Such definable subcategories of $\Modr R$ arise ''classically'' (cf.~\cite{CB94, HK}) as the categories $\C$ obtained by taking the closure under direct limits ${\dis \C = \lim_{\to} (\X),}$ where $\X \subseteq \modr R$ is a covariantly finite subcategory of finitely presented modules. As part of our analysis of a pure projective $1$-tilting module $T,$ we show (Theorem~\ref{T:pure-proj}) that the definable subcategory $\T = \Gen T = T^{\perp}$ of $\Modr R,$ has enough pure projective modules, and (Theorem~\ref{T:classical}) that if $T$ is a classical $1$-tilting module, then ${\dis \T = \lim_{\to} (\X),}$ with $\X \subseteq \modr R$ covariantly finite.

The pure projective version of Saor\'{i}n's Question has a positive answer for a large class of rings. By 
Corollary~\ref{C:Krull-Schmidt}, if $R$ is a ring over which every pure projective module is a direct sum of finitely presented modules, then every $1$-tilting module $T$ is classical. This includes the {\em Krull-Schmidt} rings, that is, the rings over which every finitely presented module admits a direct sum decomposition into modules with a local endomorphism. For the general case of a commutative ring, we use a henselization argument (Theorem~\ref{T:commutative-case}) to show that every pure projective $1$-tilting module is projective. Considerably easier proofs are found for a noetherian commutative ring or an arithmetic ring.

But the answer to Saor\'{i}n's Question is not affirmative in general. In Theorem~\ref{T:noetherian-case}, we identify a criterion  for an idempotent ideal $I$ of $R,$ finitely generated on the left, sufficient to yield a pure projective $1$-tilting module $T$ that is not classical. The construction is an elaboration of J.\ Whitehead's method of producing a projective module whose trace is such an idempotent ideal. The tilting class that arises is given by $\T = \{ M \in \Modr R \; | \; M = MI \}.$ For example, if $R = U(L)$ is the universal enveloping algebra of the Lie algebra $L = {\rm sl}(2,\mathbb{C})$ and $I$ is the annihilator of the trivial module $\mathbb{C}_L,$ then the conditions of Theorem~\ref{T:noetherian-case} are satisfied and one obtains a nonclassical pure projective $1$-tilting module $T_{U(L)}$ over a ring that is both left and right noetherian.

A preliminary section of the article is devoted to the role that definable subcategories of $\Modr R$ play in torsion theory. If $(\T, \F)$ is a torsion pair in $\Modr R,$ then the torsion free class $\F$ always has direct limits, given by the torsion free quotient module of a direct limit in $\Modr R.$ By the work of Parra and Saor\'{i}n, the torsion pair induces a Grothendieck heart in $\D (\Modr R)$ if and only if the direct limits in $\F$ coincide with those in the ambient category $\Modr R,$ that is, if $\F$ is {\em closed under direct limits in} $\Modr R.$ This is equivalent to the condition that the torsion free class $\F$ be a definable subcategory of $\Modr R.$ When $T$ is a $1$-tilting module, the tilting torsion class $\T = \Gen T \subseteq \Modr R$ is also a definable subcategory~\cite{BH}. We call a torsion pair $(\T, \F)$ in $\Modr R$ both of whose constituent categories $\T$ and $\F$ are definable {\em elementary} and show in Theorem~\ref{T:elementary torsion pair} that this is equivalent to the condition that the associated torsion radical is definable in the language of right $R$-modules by a positive primitive formula. This implies that both $\T$ and $\F$ are finitely axiomatizable.

The first example of a definable subcategory with enough pure projectives that is not the direct limit closure of a covariantly finite subcategory of $\modr R$ was obtained~\cite{HR} over a {\em Dubrovin-Puninski} ring. If $R$ is a nearly simple uniserial domain, Puninski (cf.~\cite[Chapter 14]{Pun}) proved that there exists, up to isomorphism, a unique cyclically presented torsion module $X_R.$ A Dubrovin-Puninski ring~\cite{DP} is a ring $S$ that arises as the endomorphism ring $S = \End_R X_R$ of such a module. For this reason, we study left modules over $S$ and devote an appendix to a comprehensive treatment of the category $\Modl S.$ The 12 definable subcategories of $\Modl S$ are classified: 10 have enough pure projective modules, and $6$ of those do not arise from covariantly finite subcategories. Among the 12 definable subcategories, three are tilting classes. Beside the trivial tilting class $\Modl S$ one finds a $1$-tilting class and a $2$-tilting class, neither of which arise from classical tilting modules. 

Dubrovin-Puninski rings may seem pathological from the point of view of the existing theory, but we hope to convince the reader that they are actually quite well-behaved and provide a constant source of counterexamples.  For example, the fourth author has recently proved~\cite{JanIM} that if a nonprojective module ${_S}C$ is the codomain of a right almost split morphism in the category $\Modl S,$ then there exists an almost split sequence in $\Modl S$ ending with $C.$ Over a Dubrovin-Puninski ring, however, there exists (Proposition~\ref{P:left asm}) in $\Modl S$ a noninjective module $A$ and a left almost split morphism $m: A \to B,$ but no almost split sequence beginning with $A.$ 

A section of the article is devoted to the study of $\tau$-rigid modules~\cite{W} and silting modules~\cite{AMV}, and their relationship to the finendo and quasi-tilting properties. We prove (Corollary~\ref{cor}) that every quasi-tilting module is $\tau$-rigid and conditions are given in Proposition~\ref{prop} and Corollary~\ref{cor} under which $\tau$-rigidity is equivalent to $\Gen T \subseteq T^{\perp}.$ Regarding the silting property, we give an example (Example~\ref{prod}) of a quasi-tilting module over a commutative von Neumann regular ring that is not silting and an example (Example~\ref{dubpun}) of a finendo quasi-tilting module over a Dubrovin-Puninski ring that is not silting.  
\bigskip

In what follows, $R$ will be an associative ring with unit; $\Modr R$ ($\Modl R$) will denote the category of right (left) $R$-modules; $\modr R$  ($\modl R$) the subcategory of finitely presented right (left) $R$-modules. The category of abelian groups is denoted by $\Ab.$ For a subcategory $\C$ of $\Modr R,$ $\Add (\C)$ (resp., $\add (\C)$) will denote the class of modules isomorphic to a summand of a (resp., finite) direct sum of modules in $\C.$ If the class $\C = \{ T \}$ is singleton, we write $\Add T$ (resp., $\add T$). Similarly, $\Gen (\C)$ (resp., $\Gen T$) will denote the class of epimorphic images of direct sums of modules in $\C$ (resp., of copies of $T$) and $\Pres T \subseteq \Gen T$ will denote the class of those modules $M$ that admit a $T$-{\em presentation,} that is, an exact sequence $T_1 \to T_0 \to M \to 0,$ where $T_0,$ $T_1 \in \Add T.$ The language for right $R$-modules is $\Lang (R) = (+,-,0,r)_{r \in R};$ the standard axioms for a right $R$-module are expressible in $\Lang (R)$ and the theory $\Th (\Modr R)$ of right $R$-modules consists of their consequences.

Consider a subcategory $\C$ of $\Modr R$ and a module $M \in \C.$ A homomorphism $\phi : M \to C_M$ is a $\C$-{\em preenvelope} of $M,$ if for every homomorphism $f : M \to C$ with $C \in \C$ there exists a homomorphism $f_M : C_M \to C$ such that the diagram
$$\vcenter{
\xymatrix{
M \ar[r]^{\phi} \ar[dr]_f & C_M \ar[d]^{f_M} \\
& C
}}$$
commutes. A $\C$-{\em preenvelope} $\phi : M \to C_M$ is a $\C$-{\em envelope} if every endomorphism $f : C \to C$ such that $f \circ \phi = \phi$ is an automorphism of $C$.  We say that $\C$ is {\em preenveloping in} an additive category $\A \supseteq \C$ if every module $A \in \A$ has a $\C$-preenvelope. If $\A = \Modr R,$ we just say that $\C$ is {\em preenveloping;} if $\A = \modr R,$ we call $\C$ a {\em covariantly finite} subcategory of $\modr R.$

The superscript $\perp_i$ is used to denote orthogonality with respect to the bifunctor $\Ext^i_R (-,-),$ so if \lb $C \in \Modr R,$ then $C^{\perp_i} = \Ker \Ext^i (C,-)$ and similarly for ${^{\perp_i}}C;$ if $\C \subseteq \Modr R$ is a class, then 
$\C^{\perp_i} = \cap \{ C^{\perp_i} \; | \; C \in \C \}.$ The unadorned superscript $\perp$ refers to $\perp_1$ and $\perp_{\infty}$ will be used to denote the class orthogonal (on the appropriate side) with respect to all the bifunctors $\Ext^i,$ $i \geq 1.$ A $\C$-preenvelope $\phi: M \to M_{\C}$ is {\em special} if it is a monomorphism with $\Coker \phi \in {^\perp}\C$.

The notions of $\C$-{\em precover}, {\em precovering,} $\C$-{\em cover} and  {\em special} $\C$-{\em precover} are defined dually.

\section{Elementary torsion pairs}

Definable subcategories of $\Modr R$ were introduced by Crawley-Boevey~\cite{CB} to characterize in non-logical terms the elementary additive classes of right $R$-modules, which arise in the model theory of modules and are in bijective correspondence with the closed subsets of the Ziegler spectrum~\cite{Zg}.

\begin{defn} \label{D:definable}  
A full subcategory $\C \subseteq \ModR$ is {\em definable} if it is closed under direct products, pure submodules and direct limits.
\end{defn}

Definable subcategories are also closed under direct sums, which may be regarded as direct limits of finite direct products, or pure submodules of direct products. There are several characterizations of definable categories that are useful. Recall that an additive functor $F : \Modr R \to \Ab$ is {\em coherent} if it commutes with direct limits and direct products. A subcategory $\C \subseteq \Modr R$ is elementary if it is the class of models of some collection $\Sigma$ of sentences 
$\Lang (R),$ $\C = \Mod (\Sigma).$ Equivalently, the class $\C \subseteq \Modr R$ is {\em axiomatized} by $\Sigma.$

\begin{prop}  \label{P:definable} 
Let $\C$ be a full subcategory of $\ModR.$ The following are equivalent:
\begin{enumerate}
\item $\C$ is definable;
\item $\C$ is defined by the vanishing of some set of coherent functors;
\item $\C$ is closed under direct products, pure submodules and pure epimorphic images; and
\item $\C$ is an elementary class, closed under direct sums and direct summands.
\end{enumerate}
\end{prop}

\begin{proof}
The equivalences ($1$) $\Leftrightarrow$ ($2$) and ($1$) $\Leftrightarrow$ ($4$) are from~\cite[\S 2.1, 2.3]{CB}, where they are stated for an algebra over an infinite field; for the general case, one must include in Condition ($4$) that the elementary class is closed under direct sums. \\
($2$) $\Rightarrow$ ($3$). This is a consequence of the fact that coherent functors are exact on pure short exact sequences (see~\cite[\S 2.1, Lemma 1]{CB}).\\  
($3$) $\Rightarrow$ ($1$).  This holds because every direct limit is a pure epimorphic image of the direct sum of the modules in the directed system.
\end{proof}

An important consequence of Proposition~\ref{P:definable}(4) is that definable subcategories of $\Modr R$ is also closed under pure injective envelopes. This is because the pure injective envelope $M \to \PE (M)$ of a module $M$ is an elementary embedding~\cite[Theorem 4.3.21]{PSL}.

A torsion pair $(\T, \F)$ in $\Modr R$ is called {\em elementary} if both classes $\T$ and $\F$ are definable. All torsion and torsion free classes are additive, so this is equivalent to both of the classes $\T$ and $\F$ being elementary. The torsion free class $\F$ is already closed under products and submodules, so that it is definable if and only if it is closed under direct limits. On the other hand, the torsion class $\T$ is closed under coproducts and quotients, so that it is already closed under direct limits. It is therefore definable if and only if it is closed under pure submodules and products.

\begin{prop} \label{P:elementary torsion pair}
Let $(\T, \F)$ be a torsion pair in $\Modr R,$ with associated radical $t : \Modr R \to \Modr R.$ The torsion free class $\F$ is definable if and only if $t$ respects direct limits. If $\T$ is closed under direct products, then $t$ respects direct products.  
\end{prop} 

\begin{proof}
If $t$ respects direct limits, then it is clear that a direct limit of torsion free modules is torsion free. To prove the converse, suppose that a directed system $M_i,$ $i \in I,$ in $\Modr R$ is given, with direct limit ${\dis M = \lim_{\to} M_i.}$ Each $M_i$ is an extension of its torsion free quotient by its torsion submodule. These extensions themselves form a directed system of short exact sequences, with limit as shown in
$$\vcenter{
\xymatrix{
0 \ar[r] & t(M_i) \ar[r] \ar[d] & M_i \ar[r] \ar[d] & M_i/t(M_i) \ar[r] \ar[d] & 0 \\
0 \ar[r] & {\dis \lim_{\to} t(M_i)} \ar[r] & M \ar[r] & {\dis \lim_{\to} M_i/t(M_i)} \ar[r] & 0. 
}}$$
A torsion class is closed under direct limits so that ${\dis \lim_{\to} t(M_i)}$ is torsion. By assumption, the direct limit ${\dis \lim_{\to} M_i/t(M_i)}$ of torsion free modules is torsion free. The short exact sequence in the bottom row thus represents $M$ as the  extension of its torsion free quotient by it torsion submodule. This implies that the canonical morphism 
${\dis \lim_{\to} t(M_i) \to t(\lim_{\to} M_i)}$ is an isomorphism.

To prove the second statement, let $N_j,$ $j \in J,$ be a family of modules with product $N = \prod_j N_j,$ and use a similar argument, but taking a product instead of a direct limit, to obtain the short exact sequence
$$\vcenter{
\xymatrix{
0 \ar[r] & \prod_j t(N_j) \ar[r] & N \ar[r] & \prod_j N/t(N_j) \ar[r] & 0.
}}$$
A torsion free class is closed under products, so the right term $\prod_j N_j/t(N_j)$ is torsion free. By assumption, the product $\prod_j t(N_j)$ is torsion. Arguing as above shows that the canonical morphism \lb $t(\prod_j N_j) \to \prod_j t(N_j)$ is an isomorphism.
\end{proof}

Recall that an elementary class $\E \subseteq \Modr R$ is {\em finitely axiomatizable} (relative to the theory $\Th (\Modr R)$ of right $R$-modules) if there exists a sentence $\sigma$ in $\Lang (R)$ such that $\E = \Mod (\Th (\Modr R) \cup \{ \sigma \})$ is the subcategory of $\Modr R$ of modules $M$ that satisfy $\sigma,$ $M \models \sigma.$ 

\begin{thm} \label{T:elementary torsion pair}
A torsion pair $(\T, \F)$ in $\Modr R$ is elementary if and only if the associate torsion radical is a coherent functor. Equivalently, the torsion radical $t = \tau (-)$ is definable by a positive primitive formula $\tau (u)$ in $\Lang (R)$ and the torsion class $\T = \Mod (\forall u \; (\tau (u)))$ and the torsion free class $\F = \Mod [\forall u \; (\tau (u) \to (u \doteq 0))]$ are finitely axiomatizable in $\Lang (R).$ 
\end{thm}

\begin{proof}
If $(\T, \F)$ is elementary, then $t$ is coherent by Proposition~\ref{P:elementary torsion pair}. The torsion radical $t$ is coherent if and only if it is definable by a positive primitive formula, by~\cite[Lemma 2, \S 2.1]{CB}. The given axiomatizations of $\T$ and $\F$ follow from the respective definitions of a torsion and torsion free class of a torsion pair. Since both $\T$ and $\F$ are finitely axiomatizable, the torsion pair is therefore elementary. 
\end{proof}

Finitely axiomatizable definable subcategory correspond to basic Zariski open subsets of the Ziegler spectrum in the sense of~\cite[Chapter 14]{PSL}. Examples of elementary torsion pairs will be found with the help of pure projective modules. 

\begin{defn} \label{D:pure projective} 
An $R$-module is {\em pure projective} if it has the projective property with respect to pure exact sequences. Other equivalent formulations are the following:
\begin{enumerate}
\item A module is pure projective if and only if it is a direct summand of a direct sum of finitely presented modules. For this reason, the subcategory of pure projective modules is denoted by $\Add (\modr R).$
\item A module $M$ is pure projective if and only if every pure short exact sequence $0 \to A \to B \to M \to 0$ splits.
\end{enumerate}
\end{defn}

\begin{cor} \label{T and T perp is pure projective}
If $(\T, \F)$ is an elementary torsion pair in $\Modr R,$ then $\T \cap {^{\perp}}\T \subseteq \Add (\modr R).$
\end{cor}

\begin{proof}
Consider a pure short exact sequence $0 \to A \to B \stackrel{g}{\rightarrow} P \to 0,$ with $P$ in $\T \cap {^{\perp}}\T,$ and apply the torsion radical,
$$\vcenter{
\xymatrix{
0 \ar[r] & t(A) \ar[r] & t(B) \ar[r]^{t(g)} & P \ar[r] & 0.
}}$$
The sequence is exact, because $t$ is a coherent functor, and the third term is $t(P) = P.$ But $t(A) \in \T$ and
$P \in {^{\perp}}\T,$ so the sequence splits. If $r : P \to t(B)$ is a retraction of $t(g),$ then, because $t$ is a subfunctor of the identity functor on $\Modr R,$ the composition $r : P \to t(B) \subseteq B$ is a retraction of $g : B \to P,$ as required.
\end{proof}

A definable subcategory $\D \subseteq \Modr R$ {\em has enough pure projectives}~\cite{HR} if for every module $D_R \in \D,$ there exists a pure epimorphism $g : P \to D,$ where $P$ is a pure projective module in $\D.$ This is equivalent~\cite[Theorem 8]{HR} to the condition that every finitely presented module $A \in \modr R$ have a pure projective $\D$-preenvelope.  

\begin{lem} \label{L:first} 
If $0 \to H \to M \to B \to 0$ is an exact sequence with $H$ finitely generated and $M$ pure projective, then $B$ is pure projective.
\end{lem}

\begin{thm} \label{T:enough pure projectives}
The following are equivalent for a torsion pair $(\T, \F)$ in $\Modr R$ with $\T$ definable:
\begin{enumerate}
\item the module $R_R$ has a pure projective $\T$-preenvelope;
\item $\T$ has enough pure projectives;
\item $\T = \Gen P$ for some pure projective module.
\end{enumerate}
Moreover, if these conditions are satisfied, then $(\T, \F)$ is an elementary torsion pair.
\end{thm}

\begin{proof}
($1$) $\Rightarrow$ ($2$). Let $\varepsilon : R \to T_R$ be a pure projective $\T$-preenvelope. We will use the lemma to build a pure projective $\T$-preenvelope of a finitely presented module $A.$ There is a short exact sequence \lb 
$0 \to H \to R^n \overset{\pi} \to A \to 0$ with $H$ a finitely generated module. Take the pushout of $\pi$ and the pure projective $\T$-preenvelope $\varepsilon^n : R^n \to T_R^n$ to obtain the commutative diagram
$$\vcenter{
\xymatrix{
0 \ar[r] & H \ar@{=}[d] \ar[r] & R^n \ar[r]^{\pi} \ar[d]^{\varepsilon} & A\ar[d]^{\delta} \ar[r] & 0 \\
0 \ar[r] & H \ar[r] & T_R^n  \ar[r] & T_A \ar[r] & 0.
}}$$
The morphism $\delta : A \to T_A$ is then a $\T$-preenvelope of $A,$ with $T_A$ pure projective, by Lemma~\ref{L:first}.\\
($2$) $\Rightarrow$ ($3$). In general, every definable subcategory $\D$ with enough pure projectives is of the form $\D = \Gen P$ for some pure projective module~\cite[Lemma 3]{HR}. One may take $P$ to be the coproduct of pure projective $\D$-preenvelopes of the finitely presented modules. \\
($3$) $\Rightarrow$ ($1$). Rada and Saorin~\cite{RS} proved that every subcategory closed under direct products and pure submodules is preenveloping. Every definable subcategory is therefore preenveloping. Let $\varepsilon : R \to T_R$ be a $\T$-preenvelope of $R.$ By hypothesis, there is an epimorphism $\phi : P^{(I)} \to T_R$ from a coproduct of copies of $P.$ The $\T$-preenvelope $\varepsilon$ then factors through $\phi,$ and yields the module $P^{(I)}$ is a pure projective $\T$-preenvelope of $R_R.$ 

To prove the last statement, use Condition ($3$) to see that $\F = \{ M \in \Modr R \; | \; \Hom_R (P,M) = 0   \}.$ As $P$ is pure projective, $\F$ is closed under pure epimorphic images. By Proposition~\ref{P:definable}(3), it is definable. 
\end{proof}

\begin{conj} \label{C:enough-pps}
If $(\T, \F)$ is an elementary torsion pair, then $\T$ has enough pure projective modules.
\end{conj}

If $R$ is right noetherian, the conjecture can be verified by expressing a torsion module ${\dis T = \lim_{\to} M_i}$ as the direct limit of its finitely generated submodules. Then ${\dis T = \lim_{\to} t(M_i)}$ is the direct limit of its finitely  generated torsion submodules. These are all finitely presented and $T$ is a pure epimorphic image of the pure projective module obtained by taking the direct sum of the $t(M_i).$ 

\section{Pure projective tilting modules} \label{S:pptm}

\begin{defn}\label{D:tilting} A right $R$-module $T$ is a (large) $n${\em -tilting} module if it satisfies the following conditions:
\begin{enumerate}
\item[(T1)] $\pd (T) \leq n;$
\item[(T2)] $\Ext_R^i (T, T^{(\lambda)}) = 0,$ $i > 0,$ for every cardinal $\lambda;$
\item[(T3)] there exists an exact sequence:
$$\vcenter{
\xymatrix{
0 \ar[r] & R \ar[r] & T_0 \ar[r] & {\cdots} \ar[r] & T_{1} \ar[r] & 0 
}}$$
where  $T_i \in \Add T$ for $i = 0, \ldots, n.$ The notion of an $n${\em -cotilting} module is defined dually.
\end{enumerate}
\end{defn}

By \cite{CT} a right $R$-module $T$ is $1$-tilting if and only if $T^{\perp} = \Gen T,$ which is called the tilting class of
$T.$ By~\cite[Lemma 14.2]{GT12}, the tilting class $\T = T^{\perp}$ is a torsion class in $\Modr R$ and gives rise to a torsion pair $(\T, \F)$ where $$\F = \{M \in  \ModR \; | \; \Hom_R(T, M) = 0 \}.$$ 
Two $1$-tilting modules $T$ and $U$ are said to be {\em tilting equivalent} if $T^\perp = U^\perp$ or, equivalently, if 
$\Add T = \Add U.$ A $1$-tilting module equivalent to a finitely presented $1$-tilting module is called {\em classical.} For later reference, we recall a useful criterion for a module $U$ to be a $1$-tilting module equivalent to a given one $T.$

\begin{lem} \label{L:tilting-crit}
Let $T$ be a $1$-tilting module and suppose there is a short exact sequence 
$$\vcenter{
\xymatrix{
0 \ar[r] & R \ar[r] & U_0 \ar[r] & U_1 \ar[r] & 0,
}}$$
where $U_0,$ $U_1 \in \Add T.$ Then $U = U_0 \oplus U_1$ is a $1$-tilting module equivalent to $T.$
\end{lem} 

\begin{proof} (cf.\ the proof of~\cite[Theorem 13.18]{GT12}). Every module in $\Add T$ satisfies Conditions (T1) and (T2) for a $1$-tilting module. This is because the class of modules that satisfy Condition (T1) is closed under direct sums and summands. Similarly, the class of modules that satisfy Condition (T2) is clearly closed under direct summands, and it is readily verified that for every index set $I,$ the direct sum $T^{(I)}$ of copies of $T$ also satisfies Condition (T2). Because $U$ belongs to $\Add T$ it satisfies the Conditions (T1) and (T2). It satisfies Condition (T3) by hypothesis, so that $U$ is itself a $1$-tilting module. Therefore $\Gen U = U^{\perp}.$ Now $U \in \Add T$ implies that the tilting class of $U,$ $\Gen U \subseteq \Gen T$ is contained in that of $T.$ On the other hand, $U$ is a direct summand of some direct sum $T^{(I)},$ so that $U^{\perp} \supseteq (T^{(I)})^{\perp} = T^{\perp},$ as required.
\end{proof}
 
If $I \subseteq R$ is a two sided ideal, then we may think of the module category $\Modr R/I$ as a full subcategory of $\Modr R$ induced by restriction of scalars along the quotient map $R \to R/I$ of rings. It consists of the modules $M \in \Modr R$ for which $MI = 0 .$ Given a torsion pair $(\T, \F)$ in $\Modr R,$ with torsion ideal $I = t(R) \subseteq R,$ we have that for every $F \in \F,$  $\Hom_R(I, F) = 0$, hence $\Hom_R(R/I, F) \cong F,$ that is, $FI = 0.$ We may therefore think of $\F$ as a subcategory of $\Modr R/I.$ 

Parra and Saor\'{i}n~\cite{PS, PS16} showed that the heart $\clH_t$ of the $t$-structure in $\D (\Modr R)$ induced by a torsion pair $(\T, \F)$ is Grothendieck if and only if $\F$ is closed under direct limits in $\Modr R,$ or, equivalently, a definable subcategory. 

\begin{prop} \label{F is definable}
Let $(\T, \F)$ be a torsion pair in $\Modr R,$ with torsion radical $t$ and torsion ideal $I = t(R).$ The torsion free class
$\F$ is closed under direct limits in $\Modr R$ if and only if there exists a $1$-cotilting $R/I$-module $C,$ such that 
$$\F = \{ M \in \Modr R \; | \; MI = 0 \; \mbox{and} \; \Ext^1_{R/I}(M,C) = 0 \}.$$
\end{prop}

\begin{proof}
As a consequence of the result that every $1$-cotilting module $C_R$ is pure injective~\cite{B03}, the cotilting class 
${^\perp}C \subseteq \Modr R/I$ is a definable subcategory. It follows that $\F \subseteq \Modr R$ is definable.

For the converse, note that submodules, direct products and direct limits of modules in $\F$ are $R/I$-modules. Thus, we may consider $\F$ as a definable subcategory of $\Modr R/I$, which is therefore closed under pure epimorphic images in $\Modr R/I.$ Moreover, $\F$ contains all the projective $R/I$-modules, thus by \cite[Proposition 5.2.2]{EJ} or \cite[Theorem 2.5]{HoJo}, $\F$ is a covering class in $\Modr R/I.$ Every cover is an epimorphism, since $\F$ contains the projective $R/I$-modules. As $\F$ is closed under extensions, Wakamatsu's Lemma implies that $\F$ is a special precovering class in $\Modr R/I$. By \cite[Theorem 15.22]{GT12}, there is a $1$-cotilting $R/I$-module $C$ such that $\F = \Ker \Ext^1_{R/I}(-, C).$
\end{proof}

If the torsion pair $(\T, \F)$ is generated by a $1$-tilting module $T,$ then the tilting class $\T = \Gen T$ is definable. This was proved by the first author and Herbera~\cite{BH}, who showed that there exists a collection $\A \subseteq \modr R$ of finitely presented modules of projective dimension at most $1,$ such that $\T = \A^{\perp}.$ If $A_R$ is a finitely presented module of projective dimension at most $1,$ then the functor $\Ext^1_R (A,-)$ is coherent, so that $\T$ is definable by Condition (2) of Proposition~\ref{P:definable}.

\begin{thm}  \label{T:pure-proj}
The following are equivalent for a $1$-tilting module $T,$ with tilting class $\T = \Gen T:$
\begin{enumerate}
\item the module $T$ is a pure projective $1$-tilting module;
\item the definable subcategory $\T$ has enough pure projective modules;
\item the torsion pair $(\T, \F)$ is elementary;
\item $\T \cap {^{\perp}}\T \subseteq \Add (\modr R);$
\item the module $R_R$ admits a pure projective special $\T$-preenvelope;
\item every finitely presented module admits a special pure projective $\T$-preenvelope; 
\item every finitely presented module in ${^{\perp}}\T$ admits a special pure projective $\T$-preenvelope in $\Add T;$ and
\item $T$ is tilting equivalent to a countably presented pure projective $1$-tilting module.
\end{enumerate}
If $R$ is a right noetherian ring, these conditions are equivalent to the condition that $\X = \T \cap \modr R$ is a covariantly finite subcategory of $\modr R$ and ${\dis \T = \varinjlim (\X).}$
\end{thm}

\begin{proof}
($1$) $\Rightarrow$ ($2$) $\Rightarrow$ ($3$) $\Rightarrow$ ($4$) $\Rightarrow$ ($1$). The first two implications follow from Theorem~\ref{T:enough pure projectives}; the third from Theorem~\ref{T:elementary torsion pair} and Corollary~\ref{T and T perp is pure projective}; and the fourth from the general fact about a $1$-tilting module $T$ that $\T \cap {^{\perp}}\T = \Add T.$\\
($1$) $\Rightarrow$ ($5$). By Condition (T3) for a $1$-tilting module, there is an exact sequence
$0 \to R \to T_0 \to T_1 \to 0$ where both $T_0$ and $T_1$ are in $\Add T = \T \cap {^{\perp}\T}$. Thus $0 \to R \to T_0$ is a special $\T$-preenvelope of $R$ with $T_0$ pure projective.\\
($5$) $\Rightarrow$ ($6$). Let $A$ be a finitely presented module. Consider a short exact sequence 
$0 \to H \to R^n \overset{\pi} \to A \to 0$ with $H$ a finitely generated module, and argue as in the proof of 
($1$) $\Rightarrow$ ($2$) of Theorem~\ref{T:enough pure projectives}, by taking the pushout of $\pi$ and $\varepsilon^n,$ where
$\varepsilon : R \to T_0$ is a special pure projective $\T$-preenvelope. By the properties of a pushout, the morphisms
$\varepsilon^n : R^n \to T_0^n$ and $\delta : A \to T_A$ are also special pure projective $\T$-preenvelopes. \\ 
($6$) $\Rightarrow$ ($7$). If $A$ is a finitely presented module in ${^{\perp}}\T,$ then the special $\T$-preenvelope given by the hypothesis lies in $\T \cap {^{\perp}}\T = \Add T.$ \\ 
($7$) $\Rightarrow$ ($2$). Apply the hypothesis to the module $R_R$ and apply the implication ($1$) $\Rightarrow$ ($2$) of Theorem~\ref{T:enough pure projectives}. \\  
($1$) $\Leftrightarrow$ ($8$). By assumption, every module $T_0$ in $\Add T$ is pure projective, a direct summand of a direct sum $\oplus_{i\in I} E_{i}$ with $E_i$ finitely presented modules, hence in particular countably generated. By Kaplansky's Theorem~\cite[Theorem 1]{Ka}, $T_0$ is a direct sum of countably generated submodules $X_{\alpha}$ in $\Add T,$ 
$T_0 = \oplus_{\alpha\in \Lambda} X_{\alpha}.$ By Condition (T3) for a $1$-tilting module, there is a short exact sequence
$0 \to R \overset{\varepsilon}\to T_0 \to T_1 \to 0$ with $T_0$ as described. The image of $\varepsilon$ is contained in a summand $U_0$ of $T,$ $U_0 = \oplus_{\beta\in F_0} X_{\beta}$ where $F_0$ is a finite subset of $\Lambda.$ Hence, $U_0$ is countably generated and ${\dis T_1 \cong U_0/\varepsilon(R) \oplus \oplus_{\alpha\in \Lambda \setminus F_0} X_{\alpha}.}$ 
The quotient module $U_1 = U_0/\varepsilon(R)$ is a summand of $T_1$ and so also belongs to $\Add T.$ It follows from Lemma~\ref{L:tilting-crit} that $U = U_0 \oplus U_1$ is a countably generated $1$-tilting module in $\Add T$ that is tilting equivalent to $T.$ It is well known that, for every cardinal $\kappa,$ a $\kappa$-generated $1$-tilting module is $\kappa$-presented.

If the ring $R$ is right noetherian, then condition in the last statement follows from the remark in the last paragraph of Section 1, together with~\cite[Lemma 8.35]{GT}; the converse from~\cite[Lemma 8.36]{GT}.
\end{proof}

The best understood derived equivalence induced by a torsion pair $(\T, \F)$ is one where the torsion class $\T = \Gen T$ is a tilting class, with tilting module $T.$ In that case, $(\T, \F)$ is an elementary torsion pair. The corresponding torsion pair 
$(\F[-1], \T[0])$ in the heart $\clH_t$ of the associated $t$-structure has the property that $T = T[0]$ is a projective generator of $\clH_t.$ The condition that $\X = \T \cap \modr R$ be covariantly finite in $\modr R$ with ${\dis \T = \varinjlim (\X)}$ is equivalent to the existence in $\clH_t$ of a small generating class of finitely generated projective objects.

\begin{cor} \label{C:pure-proj}
Let $T$ be a $1$-tilting module. The heart $\clH_t$ of the $t$-structure in the derived category $\D (\Modr R)$ induced by the torsion pair $(\Gen T, \F)$ is a Grothendieck category if and only if $T$ is a pure projective module.
\end{cor}

If $(\T, \F)$ is the torsion pair in $\modr R$ with $\T$ a cogenerating class, then, by~\cite[Prop 5.3]{PS}, the heart 
$\clH_t \subseteq \D (R)$ of the associated $t$-structure is a module category if and only if $\T$ is a tilting class induced by a {\em self-small} $1$-tilting module $T.$ By~\cite[Prop 1.3]{CT} a $1$-tilting module is self-small if and only if it is finitely presented. The following then is the ''classical'' version of Theorem~\ref{T:pure-proj}.

\begin{thm} \label{T:classical} 
Let $T$ be $1$-tilting module with tilting class $\T = \Gen T$. The following are equivalent:
\begin{enumerate}
\item $T$ is classical;
\item $\T \cap {^{\perp}}\T \subseteq \Add (\T \cap {^{\perp}}\T \cap \modr R);$
\item the module $R_R$ admits a finitely presented special $\T$-preenvelope; 
\item every finitely presented module admits a finitely presented special $\T$-preenvelope; and 
\item every finitely presented module in ${^{\perp}}\T$ admits a finitely presented special $\T$-preenvelope in $\Add T.$
\end{enumerate}
In case these conditions hold, the subcategory $\X = \T \cap \modr R$ is covariantly finite in $\modr R$ and 
${\dis \T = \lim_{\to} (\X).}$
\end{thm}

\begin{proof} ($1$) $\Rightarrow$ ($2$). If $T$ is equivalent to a finitely presented module $E,$ then 
$$\T \cap {^{\perp}}\T = \Add T = \Add E \subseteq \Add (\T \cap {^{\perp}}\T \cap \modr R).$$
($2$) $\Rightarrow$ ($3$). Condition (T3) implies that there is a short exact sequence 
$0 \to R \to T_0 \to T_1 \to 0$ with $T_0$ and $T_1$ in $\Add T.$ By hypothesis, $T_0$ is a direct summand of a direct sum
$T_0 \oplus T'_0 = \oplus_{i\in I} E_{i}$ with $E_i$ finitely presented modules in $\Add T.$ Adding to both of the modules $T_0$ and $T_1$ the direct summand $T'_0,$ we may assume that there exists a sequence $0 \to R \to V_0 \to V_1 \to 0$ with $V_0$ and $V_1$ in $\Add T$ and $V_0$ is a direct sum $\oplus_{i\in I} E_{i}$ of finitely presented modules $E_i.$ Arguing as in the proof of ($1$) $\Rightarrow$ ($8$) of Theorem~\ref{T:pure-proj} yields a short exact sequence $0 \to R \to U_0 \to U_1 \to 0$ where $U_0$ and $U_1$ are finitely presented modules in $\Add T = \T \cap {^{\perp}}\T.$ Thus $0 \to R \to U_0$ is a special $\T$-preenvelope of $R$ with $U_0$ finitely presented. \\
($3$) $\Rightarrow$ ($4$). Let $A$ be a finitely presented module and argue as in the proof of ($1$) $\Rightarrow$ ($2$) of Theorem~\ref{T:pure-proj}, with $\varepsilon : R \to T_0$ a special finitely presented $\T$-preenvelope. Then $\varepsilon^n$ is also such and $\delta : A \to T_A$ is the required special finitely presented $\T$-preenvelope of $A.$ \\
($4$) $\Rightarrow$ ($5$). Clear, since a special $\T$-preenvelope of a module in ${^{\perp}}\T$ belongs to $\Add T.$\\
($5$) $\Rightarrow$ ($1$). Let $0 \to R \to U_0 \to U_1 \to 0$ be a special $\T$-preenvelope of $R$ with $U_0$ finitely presented. Then both $U_0$ and $U_1$ belong to $\T \cap {^{\perp}}\T = \Add T,$ so that $U = U_0 \oplus U_1$ is, by
Lemma~\ref{L:tilting-crit}, a $1$-tilting module equivalent to $T.$ 

For the last statement, note that Condition ($4$) implies that $\T \cap \modr R$ is covariantly finite and use Lenzing's special case~\cite{Lenz83} of~\cite[Theorem 8]{HR}.
\end{proof}

\begin{cor} \label{C:classical}
Let $T$ be a $1$-tilting module. The heart $\clH_t$ of the $t$-structure in the derived category $\D (\Modr R)$ induced by the torsion pair $(\Gen T, \F)$ is equivalent to a module category if and only if $T$ is classical.
\end{cor}

\begin{cor} \label{C:Krull-Schmidt}
If $R$ is a ring over which every right pure projective module is a direct sum of finitely presented modules, then every pure projective $1$-tilting module is classical.
\end{cor}

\begin{proof}
If $T$ is a pure projective $1$-tilting module over such a ring, with tilting class $\T,$ then $\T \cap {^{\perp}}\T = \Add T$ consists of pure projective modules, so that the hypothesis implies Condition ($2$) of Theorem~\ref{T:classical}. 
\end{proof}

If $A \in \modr R$ is a finitely presented module of projective dimension at most $1,$ then the functor $\Ext^1_R (A,-) : \modr R \to \Ab$ is coherent so that the condition for a module $M_R$ to satisfy $\Ext^1 (A,M) = 0$ is elementary~\cite[]{PSL}: there is a sentence $\sigma_A$ in $\Lang (R)$ such that $\Ext^1_R (A,M) = 0$ if and only if $M \models \sigma_A.$ 

\begin{prop} \label{fin axiom}
Let $T$ be a $1$-tilting module. The tilting class $\T = \Gen T = T^{\perp}$ is finitely axiomatizable if and only if there exists a finitely presented module $A_R$ of projective dimension at most $1$ such that $\T = A^{\perp}.$ 
\end{prop}

\begin{proof}
If $\T = A^{\perp},$ then $\T$ is axiomatized by the sentence $\sigma_A.$ For the converse, suppose that $\T$ is finitely axiomatized by the sentence $\sigma,$ and let $\A \subseteq \modr R$ be an additive subcategory of modules of projective dimension at most $1$ for which $\T = \A^{\perp}.$ The collection 
$\Th (\Modr R) \cup \{ \neg \sigma \} \cup \{ \sigma_A \; | \; A \in \A \}$ of sentences in $\Lang (R)$ is inconsistent. By the Compactness Theorem, some finite subcollection is already inconsistent, which implies that there are finitely many 
$A_1, A_2, \ldots, A_n \in \A$ such that $\Th (\Modr R) \vdash \bigwedge_{i=1}^n \sigma_{A_i} \to \sigma.$ In words, if $A = \oplus_{i=1}^n A_i,$ then $\Ext^1_R (A,M) = 0$ implies $M \in \T.$
\end{proof}

According to Theorem~\ref{T:pure-proj}($8$), a pure projective $1$-tilting module is equivalent to a countably presented pure projective $1$-tilting module $T.$ For many purposes, we may thus replace the given $1$-tilting module by $T$ and extract some further information when $T$ is pure projective. Arguing as in~\cite[Lemma 4.4]{Tr07}, every countably presented $1$-tilting module $T,$ with tilting class $\T$ may be represented as the limit of a linear system ${\dis T = \lim_{n \to \infty} A_n,}$ where for each $n:$
\begin{enumerate}
\item $A_n \in \modr R;$ 
\item $\pd (A_n) \leq 1;$ and
\item $A_n \in {^{\perp}}\T.$ 
\end{enumerate}
Such a linear system $(A_n, f_n : A_n \to A_{n+1})_{n \in \mathbb{N}}$ is called a {\em system associated} to $T;$ it yields a pure exact sequence
$$\vcenter{
\xymatrix{
0 \ar[r] & {\dis \oplus_{n \in \mathbb{N}}} A_n \ar[r]^{\phi} & {\dis \oplus_{n \in \mathbb{N}}} A_n \ar[r] & T \ar[r] & 0
}}$$ 
in the usual way, where $\phi (a_1, a_2, \cdots, a_n, \cdots) = (a_1, a_2 - f_1 (a_1), \cdots, a_n - f_{n-1}(a_{n-1}), \cdots).$ 

\begin{prop} \label{P:system-An}
If $T$ is a countably presented pure projective $1$-tilting module with associated system $(A_n | n \in \mathbb{N}),$ then 
\begin{enumerate}
\item $T \oplus \left( {\dis \oplus_{n \in \mathbb{N}} A_n} \right) \cong {\dis \oplus_{n \in \mathbb{N}}} A_n;$
\item there exists a $k \in \mathbb{N}$ such that $B_k = \oplus_{n \leq k} A_n$ satisfies $B_k^{\perp} = \T;$ and
\item ${\dis \lim_{n \to \infty} \; A_n/t(A_n) = 0.}$
\end{enumerate}
\end{prop}

\begin{proof}
Because $T$ is pure projective, the pure exact sequence above splits and yields (1). This implies that $\bigcap_n A_n^{\perp} \subseteq T^{\perp} \subseteq A_k^{\perp}$ for every $k \in \mathbb{N},$ since $A_k \in {^{\perp}}\T.$ The first inclusion is thus an equality. As $T$ is pure projective, $\T = T^{\perp}$ is finitely axiomatizable, so that we can argue as in the proof of Proposition~\ref{fin axiom} to get (2). Finally, the pure projective assumption on $T$ also entails, by Theorem~\ref{T:elementary torsion pair}, that the torsion radical $t : \modr R \to \modr R$ is coherent, and so respects direct limits. This implies that ${\dis \lim_{n \to \infty} t(A_n) = t (\lim_{n \to \infty} A_n) = t(T) = T.}$ Consider the linear system of short exact sequences
$$\vcenter{
\xymatrix{
0 \ar[r] &  t(A_n) \ar[r] & A_n \ar[r] & A_n/t(A_n) \ar[r] & 0
}}
$$
associated to the system $(A_n | n \in \mathbb{N})$ and take the limit to obtain the short exact sequence 
$$\vcenter{
\xymatrix{
0 \ar[r] &  T \ar[r]^{1_T} & T \ar[r] & {\dis \lim_{n \to \infty} A_n/t(A_n)} \ar[r] & 0,
}}
$$
which establishes (3). 
\end{proof}

\section{The Commutative case}\label{s:commutative} 
In this section we prove that a pure projective $1$-tilting module over a commutative ring is projective. The result involves the notion of the henselization of local rings, but for some classes of commutative rings, like noetherian or arithmetic rings the arguments are simpler. First, recall that a $1$-tilting module $T$ over a commutative ring is equivalent to a classical tilting module if and only if $T$ is projective (see for instance \cite[Lemma 1.2]{PT}).
\begin{lem}\label{L:local-reduction} Let $T$ be a $1$-tilting module over a commutative ring $R$ and let $S$ be a multiplicative subset of $R$. Then $T_S$ is a $1$-tilting $R_S$-module and if $T$ is pure projective, so is $T_{S}$.
\end{lem}
\begin{proof}
By \cite[Proposition 13.50]{GT12}, $T_{S}$ is a $1$-tilting $R_S$-module. (Note that the proof given there becomes simpler in the case of $1$-tilting modules, since the first syzygy of $T$ is projective). If $T$ is pure projective, it follows immediately that $T_S$ is pure projective, since $R_S$ is a flat $R$-module.
 \end{proof}
 The next proposition allows to reduce the investigation to the case of local commutative rings.
  \begin{prop} \label{P:localizing} Let $T$ be a $1$-tilting module over a commutative ring $R.$ Then $T$ is projective if and only if 
$T_{\m}$ is a projective $1$-tilting $R_{\m}$-module, for every $\m\in \Max R$.
    \end{prop}
\begin{proof} By Lemma~\ref{L:local-reduction}, $T_{\m}$ is a $1$-tilting $R_{\m}$-module, for every $\m\in \Max R$. If $T$ is projective, then clearly $T_{\m}$ is a projective $R_{\m}$-module.
  
For the converse, let us apply again the result~\cite{BH} that every $1$-tilting module $T$ is of finite type, that is, there exists a set $\{ A_i; i \in I \}$ of finitely presented modules with $\pd (A_i) \leq 1$ such that $M \in T^\perp$ if and only if $\Ext^1_R(A_i, M)=0$, for every $i\in I$. For every maximal ideal $\m$ of $R$ we have $\Ext^1_{R_{\m}}((A_i)_{\m}, T_{\m})=0$. This implies that $(A_i)_{\m}\in {^\perp}(T_{\m}^\perp)$.  By assumption, $T_{\m}$ is a projective $R_{\m}$-module. Thus $(A_i)_{\m}$ is projective, too. We conclude that $A_i$ is a projective $R$-module, for every $i\in I$. Hence $T$ is projective.
\end{proof}
A module $M$ over a ring $R$ is FP$_2$ if $M$ is finitely presented and a first syzygy of $M$ is finitely presented.

 \begin{lem}\label{L:Ext-Tor} Let $(R, \m)$ be  a commutative local ring and let $A$ be an FP$_2$-module $R$-module. The following are equivalent:
 \begin{enumerate}
 \item $A$ is projective;
 \item $\Tor_1^R(A, R/\m)=0;$ 
 \item $\Ext^1_R(A, R/\m)=0$.
 \end{enumerate}
 \end{lem}
 \begin{proof} (1) $\Leftrightarrow$ (2). This is well known (see for instance~\cite[Lemma 2.5.8]{Glaz89}). \\
(2) $\Leftrightarrow$ (3). For every $R$-module $M$ let $M\dual=\Hom_R(M, E)$ where $E$ is an injective envelope of $R/\m$. Then $(R/\m)\dual\cong R/\m$ and by well known homological formulas, $\Ext^1_R(A, R/\m)=0$ if an only if $\Tor_1^R(A, R/\m)=0$.
 \end{proof} 

 \begin{prop} \label{P:local-case} 
Let $(R, \m)$ be a local commutative ring. A $1$-tilting module $T$ is projective if and only if $R/\m \in T^\perp.$ So if $T$ is not projective, then no nonzero finitely generated module is torsion. 
 \end{prop}
 \begin{proof} Necessity is clear. As in the proof of Proposition~\ref{P:localizing}, let $\{A_i, i\in I\}$ be a set of finitely presented modules of projective dimension at most one such that $T^\perp=(\bigoplus\limits_{i\in I}A_i)^\perp$. By assumption $\Ext^1_R(A_i, R/\m)=0$ for every $i\in I$. By Lemma~\ref{L:Ext-Tor} we conclude that every $A_i$ is projective, hence $T^\perp =\ModR$ and $T$ is projective.

If $T$ is not projective and $M$ is a nonzero finitely generated torsion module $M \in \T,$ Nakayama's Lemma implies that the nonzero quotient $M/\m M$ and therefore $R/\m$ belongs to $\T.$
\end{proof}

We can use Proposition~\ref{P:local-case} to see that every pure projective $1$-tilting module $T$ over a commutative noetherian ring is projective. By Proposition~\ref{P:localizing} and Theorem~\ref{T:pure-proj}, it suffices to verify the case when $R$ is a local and $T$ countably presented. If $(A_n | n \in \mathbb{N})$ is a system associated to $T,$ then 
Proposition~\ref{P:system-An}(1) implies that 
$T \oplus \left( {\dis \oplus_{n \in \mathbb{N}} A_n} \right) = {\dis \oplus_{n \in \mathbb{N}}} A_n.$
Because $t(T) = T,$ $t(A_n) \subseteq A_n$ must be nonzero for some $n.$ As $R$ is noetherian, this implies that $t(A_n)$ is a finitely generated torsion module. By Proposition~\ref{P:local-case}, $T$ must be projective. 

\begin{prop} \label{P:K-S}
If $R$ is a commutative local Krull-Schmidt ring, then every pure projective $1$-tilting module over $R$ is projective.
\end{prop} 

\begin{proof}
Over a Krull-Schmidt ring, every finitely presented module is a direct sum of indecomposable modules with a local endomorphism ring, so that every pure projective module is a direct sum of finitely presented modules, and therefore cannot be torsion. 
\end{proof}

Recall that a commutative ring $R$ is a {\em chain ring} if the lattice of its ideals is linearly ordered and $R$ is {\sl arithmetic} if the lattice of its ideals is distributive. By results of Jensen, a ring $R$ is arithmetic if and only every localization of $R$ at a maximal ideal is a chain ring.  It is well known that if $R$ is a chain ring, then every finitely presented module is a direct sum of cyclically presented modules (see e.g.~\cite[Theorem 9.1]{K2}), that is, modules of the form $R/rR,$ for some $r\in R.$ The endomorphism rings of such modules are clearly local, so that every chain ring is Krull-Schmidt. Propositions~\ref{P:localizing} and~\ref{P:K-S} imply that every pure projective $1$-tilting module over an arithmetic ring is projective. 

Recall that a local commutative ring $(R,\m)$ with residue field $k$ is {\em henselian} if for every monic polynomial 
$f \in R[X]$ and every factorization $\overline{f}= g_0h_0$ in $k[X]$ with $g_0$ and $h_0$ comaximal, there is a factorisation
$f=gh$ in $R[X]$ such that $\overline{g}=g_0$ and $\overline{h}=h_0.$ Examples of henselian rings include $0$-dimensional local rings and local rings $(R, \m)$ which are complete in the $\m$-adic topology. An important result about henselian rings is that they are Krull-Schmidt (see~\cite{Sid} or~\cite[V Section7]{FS}). 
 
In order to prove the main result of this section, we need to recall that every commutative local ring admits a {\em henselization} (see~\cite[Chap. VII]{Na} or \cite[18.6]{Groth}).

\begin{prop} \label{P:henselian}
Let $(R, \m)$ be a local commutative ring. There is a local ring $R^H$ and a local ring homomorphism 
$h: R\to R^H$, such that:
\begin{enumerate} 
\item $R^H$ is henselian;
\item $R \to R^H$ is faithfully flat;
\item $\m R^H$ is the maximal ideal of $R^H;$
\item for every ring homomorphism $f: R\to R'$ with $R'$ henselian, there is a unique ring homomorphism $g: R^H \to R'$ such that $f = g \circ h.$
\end{enumerate}
\end{prop}
 \begin{thm}\label{T:commutative-case} Let $R$ be a commutative ring. Then every pure projective $1$-tilting module is projective, and therefore classical.
 \end{thm}
 \begin{proof}  By Proposition~\ref{P:localizing} we can assume that $R$ is local. Let $T$ be a $1$-tilting $R$-module and consider a henselization $R^H$ of $R.$ Since $R^H$ is flat, we can argue as in the proof of~\cite[Proposition 13.50]{GT12}, to conclude that the $R^H$-module $T^H = T \otimes_R R^H$ is $1$-tilting. Moreover, again by flatness, if $T$ is pure projective so is $T^H$. By Proposition~\ref{P:henselian}, $T^H$ is a projective $R^H$-module and since projectivity descends along faithfully flat ring homomorphisms (see~\cite[Part II]{RG}, corrected in~\cite{Gru}, or by [10.92.1] in the Stacks project), we conclude that $T$ is a projective $R$-module.
  \end{proof}

\section{The Noetherian case} \label{s:noether}

In this section we show how to construct examples of nonclassical pure projective 1-tilting modules over noetherian rings. The construction is based on the following result of J.\ Whitehead \cite{Wh}.

 \begin{thm}\label{T:Whitehead} Let $I$ be an idempotent ideal of $R$ finitely generated on the left. Then there exists a countably generated projective module $P \in \ModR$ with trace ideal $I$. 
 \end{thm} 
 
 Let us briefly explain how to construct $P$.
Suppose that $I = R i_1+\cdots+R i_k$ and let $c = (i_1,\dots,i_k)^T$ be a column containing the generators of $I$. 
For every $n \in \bbN$ let $F_n = R^{k^{n-1}}$ and
let $f_n  \colon F_{n} \to F_{n+1}$ be the homomorphism given by the block-diagonal matrix having every diagonal 
block equal to $c$. For example, if $k = 2$ we have 
$$
f_1 = \left( \begin{array}{c} i_1 \\ i_2 \end{array}\right) \times -, f_2 = \left( \begin{array}{cc} i_1 & 0 \\ i_2 & 0 \\
0 & i_1 \\ 0 & i_2 \end{array}\right)\times -
$$
Then it is possible to show the existence of homomorphisms $g_n \colon F_{n+1} \to F_n$ such that $f_{n} = g_{n+1}f_{n+1}f_{n}$
for every $n \in \bbN$. Let $P = \li F_i$, then the canonical presentation of the direct limit in the short exact sequence
$$\vcenter{
\xymatrix{
0 \ar[r] &  \oplus_{i \in \bbN} F_i \ar[r] &  \oplus_{i \in \bbN} F_i \ar[r] & P \ar[r] & 0
}}$$
splits, so $P$ is projective.

Our aim is to modify the construction a bit to obtain a pure projective $1$-tilting module. From now on assume that 
$I = Ri_1+ \cdots + Ri_k$ is an idempotent ideal satisfying the following property: If $Ir = 0$ for some $r \in R$ then $r = 0$.
Observe that this property holds if and only if $\varepsilon_{l+1} = f_l \cdots f_2f_1$ is a monomorphism for every $l \in \bbN$.
Further let $\pi_l \colon F_l \to M_l$ be the cokernel of $\varepsilon_l$. Consider the following diagram whose columns are 
short exact sequences

\[\xymatrix{
F_1 \ar@{=}[d]\ar@{=}[r] & F_1 \ar[d]^{\varepsilon_2}\ar@{=}[r] &  F_1 \ar[d]^{\varepsilon_3}\ar@{=}[r] & F_1 \ar[d]^{\varepsilon_4}
\ar@{=}[r] &\cdots \\
F_1 \ar[d]\ar[r]^{f_1} & F_2 \ar[d]^{\pi_2}\ar[r]^{f_2} &  F_3 \ar[d]^{\pi_3}\ar[r]^{f_3} &  F_4 \ar[r]^{f_4} \ar[d]^{\pi_4} & \cdots\\
0 \ar[r] & M_2 \ar[r]^{\overline{f_2}} & M_3 \ar[r]^{\overline{f_3}} & M_4  \ar[r]^{\overline{f_4}} & \cdots
}\]

Considering the direct limits of the rows in this diagram we obtain an exact sequence
$$\vcenter{
\xymatrix{
0 \ar[r] & R \ar[r]^{\varepsilon}& P \ar[r]^{\pi} & Q \ar[r] & 0, 
}}$$
where $P$ is projective and $Q$ is pure projective by Lemma \ref{L:first}. In fact, it is easy to see that 
every homomorphism $g_n \colon F_{n+1} \to F_n$ induces a homomorphism $\overline{g_{n}} \colon M_{n+1} \to M_n$ such that 
$\overline{g_{n+1}} \overline{f_{n+1}} ~ \overline{f_{n}} = \overline{f_{n}}$ holds for every $n \geq 2$. Then the canonical presentation of 
$Q$ as $\li M_n$ splits, in particular $Q$ is a direct summand of $\oplus_{2 \leq i \in \bbN} M_i$. 

 \begin{prop} \label{P:noetherian-case} 
The module $T = P \oplus Q$ constructed above is tilting. In fact, $$T^\perp = \Gen T = \{M \in \ModR \mid MI = M\}.$$ 
 \end{prop}

 \begin{proof}
It is easy to see and well known that $\Gen T = \Gen P = \{M \in \ModR \mid MI = M\}$ since $I$ is the trace ideal of $P$.
Further observe that $T^\perp = Q^\perp$ and $0 \to R \overset{\varepsilon}\to P \overset{\pi}\to Q \to 0$
is a projective presentation of $Q$. Therefore if $N \in Q^\perp$, then every homomorphism $f \colon R \to N$ is of the 
form $f'\varepsilon$ for some $f' \colon P \to N$. Now $NI = N$ is an easy consequence of $PI = P$. So $T^{\perp} \subseteq 
\{M \in \ModR \mid MI = M\}$.

Thus we are left to prove that if $N \in \{M \in \ModR \mid MI = M\}$ then $N \in Q^\perp$ that is for every 
$\varphi \colon R \to N$ there exists $\varphi' \colon P \to N$ such that $\varphi = \varphi' \varepsilon$.
Recall that $P$ is a direct limit of the sequence 
$$\vcenter{
\xymatrix{
F_1 \ar[r]^{f_1} & F_2 \ar[r]^{f_2} & F_3 \ar[r]^{f_3} & \cdots.
}}$$
For any $i \in \bbN$ let $\iota_i \colon F_i \to P$ be the colimit injection. Observe that $F_1 = R$ and $\varepsilon = \iota_1$.
Further $f_1$ is a multiplication by the column consisting of generators of ${}_RI$. Then it is easily verified that for 
a given homomorphism $\varphi\colon F_1 \to N$ there exists $\psi \colon F_2 \to N$ such that $\psi f_1 = \varphi$. For 
every $i \geq 2$ put $\varphi_i = \psi g_2g_3 \cdots g_if_i \colon F_i \to N$. The property $g_{i+1}g_{i}f_i = f_i$ shows that 
$\varphi_{i+1}f_i = \varphi_i$ for every $i \geq 2$. The universal property of direct limits gives the homomorphism $\varphi' \colon P \to N$ such 
that $\varphi'\iota_i = \varphi_i$ for every $i \geq 2$. In particular, $\varphi' \iota_2 f_1= \varphi_2 f_1$. On the LHS of this equality 
is just $\varphi'\varepsilon$ and on the RHS there is $\varphi_2 f_1 = \psi g_2f_2f_1 = \psi f_1 = \varphi$. So 
$\varphi'\varepsilon = \varphi$ and we are done. 
 \end{proof}
 
 \begin{thm}\label{T:noetherian-case} Let $R$ be a ring and let $I \subseteq R$ be an idempotent ideal finitely generated on the left
satisfying $Ir = 0 \Rightarrow r = 0$. Further suppose that the following conditions hold: 
 \begin{enumerate}
  \item every finitely generated projective $R$-module is stably free;
  \item there exists a proper ideal $K$ containing $I$ such that every finite multiple of  $R/K$ is a directly finite module, that is, it is not isomorphic to a proper direct summand of itself;
  \item there exists a flat homomorphism $\varphi \colon R \to S$ of rings such that $S$ is a nontrivial semisimple artinian ring and $\varphi(I) \neq 0$.
 \end{enumerate} 
Then $\ModR$ contains a pure projective $1$-tilting module which is not a direct sum of finitely presented modules. 
\end{thm}

 \begin{proof} We claim that every finitely presented module $M$ of projective dimension at most $1$ satisfying $MI = M$ satisfies 
 $M \otimes_R S = 0$: Let 
\begin{equation} \label{ast} 
\xymatrix{ 
0 \ar[r] & P_1 \ar[r] & P_2 \ar[r] & M \ar[r] & 0
}
\end{equation}
be a projective presentation of such a module, where $P_1,P_2$ are finitely generated projectives.
Because of (i) we may assume $P_1 \simeq R^m$ and $P_2 \simeq R^n$ for some $m,n \in \bbN$. Now apply the functor 
$- \otimes_R R/K$ to the short exact sequence (\ref{ast}). Using $MK=M$ we get an epimorphism $(R/K)^m$ onto $(R/K)^n$. Since $(R/K)^m$ is directly finite, $m \geq n$ follows. 
Finally apply $- \otimes_R S$ to the presentation of $M$ to obtain 
$$\vcenter{
\xymatrix{
0 \ar[r] & S^m \ar[r] & S^n \ar[r] & M \otimes_R S \ar[r] & 0.
}}$$
Since $m \geq n$ and $S$ is an $S$-module of finite length then $m = n$ and $M \otimes_R S =0$. This proves the claim. 

We can complete the proof easily. Let $T$ be a pure projective $1$-tilting module such that $T^{\perp} = \{M \in \ModR \mid MI = M\}$.
Assume that $T = \oplus_{i \in I} M_i$ where every $M_i$ is a finitely presented module. Since $T$ is $1$-tilting, every 
$M_i$ is of projective dimension at most 1 and $M_iI = M_i$. Then $T \otimes_R S = 0$ by the claim. On the other hand 
there is a projective module $P$ of the trace ideal $I$, so $PI = P \in \Gen T$ and $P$ is a direct summand of $T^{(\kappa)}$.
In particular, $P \otimes_R S = 0$. But it is not possible if $\varphi(I) \neq 0$.  
(In fact we proved that $T$ cannot be a direct summand of finitely presented modules from $T^{\perp}$.)
 \end{proof} 
 
Observe that if $R$ is noetherian and $0 \neq I \neq R$ then Condition (2) of Theorem \ref{T:noetherian-case} holds with $I = K$.
Moreover, if $R$ is noetherian semiprime then Condition (3) is a consequence of Goldie's theorem \cite[Theorem~2.3.6, Proposition~2.1.16(ii)]{MR}. So in the noetherian context
we have to care only about the existence of suitable idempotent ideal $I$ and Condition (1). 

Now we can give promised examples. First, let us consider the universal enveloping algebra of ${\rm sl}(2,\mathbb{C})$,  
that is $R = \mathbb{C} \langle h,e,f\rangle /(h=ef-fe,2e = he-eh,-2f = hf-fh)$.
This is a noetherian domain, let $I$ be the ideal generated by $h,e,f$. Obviously $I^2 = I$ and $R/I \simeq \mathbb{C}$. So 
we can take $K = I$ in order to check (2). The condition (3) is satisfied as the inclusion of $R$ into its  
(right) quotient division ring is flat. Finally, according to \cite[Corollary~12.3.3]{MR} every finitely generated projective $R$-module is stably free.

We give one more example, this noetherian domain is a bit artificial but on the other hand it is semilocal and its projective modules 
are classified. For details see \cite[Example~5.1]{HP}. Let $R$ be a semilocal principal ideal domain such that 
$R/J(R) \simeq M_{3}(F) \times M_{3}(F)$, where $F$ is a field  (existence of such a ring follows from the work of Fuller and Shutters \cite{FuSh}, see for example \cite[Examples~3.3]{HP}). Let $\pi \colon R \to M_{3}(F) \times M_3(F)$ be the canonical surjection and let $\iota\colon F \times F$
be given by 
$\iota(x,y) = {\rm diag}(x,y,y) \times {\rm diag}(x,x,y)\,.$
Consider $\Lambda$ to be the pullback of the diagram

\[\xymatrix{
\Lambda \ar[r] \ar[d] & R \ar[d]^{\pi} \\
F \times F \ar[r]^{\iota} &  M_{3}(F)^2 
}\]

Then $\Lambda$ is a noetherian domain, $\Lambda/J(\Lambda) \simeq F \times F$, every finitely generated projective right $\Lambda$-module is free but there exists a countably generated projective module $P \in \Modr{\Lambda}$ such that $0 \neq {\rm Tr} (P) \neq R$. So we can apply Theorem \ref{T:noetherian-case} directly.
 
 
\section{$\tau$-rigidity and silting}

In this section, we will consider various classes of large (that is, not necessarily finitely generated) modules related to tilting: the silting, quasi-tilting and $\tau$-rigid ones. The detailed analysis of the Dubrovin-Puninski example in the Appendix will help us here in finding counter-examples. First, we recall the relevant definitions and basic properties.

\begin{defn}\label{def_tau}\cite{W} \rm
Let $R$ be a ring and $T$ be a module. Then $T$ is a (large) \emph{$\tau$-rigid} module provided there exists a projective presentation 
\begin{equation} \label{e1}
\xymatrix{P_1 \ar[r]^{\varphi} & P_0 \ar[r] & T \ar[r] & 0 }
\end{equation}
such that $\HomJT R{\varphi}M$ is surjective for each module $M \in \Gen T.$
\end{defn}

Note that $\tau$-rigidity can be tested in a simpler form:

\begin{lem}\label{variants} 
\begin{enumerate}
\item In Definition~\ref{def_tau} above, $\Gen T$ can be replaced with $\Dsum T$, where $\Dsum T$ denotes the class of all direct sums of copies of the module $T$. Moreover, we can assume that $P_1$ and $P_0$ are free modules.
\item If $M$ is $\tau$-rigid, then $\Gen T \subseteq T^\perp.$
\end{enumerate}
\end{lem} 
\begin{proof} (1) The first assertion follows from the projectivity of $P_1$, and the second by Eilenberg's trick (i.e., the fact that for each projective module $P$ there is a free module $F$ such that $P \oplus F \cong F$).

(2) The projective presentation in Definition \ref{def_tau} has the property that each homomorphism $h \in \HomJT R{\mbox{Im}(\varphi)}M$ with $M \in \Gen T$ extends to $P_0$. Since $P_0$ is projective, this just says that $\Gen T \subseteq T^\perp$. 
\end{proof}

\begin{defn}\label{classic} \rm
Let $R$ be a ring and $T$ be a module. 
\begin{enumerate} 
\item $T$ is \emph{finendo}, provided $T$ is finitely generated over its endomorphism ring. 
\item $T$ is (large) \emph{quasi-tilting} provided that $\Pres T = \Gen T \subseteq T^\perp.$ 
\item $T$ is (large) \emph{silting} provided that there exists a projective presentation (\ref{e1}) such that for each module
$M,$ $\HomJT R{\varphi}M$ is surjective iff $M \in \Gen T$.    
\end{enumerate}
\end{defn}

The following Lemma recalls some relations among the notions defined above:

\begin{lem}\label{rel} Let $R$ be a ring and $T$ be a module.
\begin{enumerate}
\item Each silting module is finendo, quasi-tilting, and $\tau$-rigid.
\item $T$ is finendo and quasi-tilting, iff $T$ is a $1$-tilting $R/\ann T$-module.
\item $T$ is $1$-tilting, iff $T$ is faithful silting, iff $T$ is faithful finendo and quasi-tilting. 
 \end{enumerate}
\end{lem}
\begin{proof} (1) follows by \cite[Proposition 3.13(1)]{AMV}, and (2) by \cite[Lemma 3.4]{AMV}.
(3) Clearly, each $1$-tilting module is faithful and silting; the rest follows by parts (1) and (2).
\end{proof}
  
\begin{expl}\label{artina} The general picture simplifies a lot in the particular setting of finitely generated modules over artin algebras: 
The silting and quasi-tilting modules coincide (e.g., by \cite[3.16]{AMV}), and the converse of Lemma \ref{variants}(ii) holds, because the $\tau$-rigidity of $T$ can be expressed simply by the formula $\HomJT RT{{\tau}T} = 0$, where $\tau$ denotes the AR-translation \cite{AIR} (hence the term '$\tau$-rigid'); the latter is further equivalent to the condition of $\Gen T \subseteq T^\perp$ by \cite{AS}.
\end{expl} 

We will briefly stop at the relation of $\tau$-rigidity to the condition of $\Gen T \subseteq T^\perp$ (cf.\ Lemma \ref{variants}(2)). They are known to be equivalent for any ring $R$ in the case when $T$ has projective dimension $\leq 1$ (so, in particular, for all right hereditary rings $R$), cf.\ \cite{W}. The following proposition extends this equivalence to further classes of rings and modules:
 
\begin{prop}\label{prop} Let $R$ be a ring and $T$ be a module possessing a projective resolution $\sigma$ such that the first syzygy $\Omega^1(T)$ of $T$ in $\sigma$ has a projective cover. Then $T$ is $\tau$-rigid, if and only if $\Gen T \subseteq T^\perp$.
\end{prop}
\begin{proof} We only have to prove the if-part. 

Consider the exact sequences $0 \to \Omega^1(T) \overset{\mu}\hookrightarrow P_0 \to  T \to 0$ and  $0 \to K \hookrightarrow P_1 \overset{\varphi}\to \Omega^1(T) \to 0$ where $\varphi$ is a projective cover of $\Omega^1(T)$. Let $X$ be a set and $g \in \HomJT R{P_1}{T^{(X)}}$. We will prove that $g$ factors through $\mu \varphi$ provided that $\Gen T \subseteq T^\perp$.

Consider the pushout of $\varphi$ and $g$:
$$\begin{CD} 
	0 @>>> K @>{\subseteq}>>	P_1 @>{\varphi}>> \Omega^1(T) @>>> 0 	\\
	@. @V{g^\prime}VV			@V{g}VV	  @V{k}VV   @. \\
	0 @>>> g(K) @>{\subseteq}>>	T^{(X)} @>{\pi}>> G  @>>> 0 \\
\end{CD}$$
Here, $g^\prime$ denotes the restriction of $g$ to $K$. Clearly, $G \in \Gen T$.

If $\Gen T \subseteq T^\perp$, then $k$ extends to some $k^+ \in \HomJT R{P_0}G$. As $P_0$ is projective, there exists $h \in \HomJT R{P_0}{T^{(X)}}$ such that $\pi h = k^+$.

Let $\delta = g - h \mu \varphi$. Then $\pi \delta = \pi g - \pi h \mu \varphi = \pi g - k^+ \mu \varphi = \pi g - k \varphi = 0$.  

Let $x \in P_1$. Then $\delta(x) = g(s)$ for some $s \in K$. Since $h \mu \varphi (s) = 0$, also $\delta(x) = \delta(s)$, whence $x - s \in \mbox{Ker}(\delta)$. This proves that $P_1 = \mbox{Ker}(\delta) + K$. Since $\varphi$ is a projective cover of $\Omega^1(T)$, the module $K$ is superfluous in $P_1$. This yields $\delta = 0$, and $g = h \mu \varphi$ is the desired factorization of $g$ through $\mu \varphi$. So $T$ is $\tau$-rigid by Lemma \ref{variants}(1).  
\end{proof}      

There are several cases that fit the setting of Proposition \ref{prop}:

\begin{cor}\label{cor} Let $R$ be a ring and $T$ a module. Assume that either 
\begin{enumerate}
\item $R$ is right pefect, or
\item $R$ is semiperfect and $T$ is finitely presented, or
\item $T$ has projective dimension $\leq 1$.
\end{enumerate}
Then $T$ is $\tau$-rigid if and only if $\Gen T \subseteq T^\perp.$ In particular, if $T$ is quasi-tilting, then $T$ is $\tau$-rigid.
\end{cor}

\begin{rem} \rm Note that Corollary \ref{cor} is new even for artin algebras (which are covered by the first case above), since the characterization of $\tau$-rigidity in \cite{AIR} mentioned in Example \ref{artina} is restricted to finitely generated modules. 
\end{rem}  

Let us now look at higher projective dimensions:

\begin{lem}\label{more} Let $T$ be a faithful module of projective dimension $> 1$. Then
\begin{enumerate}
\item $T$ is not $\tau$-rigid.
\item If $R$ embeds into $T^k$ for some finite $k$, $T$ has projective dimension $2$, and $\ExtJT 2RT{T^{(\kappa)}} = 0$ for each $\kappa$, then $\Gen T \nsubseteq T^\perp$.  
\end{enumerate}
\end{lem}
\begin{proof} (1) Let  $F_1 \overset{\varphi}\to F_0 \to T \to 0$ be a presentation of $T$ such that $F_1$ and $F_0$ are free modules (see Lemma \ref{variants}(1)). 
Since $T$ has projective dimension $> 1$, $\varphi$ is not monic. Let $B = \{ b_i \mid i \in I \}$ be a free basis of $F_1$ and consider $0 \neq x = \sum_{i \in I} b_i r_i \in \mbox{Ker}(\varphi)$. W.l.o.g., $r_0 \neq 0$. Since $T$ is faithful, there exists $y \in T$ such that $y r_0 \neq 0$.  

Define $g \in \HomJT R{F_1}T$ by $g(b_0) = y$ and $g(b_i) = 0$ for $0 \neq i \in I$. Then $g(x) \neq 0$, so $g$ does not factorize through $\varphi$. Thus $\HomJT R{\varphi}T$ is not surjective, and $T$ is not $\tau$-rigid.

(2) Assume that $\Gen T \subseteq T^\perp$. Since $R \subseteq T^k$, also $R^{(\kappa)} \subseteq T^{(k . \kappa)}$ for each cardinal $\kappa$. Both $\ExtJT 1RT{T^{(k . \kappa)}/R^{(\kappa)}} = 0$ and $\ExtJT 2RT{T^{(k . \kappa)}} = 0$, whence $\ExtJT 2RT{R^{(\kappa)}} = 0$, and $\ExtJT 1R{\Omega^1(T)}{R^{(\kappa)}} = 0$ by dimension shifting. Since $\Omega^{1}(T)$ has projective dimension $1$, we infer that $\ExtJT 1R{\Omega^{1}(T)}M = 0$ for each module $M$. Then $T$ has projective dimension $1$, a contradiction. 
\end{proof}

\begin{cor}\label{2-tilt} Let $T$ be a tilting module of projective dimension $2$. Then $\Gen T \nsubseteq T^\perp$, and hence $T$ is not $\tau$-rigid.
\end{cor} 

In general, $\Gen T \subseteq T^\perp$ is a weaker condition than the $\tau$-rigidity of $T$: in fact, in contrast with Example \ref{artina}, quasi-tilting modules need not be silting, and faithful quasi-tilting modules need not be tilting:

\begin{expl}\label{prod} \rm Let $R$ be a commutative von Neumann regular ring which is not artinian (e.g., let $R$ be an infinite direct product of fields). Then each finitely generated submodule of a projective module is a direct summand \cite[1.11]{G}, and each projective module is isomorphic to a direct sum of cyclic modules of the form $eR$ where $e$ is an idempotent in $R$. 

We claim that each cyclic module $C = R/I$ is silting. First, the ideal $I$ is completely determined by the set $E$ of the idempotents it contains. Let $P = \bigoplus_{e \in E} eR$. Consider the presentation         
\begin{equation} \label{e2}
\xymatrix{0 \ar[r] & K \ar[r] & P^2 \ar[r]^{\varphi} & R \ar[r] & C \ar[r] & 0} 
\end{equation}
where $\varphi (e,f) = e$ for each $e, f \in E$ (whence $\Img \varphi = I$), and $K = \Ker \varphi$. If $M \in \Gen C$, then for each $e \in E$, $\HomJT R{eR}M = 0$, because each element of $M$ is annihilated by $e$. So $\varphi$ witnesses the $\tau$-rigidity of $C$.  
Conversely, if $\HomJT R{eR}M = 0$ for each $e \in E$, then $M.I = 0$, whence $M$ is an $R/I$-module, i.e., $M \in \Gen C$. 

Let $M$ be a module such that $\HomJT R{eR}M \neq 0$ for some $e \in E$. Since $K$ contains a copy of $P$, there exists $g \in \HomJT R{P^2}M$ such that $g \restriction K \neq 0$, whence $g$ does not factorize through $\varphi$. This shows that if $M$ is a module such that each homomorphism from $P^2$ into $M$ factorizes through $\varphi$, then $\HomJT R{eR}M = 0$ for each $Xe \in E$, whence $M \in \Gen C$, and the claim is proved. 

Notice that in the sequence (\ref{e2}) witnessing that $C$ is silting, the map $\varphi$ is not monic, even if $I$ is non-zero projective (which is the case. e.g., when $I \neq 0$ is countably generated). Indeed, in the latter case each projective resolution $0 \to P_1 \overset{\varphi}\to P_0 \to C \to 0$ of $C$ witnesses that $C$ is $\tau$-rigid, but not that $C$ is silting, because $\Gen C \subsetneq C^\perp$.      
  
Moreover, $\rmod {R/I} = \Pres C = \Gen C \subseteq C^\perp$, because $I$ is an idempotent ideal in $R$, so $\ExtJT 1R{R/I}M \cong \ExtJT 1{R/I}{R/I}M = 0$ for each $R/I$-module $M$. So $C$ is a finitely generated finendo quasi-tilting module (which follow also by Lemma \ref{rel}(1)). 

Notice that each ideal $I$ of $R$ defines the cyclic silting module $C = R/I$, and different ideals yield non-equivalent silting modules. Such abundance of silting modules contrasts with the fact that there exists only one tilting module up to equivalence, namely $R,$ because all finitely presented modules are projective. 

Let $\simpR$ denote a representative set of all simple modules. By the above, each $S \in \simpR$ is silting, and it is $\sum$-injective by \cite[3.2 and 6.18]{G}. Moreover, for each finite subset $F$ of $\simpR$, the direct sum $D = \bigoplus_{S \in F} S$ is silting (and $\sum$-injective), because $D \cong R/I$ for $I = \bigcap_{S \in F} \Ann S$. 

\medskip
However, infinite direct sums of simple modules need not be silting, even if they are quasi-tilting: to see this, we consider the particular case when $R = K^\kappa$ is an infinite product of copies of a field $K$. For each $r = (r_i)_{i < \kappa} \in R$, we denote by $\mbox{supp}(r)$ the support of $r$, that is, the set of all $i < \kappa$ such that $r_i \neq 0$.  For each $X \subseteq \kappa$, we denote by $e_X$ the characteristic function of $X$, that is, the unique idempotent in $R$ satisfying $X = \mbox{supp}(e_X)$. Simple projective modules coincide up to isomorphism with the simple submodules of $R$, and these are exactly the right ideals $S_\alpha$ generated by the idempotents $e_{X_\alpha} \in R$ ($\alpha < \kappa$) where $X_\alpha = \{ \alpha \}$. In particular, $\mbox{Soc}(R) = \bigoplus_{j < \kappa} S_j$ is projective. 

Let $\mathcal F$ be a filter on $\kappa$. Then $I_{\mathcal F} = \{ r \in R \mid (\kappa \setminus \mbox{supp}(r)) \in \mathcal F \}$ is an ideal of $R$. Conversely, each ideal $I$ in $R$ is determined by the idempotents it contains, and the complements of supports of those idempotents form a filter $\mathcal F$ such that $I = I_{\mathcal F}$. For example, if $\mathcal F _0$ denotes the Fr\' echet filter of all cofinite subsets of $\kappa$, then $I_{\mathcal F _0} = \mbox{Soc}(R)$. 

Let $\mathcal G$ be a non-principal ultrafilter on $\kappa$. Then for each $X \subseteq \kappa$, either $X \in \mathcal G$, or else $\kappa \setminus X \in \mathcal G$. Also $I_{\mathcal F _0} \subseteq I_{\mathcal G}$, and $I_{\mathcal G}$ is a maximal ideal in $R$.  

Let $S$ be a simple non-projective module and $T = \mbox{Soc}(R) \oplus S$. Then $T$ has projective dimension $> 1$ (in fact, assuming the Continuum Hypothesis, $R$ has global dimension $2$ by \cite[2.51]{O}, whence the projective dimension of $T$ equals $2$). 

We claim that $T$ is quasi-tilting. First, since $T$ is semisimple, $\Pres T = \Gen T$ consists of direct sums of copies of $S$ and the projective simple modules $S_\alpha$ ($\alpha < \kappa$). So we only have to prove that $\ExtJT 1RSD = 0$ where $D = \bigoplus_{\alpha < \kappa} S_\alpha^{(\lambda_\alpha)}$ for some cardinals $\lambda_\alpha$ ($\alpha < \kappa$) such that $\lambda_\alpha \neq 0$ for infinitely many $\alpha < \kappa$. Since the module $P = \prod_{\alpha < \kappa} S_\alpha^{(\lambda_\alpha)}$ is injective, it suffices to show that $\HomJT RS{P/D} = 0$. 

Assume this is not the case, take $0 \neq \varphi \in \HomJT RS{P/D}$, and let $x = (x_j)_{j < \kappa} \in P \setminus D$ be such that $\varphi(1 + I_{\mathcal G}) = x + D$. Then $xI_{\mathcal G} \subseteq D$. Let $A = \{ j < \kappa \mid x_j \neq 0 \}$. Then $A \in \mathcal G$ (otherwise $\kappa \setminus A  \in \mathcal G$, so $e_A \in I_{\mathcal G}$, but $x.e_A = x \notin D$). If we express $A$ as a disjoint union of two infinite subsets, $A = B \cup C$, then either $B \notin \mathcal G$ or $C \notin \mathcal G$, whence either $e_B \in I_{\mathcal G}$ or $e_C \in I_{\mathcal G}$. However, both $x.e_B \notin D$ and $x.e_C \notin D$, a contradiction. This proves our claim. 

Finally, since $\mbox{Soc}(R)$ is a faithful module, so is $T$, whence $T$ is not $\tau$-rigid by Lemma \ref{more}(1).
\end{expl} 

In fact, there even exist finendo quasi-tilting modules that are not silting. For an example, we use an instance of the Dubrovin-Puninski ring from the Appendix.    
 
\begin{expl}\label{dubpun} As in the Appendix, let $S$ denote the Dubrovin-Puninski ring. Using the notation of the Appendix, we let $T = S/K$ where $K = Sf$. We know that $T$ is a simple finitely presented injective left $S$-module of projective dimension $2$, in particular, $\Pres T = \Gen T$. By Corollary~\ref{tilt2}, $T$ is a direct summand in a $2$-tilting module, so $T$ is a quasi-tilting left $S$-module. Since $K$ is a two-sided ideal in $S$, $\End_S(T) = S/K$, whence $T$ is finendo. 

Assuming moreover that each ideal in $S$ is countably generated, we prove that $T$ is not $\tau$-rigid: If $F_1 \overset{\varphi}\to F_0 \to T \to 0$ is a free presentation of $T$ (see Lemma \ref{variants}(1)) and $L = \Ker \varphi$, then $L$ is non-zero projective. By assumption, $I = I_\omega$ is a countably presented projective left $S$-module, and each projective module is isomorphic to a direct sum of $S$ and $I$ (cf.\ \cite[Theorem 7.3.]{DP}).
Since $S \subseteq I_\omega$ and $T$ is injective, we infer that $\HomJT SLT \neq 0$, hence there exist homomorphisms from $F_1$ to $T$ that do not factorize through $\varphi$. Thus $T$ is not $\tau$-rigid.
\end{expl} 

Recently, Angeleri and Hrbek \cite[Example 5.4]{AHr} have shown that if $R$ is a commutative local ring with idempotent maximal ideal $m \neq 0$, then the simple module $T = R/m$ is finendo and quasi-tilting, but not silting. However, neither that module is $\tau$-rigid, so the following question remains open:  

\bigskip
{\bf Open problem:} Does the converse of Lemma \ref{rel}(1) hold in general, that is, are finendo, quasi-tilting, and $\tau$-rigid modules necessarily silting?

\section*{Appendix: Dubrovin-Puninski Rings} \label{s:Ivo}

A uniserial domain $R$ is {\em nearly simple}~\cite[Chapter 14]{Pun} if there is only one nontrivial two-sided ideal of $R,$ the Jacobson radical $J(R).$ Such rings do exist~\cite{BBT, BD} and enjoy the feature~\cite[Prop 14.7]{Pun} that if $a$ and $b$ are nonzero elements of $J(R),$ then there exists an isomorphism $R/aR \simeq R/bR$ between the associated cyclically presented torsion right $R$-modules. A {\em Dubrovin-Puninski} ring is any ring of the form $S = \End_R X,$ where $X_R$ is the cyclically presented torsion module over a nearly simple uniserial domain $R.$ We regard ${_S}X$ as a left module over $S.$ Note that the property of being a nearly simple uniserial domain is left-right symmetric and that the cyclically presented torsion right $R^{\op}$-module is given by the Auslander-Bridger transpose $\Tr (X)$ of $X,$ whose endomorphism ring, acting on the left, is isomorphic to $S^{\op}.$ The property of being a Dubrovin-Puninski ring is therefore also left-right symmetric.

In this section, $S$ will always refer to a Dubrovin-Puninski ring and all $S$-modules will be {\em left} $S$-modules.
\bigskip

\noindent{\bf a.\ The three ideals of $S.$} 
\bigskip

Because the module $X_R$ is uniserial, the ring $S = \End_R X$ contains two important (two-sided) ideals: the ideal $I$ of morphisms that are not monomorphisms and the ideal $K$ of morphisms that are not epimorphisms. Dubrovin and Puninski~\cite{DP} have undertaken a deep study of the model theory of modules over the ring $S.$ We will recall the relevant results that they obtained in the form of {\bf Facts.}

\begin{fact} \label{the ideal K} {\rm (\cite[Lemmas 5.3, 6.1]{DP})}
The following properties hold for the ideal $K \subseteq S$ of non epimorphisms:
\begin{enumerate}
\item the left ideal ${_S}K = Sf$ is principal and uniserial, generated by any monomorphism $f \in K;$
\item the right ideal $K_S$ is not finitely generated;
\item the quotient ring $S/K$ is a division ring;
\item the simple left $S$-module $S/K$ is injective, but not flat.
\end{enumerate}
\end{fact}

The subcategory $\Add (S/K) \subseteq \Modl S$ consists of the semisimple modules that are isomorphic to a direct sum of copies of $S/K.$ Let us verify that this subcategory is a torsion-torsion free (TTF) class. It is clearly closed under submodules, coproducts, and quotient modules. Because $\Add (S/K)$ is isomorphic to the category $\Modl S/K$ of left vector spaces over the division ring $S/K,$ it is closed under direct products. Finally, note that Fact~\ref{the ideal K}(1) implies that $K$ is idempotent $K^2 = K$ and therefore that $\Add (S/K)$ is closed under extensions. From the present point of view, it is preferable to regard $\Add (S/K)$ as a torsion class. The corresponding torsion free class is denoted by $\F_K$ and it consists of those left $S$-modules that do not admit $S/K$ as a submodule. Because $S/K$ is injective, it means that these modules do not admit $S/K$ as a direct summand. 

\begin{fact} \label{the ideal I} {\rm (\cite[Lemma 5.3, Corollary 5.6]{DP})}
The following properties hold for the ideal $I \subseteq S$ of non monomorphisms:
\begin{enumerate}
\item the right ideal $I_S = gS$ is principal and uniserial, generated by any epimorphism $g \in I;$
\item the left ideal ${_S}I$ is not finitely generated;
\item the quotient ring $S/I$ is a division ring;
\item the simple left $S$-module $S/I$ is flat, but not injective.
\end{enumerate}
\end{fact}

Arguments analogous to those used above to prove that $\Add (S/K)$ is a TTF class show that the ideal $I$ is idempotent, $I^2 = I,$ and that the subcategory $\Add (S/I) \subseteq \Modl S$ also constitutes a TTF class. However, it is preferable from the present point of view to regard $\Add (S/I)$ as a torsion free class. The corresponding torsion class is denoted by $\T_I$ and it consists of those left $S$-modules $M$ that satisfy 
$$M = IM = gSM = gM.$$ 

\begin{fact} \label{the ideal J} {\rm (\cite[Theorem 9.1]{Pun})}
The only nontrivial two-sided ideals of $S$ are $K,$ $I$ and the Jacobson radical $J = I \cap K.$ Considered as a left or right ideal, $J$ is uniserial.
\end{fact}  

The second statement of Fact~\ref{the ideal J} follows from Facts~\ref{the ideal K}(1) and~\ref{the ideal I}(1). The first statement of Fact~\ref{the ideal J} implies that the nonzero ideal $J^2$ must equal $J,$ and therefore that $J$ is an idempotent ideal and therefore that $\Add (S/J) = \Add (S/K \oplus S/I)$ is also a TTF class. Similarly, it shows that any product of two distinct nontrivial ideals is equal to $J.$\bigskip

\noindent {\bf b.\ Five indecomposable pure injective $S$-modules.} \bigskip

The (left) Ziegler spectrum $\Zg (S)$ of $S$ is a topological space whose points are the indecomposable pure injective left $S$-modules. The topology on $\Zg (S)$ may be introduced via the definable subcategories $\D \subseteq \Modl S.$ The rule $\D \mapsto \Cl (\D) := \D \cap \Zg (S)$ is then a bijective correspondence between the definable subcategories $\D \subseteq \Modl S$ and the closed subsets of the Ziegler spectrum $\Zg (S).$ Ziegler established that the rule $\D \mapsto \D \cap \Zg (S)$ is a bijective correspondence between the collection of {\em additive elementary classes} of $\Modl S,$ relative to the language $\Lang (S)$ of left $S$-modules, and Crawley-Boevey proved that a subcategory is an additive elementary class in $\Modl S$ if and only if it is definable.

We have already encountered several indecomposable pure injective $S$-modules. Because every injective representation is pure injective, we have that $S/K$ and the injective envelope $E(S/I)$ of the simple representation $S/I$ are both examples of indecomposable pure injective representations. Because ${_S}J$ is uniserial it is uniform. It contains neither $S/K$ nor $S/I$ as a submodule, so its injective envelope $E(J)$  is another point of the Ziegler spectrum. As ${_S}J$ is essential in ${_S}S,$ we can also represent this point as $E(J) = E(I) = E(S).$ The simple flat module $S/I$ is endosimple and therefore pure injective, by~\cite[]{PSL}. Finally, the pure injective envelope $\PE (J)$ of the left $S$-module ${_S}J$ is indecomposable, because $J$ is both left and right uniserial. This provides us with a {\em core} subset of the Ziegler spectrum
$$\Zg_0 (S) = \{ S/K, \PE (J), S/I, E(S/I), E(I) \} \subseteq \Zg (S)$$
that will appear later. Puninski has found an example (unpublished) of a Dubrovin-Puninski ring $S$ for which $\Zg_0 (S) = \Zg (S).$ 

We have already encountered several definable subcategories of $\Modl S.$ Every TTF class of $\Modl S$ is a definable subcategory, so that we can list five such categories: $\Modl S$ and the four semisimple categories: $0,$ $\Add (S/K),$ $\Add (S/I),$ and $\Add (S/J).$ The closed subsets of $\Zg (S)$ that correspond to the four semisimple TTF classes are: $\Cl (0) = \emptyset,$ $\Cl (\Add (S/K)) = \{ S/K \},$ $\Cl (\Add (S/I)) = \{ S/I \},$ and $\Cl (\Add (S/J)) = \{ S/K, S/I \}.$ Of course $\Cl (\Modl S) = \Zg (S).$

\begin{thm} \label{Zg spec}
The decomposition of $\Zg (S)$ into its connected components is given by 
$$\Zg (S) = \{ S/K \} \dotcup \{ \PE (J) \}^- \dotcup \{ S/I \},$$
where $\{ \PE (J) \}^-$ denotes the closure of the point $\PE (J).$ The points $S/K,$ $S/I,$ and $\PE (J)$ are isolated and the
$3$-point set $\{ S/K, \PE (J), S/I \}$ is dense in $\Zg (S);$ the complement 
$\Zg'(S) = \Zg (S) \setminus \{ S/K, \PE (J), S/I \}$ consists of the flat indecomposable injective $S$-modules, including $E(S/I)$ and $E(I).$ The relative subspace topology induced on $\Zg' (S)$ is the indiscrete topology.
\end{thm}

Theorem~\ref{Zg spec} asserts that the number of connected components of $\Zg (S)$ is finite, so that each connected component is clopen. The decomposition of $\Zg (S)$ into its connected components also implies that the points $S/K$ and $S/I$ are isolated. Thus we are left with the task of proving that the point $\PE (J)$ is also isolated and that 
$$\{ \PE (J) \}^- = \{ \PE (J) \} \dotcup \Zg'(S),$$
where $\Zg' (S)$ is a closed subset whose points are topologically indistinguishable. This will all be proved following the characterization in Theorem~\ref{Max Spec} of the simple objects in the functor category $(\modr S, \Ab).$

Theorem~\ref{Zg spec} implies that $\Zg (S)$ has twelve closed subsets. Every closed subset is determined by its intersection with the clopen connected components. The connected components $\{ S/K \}$ and $\{ S/I \}$ are singleton and the three possibilities for the intersection with $\{ \PE (J) \}^-$ are: $\emptyset,$ $\{ \PE (J) \}^-$ and $\Zg'(S).$ \bigskip

\begin{fact} \label{coherence} {\rm (\cite[Lemma 4.4]{DP})}
Every Dubrovin-Puninski ring is left and right coherent.
\end{fact}

If $S$ is a left coherent ring, then the subcategory $\SAbs \subseteq \Modl S$ of absolutely pure left $S$-modules constitutes a definable subcategory. An absolutely pure module is pure injective if and only if it is injective, so that the closed subset 
$\Cl (\SAbs) \subseteq \Zg (S)$ consists of the indecomposable injective left $S$-modules. Dually, if $S$ is right coherent, then the subcategory $\SFlat \subseteq \Modl S$ of flat left $S$-modules is definable. We shall prove that $\SAbs \cap \SFlat$ is a minimal nonzero definable subcategory of $\Modl S$ and that $\Zg'(S) = \Cl (\SAbs \cap \SFlat).$ \bigskip

\noindent{\bf c.\ The free resolution of $S/K.$} \bigskip

In the next few sections, we will describe six flat resolutions of the finitely presented injective left $S$-module $S/K.$ Each of these resolutions has its own distinct features that reveal various aspects of the module category $\Modl S.$

As $S$ is a Dubrovin-Puninski ring, there exists a nearly simple uniserial domain $R$ with a cyclically presented torsion module $X_R$ such that $S = \End_R X.$ The uniserial module $X_R$ has the property that whenever $Y_R \subseteq X_R$ is a finitely generated submodule, then both $Y$ and $X/Y$ are isomorphic to $X.$ Let us fix a short exact sequence 
$$
\vcenter{%
\xymatrix@C=40pt@R=40pt
{%
0 \ar[r] & X \ar[r]^{f} & X \ar[r]^{g} & X \ar[r] & 0     
}}
$$
in the category $\ModR$ of right $R$-modules. The morphism $f : X_R \to X_R$ is a monomorphism that is not an epimorphism, so that $K = Sf;$ and the morphism $g : X_R \to X_R$ is an epimorphism that is not a monomorphism, so that $I = gS.$ As $\Ker g = \Img f,$ we obtain the following free resolution of the simple injective left $S$-module $S/K.$
\begin{equation} \label{free}
\vcenter{%
\xymatrix@C=15pt@R=15pt
{%
0 \ar[rr] && S \ar[rrr]^{- \times g} &&& S \ar[drr] \ar[rrrr]^{- \times f}  &&&& S \ar[rrr]^{\pi_K} &&& S/K \ar[rr] && 0 \\
         &&                       &&&&&   K \ar[urr] \ar[drr] \\
         &&                       &&& 0 \ar[urr] &&&& 0     
}}
\end{equation}
The first syzygy of $S/K$ is given by $K = Sf;$ exactness at the second term follows from the choice of $g$ as the cokernel of $f,$ which implies that the left annihilator of $f$ in $S$ is $Sg;$ and exactness at the left term follows from the fact that $g$ is an epimorphism, and therefore that
$Sg \simeq S.$ 

Resolution~(\ref{free}) of $S/K$ allows us to calculate the higher derived functors of $(S/K, -) := \Hom_S (S/K, -).$ Given a left $S$-module ${_S}M,$ apply the contravariant functor $\Hom_S (-,M)$ and delete the left term to obtain the complex
$$
\vcenter{%
\xymatrix@C=35pt@R=35pt
{%
M \ar[r]^{f \times -} & M \ar[r]^{g \times -} & M \ar[r] & 0
}}
$$
of abelian groups, whose homology is given by $\Ext_S (S/K, M) = M[g]/fM,$ where $M[g] := \{ m \in M \; | \; gm = 0 \},$ and
\begin{equation} \label{ext equation}
\Ext_S^2 (S/K, M) = \Ext_S^1 (K,M) = M/gM = M/IM. 
\end{equation}

\begin{expl}
Let $X$ be the cyclically presented torsion $R$-module $X_R,$ considered as a left $S$-module. Then $g : X \to X$ is an epimorphism, $f: X \to X$ a monomorphism, and $X[g] = fX,$ so that $(S/K, X) = \Ext^1_S (S/K, X) = \Ext^2_S (S/K,X) = 0.$ 
\end{expl}

\begin{expl} \label{free module}
The ring $S,$ considered as a left module over itself, also satisfies $S[g] = fS,$ because $f$ is the kernel of $g : X_R \to X_R.$ Thus $\Ext^1 (S/K, S) = 0.$
\end{expl}

Resolution~(\ref{free}) of $S/K$ consists of free modules of finite rank, so that both of the syzygies $K = \Omega(S/K)$ and 
$S = \Omega^2 (S/K)$ are finitely presented.

\begin{prop} \label{max def cats}
The following subcategories of $\Modl S$ are definable:
\begin{itemize} 
\item $(S/K)^{\perp_0} = \Ker \Hom_S (S/K,-) := \{ M \in \Modl S \; | \; \Hom_S (S/K,M) = 0 \};$
\item $(S/K)^{\perp_1} = \Ker \Ext^1_S (S/K,-) := \{ M \in \Modl S \; | \; \Ext^1_S (S/K,M) = 0 \};$
\item $(S/K)^{\perp_2} = \Ker \Ext^2_S (S/K,-) := \{ M \in \Modl S \; | \; \Ext^2_S (S/K,M) = 0 \}.$
\end{itemize}
\end{prop}

We have already encountered two of these definable subcategories. Because $S/K$ is a simple injective, it is evident that 
$(S/K)^{\perp_0} = \F_K.$ Dually, we have that $\T_I = (S/K)^{\perp_2}.$ For, a module ${_S}M$ belongs to $\T_I$ if and only if $M = gM = IM$ and $\Ext^2 (S/K,M) = M/IM.$ The importance of Resolution~(\ref{free}) of $S/K$ derives from the following.

\begin{fact} \label{fp modules} {\rm (\cite[Theorem 5.5]{DP})}
Every finitely presented left $S$-module is isomorphic to a direct sum of copies of the indecomposable modules $S/K,$ 
$\Omega (S/K) = K,$ and $\Omega^2 (S/K) = S.$
\end{fact} 

A left $S$-module $M$ is {\em absolutely pure} if $\Ext_S^1 (G,M) = 0,$ for every finitely presented left $S$-module $G.$ By Fact~\ref{fp modules}, a module ${_S}M$ is absolutely pure provided that $\Ext^1_S (S/K, M) = \Ext^1_S (K,M) = \Ext^1_S (S,M) = 0.$ It follows that $$\SAbs = (S/K)^{\perp_1} \cap K^{\perp_1}= (S/K)^{\perp_1} \cap (S/K)^{\perp_2}.$$ The dual result is given by the following.

\begin{prop} \label{SFlat}
$\SFlat = (S/K)^{\perp_0} \cap (S/K)^{\perp_1}.$
\end{prop}

\begin{proof}
To verify that $\SFlat \subseteq (S/K)^{\perp_0} \cap (S/K)^{\perp_1}$ notice that the regular representation ${_S}S$ belongs to 
$(S/K)^{\perp_0}$ and $(S/K)^{\perp_1}$ (Example~\ref{free module}). As $\SFlat$ is the smallest definable subcategory of $\Modl S$ that contains $S,$ the inclusion follows.

To prove the inclusion $(S/K)^{\perp_0} \cap (S/K)^{\perp_1} \subseteq \SFlat,$ let us apply Fact~\ref{fp modules} together with the criterion that a module ${_S}M$ is flat provided that every morphism $n : G \to M$ from a finitely presented module factors through a projective module. It suffices to verify the condition for $G = S/K$ and $G = K.$ By the hypothesis $M \in (S/K)^{\perp_0},$ there are no nonzero morphisms $n: S/K \to M.$ By the hypothesis $M \in (S/K)^{\perp_1},$ a morphism $n : K \to M$ extends to $S$ as indicated by the diagram
$$
\vcenter{%
\xymatrix@C=35pt@R=35pt
{%
0 \ar[r] & K \ar[r] \ar[d]^{n} & S \ar[r] \ar@{.>}[ld] & S/K \ar[r] & 0 \\
& M.
}}
$$
\end{proof}

We note the following observation for future reference. 

\begin{prop} \label{abs and flat}
$\SAbs \cap \SFlat = (S/K)^{\perp_0} \cap (S/K)^{\perp_1} \cap (S/K)^{\perp_2}.$
\end{prop}

We will show that, while the definable subcategory $\SAbs \cap \SFlat$ is a minimal nonzero definable subcategory of $\Modl S,$ the three definable subcategories $(S/K)^{\perp_i}$ are maximal proper definable subcategories of $\Modl S$ for $i = 0,$ $1,$ and $2.$\bigskip 

\noindent{\bf d.\ The twelve definable subcategories of $\Modl S.$} \bigskip

According to Fact~\ref{the ideal J}, the structure of the two-sided ideals of $S$ may be depicted by the pushout diagram 
$$
\vcenter{%
\xymatrix@C=25pt@R=25pt
{%
& 0 \ar[d] & 0 \ar[d] \\ 
0 \ar[r] & J \ar[r] \ar[d] & I \ar[r] \ar[d] & S/K \ar[r] \ar@{=}[d] & 0 \\
0 \ar[r] & K \ar[r] \ar[d] & S \ar[r] \ar[d] & S/K \ar[r] & 0 \\
& S/I \ar@{=}[r] \ar[d] & S/I \ar[d] \\
& 0 & 0.
}}
$$
Because the left $S$-module $S/I$ is flat, the columns of the commutative diagram are pure exact. Now a pure submodule of a Mittag-Leffler (ML) module is itself ML and every pure projective module is ML, so it follows that the left modules $I$ and $J$ are ML. Let us develop another flat resolution of $S/K$ beginning with the short exact sequence in the top row of the pushout diagram.
\begin{equation} \label{ML}
\vcenter{%
\xymatrix@C=15pt@R=15pt
{%
0 \ar[rr] && S \ar[rrr]^{- \times g} &&& I \ar[drr] \ar[rrrr]^{- \times f}  &&&& I \ar[rrr]^{p} &&& S/K \ar[rr] && 0 \\
         &&                       &&&&&   J \ar[urr] \ar[drr] \\
         &&                       &&& 0 \ar[urr] &&&& 0.     
}}
\end{equation}
To see that the short sequence on the left is exact, first note that, because $g \in I,$ the image $Sg$ of the map $- \times g : S \to S$ is contained in $I$ and it is the kernel of $- \times f : I \to I,$ whose image is $If = ISf = IK = J.$ Resolution~(\ref{ML}) of
$S/K$ is a subcomplex of Resolution~(\ref{free}).

\begin{prop} \label{pure and flat}
The left $S$-module ${_S}I$ is absolutely pure and flat.
\end{prop}

\begin{proof}
We noted above that $\SAbs \cap \SFlat = (S/K)^{\perp_0} \cap (S/K)^{\perp_1} \cap (S/K)^{\perp_2} =
\F_K \cap (S/K)^{\perp_1} \cap \T_I.$ Evidently $I$ has no summand isomorphic to $S/K,$ so that $I \in (S/K)^{\perp_0}.$ On the other hand, $I^2 = I$ so that $I \in \T_I.$ Finally, to verify that $\Ext^1_S (S/K, I) = 0,$ consider the beginning of the long exact sequence
$$\vcenter{%
\xymatrix@1@C=25pt@R=15pt
{%
0 \ar[r] & (S/K,I) \ar[r] & (S/K,S) \ar[r] & (S/K, S/I) \ar[r] & \Ext^1 (S/K,I) \ar[r] & \Ext^1 (S/K, S) \ar[r] & \cdots.     
}}
$$
By Example~\ref{free module}, $\Ext^1 (S/K,S) = 0$ and the three terms on the left are also $0.$ 
\end{proof}

As ${_S}I$ is absolutely pure, its pure injective envelope coincides with its injective envelope, $\PE (I) = E(I),$ which is indecomposable. Therefore, Resolution~(\ref{ML}) of $S/K$ also serves as an absolutely pure coresolution of $S.$

\begin{prop}
The torsion class $\T_I = \Gen (I),$ the class of modules that arise as quotients of coproducts of copies of $I.$ Dually, the torsion free class $\F_K = \Cogen (E(S/I)),$ the class of those modules that admit an embedding into a product of copies of
$E(S/I).$ 
\end{prop}

\begin{proof}
Since $I \in \T_I$ and $\T_I$ is closed under coproducts and quotient objects, $\Gen (I) \subseteq \T_I.$ On the other hand, if $M = IM = gM$ and $m \in M,$ then $m = gm'$ for some element $m' \in M.$ If $\eta : S \to M$ is the morphism from the regular representation that satisfies $\eta (1) = m',$ then $\eta (g) = g\eta (1) = gm' = m.$ Thus $m \in \Img (\eta |_I)$ and $M \in \Gen (I).$

The torsion free class $\F_K$ consists of the modules $M$ that contain no submodule (resp., summand) isomorphic to the injective module $S/K.$ The injective envelope $E(S/I)$ of the flat simple module $S/I$ is clearly such a module. As $\F_K$ is closed under products and submodules, we have that $\Cogen (E(S/I)) \subseteq \F_K.$ Now every module admits an embedding into a product of copies of the minimal injective cogenerator $E(S/I) \oplus S/K,$ so if $M$ belongs to $\F_K$ then there is an embedding of $M$ into a product $(S/K)^{\alpha} \oplus [E(S/I)]^{\alpha}$ for some cardinal $\alpha.$ But the intersection of $M$ with the semisimple summand
$(S/K)^{\alpha} \simeq (S/K)^{(\beth)}$ is trivial and the embedding of $M$ composed with the projection onto $E(S/I)^{\alpha}$ is a monomorphism.     
\end{proof}

\begin{prop} \label{abs/flat minimal}
The definable subcategory $\SAbs \cap \SFlat$ is a minimal nonzero definable subcategory of $\Modl S.$
\end{prop}

\begin{proof}
Suppose that $\D \subseteq \SAbs \cap \SFlat$ is a nonzero definable subcategory, and pick a nonzero $S$-module $N \in \D.$ The module ${_S}N$ is not semisimple, so that the annihilator $\ann_S (N) = 0.$ For every $n \in N,$ there is a morphism $\eta_n : S \to N$ satisfying $\eta_n (1) = n.$ The product $\prod_n \eta_n : S \to N^N$ is therefore an embedding and yields an embedding $I \subseteq N^N.$ As $I$ is absolutely pure and the definable subcategory is closed under pure submodules, $I \in \D.$ Now if $N \in \SAbs \cap \SFlat,$ then it is absolutely pure so that 
$\Ext^1 (K,N) = 0$ which implies that $N \in \T_I = \Gen (I):$ there is an epimorphism $I^{(\alpha )} \to N.$ As $N$ is flat, the morphism is a pure epimorphism and therefore $N \in \D.$
\end{proof}

Proposition~\ref{abs/flat minimal} implies that $\Cl (\SAbs \cap \SFlat)$ is a minimal nonempty closed subset of the Ziegler spectrum. Because $E(I)$ belongs to $\Cl (\SAbs \cap \SFlat),$ every indecomposable injective flat module, considered as a point of
$\Zg (S),$ is topologically indistinguishable from $E(I).$ 

The preceding considerations allow us to exhibit twelve definable subcategories $\D.$ These twelve subcategories are distinguished by their intersections with the subset $\{ S/K, \PE (J), S/I, E(S/I) \}$ of the core Ziegler spectrum $\Zg_0 (S).$ In the sequel, we will define the {\em maximal} Ziegler spectrum $\MaxZg (S)$ and verify that it is given by this subset, endowed with the relative subspace topology. \bigskip
\begin{center} {\bf Table 1: The Definable Subcategories of $\Modl S$} 
\bigskip

\begin{tabular}{c|c|c}
$\D$ & $\D \cap \MaxZg (S)$ & $\D \cap \SPProj$ \\ \hline
$0$ & $\emptyset$ & $0$ \\
$\Add (S/I)$ & $\{ S/I \}$ & $0$ \\
$\Add (S/K)$ & $\{ S/K \}$ & $\Add (S/K)$ \\
$\Add (S/J) = \Add (S/K \oplus S/I)$ &  $\{ S/K, S/I \}$ & $\Add (S/K)$ \\
$\SAbs \cap \SFlat = (S/K)^{\perp_0} \cap (S/K)^{\perp_1} \cap (S/K)^{\perp_2}$ &  $\{ E(S/I) \}$ & $\Add (I_{\omega})$ \\
$\SAbs = (S/K)^{\perp_1} \cap K^{\perp_1} = (S/K)^{\perp_1} \cap (S/K)^{\perp_2}$ &  $\{ S/K, E(S/I) \}$ & $\Add (I_{\omega} \oplus S/K)$ \\
$\SFlat = (S/K)^{\perp_0} \cap (S/K)^{\perp_1}$ & $\{ S/I, E(S/I) \}$ & $\Add(S) = \SProj$ \\
$S/K^{\perp_0} \cap (S/K)^{\perp_2}$ & $\{ \PE(J), E(S/I) \}$ & $\Add (I_{\omega} \oplus J_{\omega})$ \\
$\F_K = \Cogen (E(S/I)) = (S/K)^{\perp_0}$ & $\{ \PE(J), S/I, E(S/I) \}$ & $\Add (S \oplus K)$ \\
$(S/K)^{\perp_1}$ & $\{ S/K, S/I, E(S/I) \}$ & $\Add (S \oplus S/K)$ \\
$\T_I = \Gen (I) = (S/K)^{\perp_2}$ & $\{ S/K, \PE(J), E(S/I) \}$ & $\Add(I_{\omega} \oplus J_{\omega} \oplus S/K)$ \\
$\Modl S$ & $\MaxZg (S)$ & $\Add(\modl S)$ \\ 
\end{tabular}
\end{center}
\bigskip

\noindent The first four do not contain the minimal definable subcategory $\SAbs \cap \SFlat;$ the others do. We can verify that the latter $8$ definable subcategories are pairwise distinct by looking at the three modules $S/K,$ $J,$ and $S/I,$ and recalling that a definable subcategory $\D$ of $\Modl S$ contains $J$ if and only if it contains its pure injective envelope $\PE (J).$ The module $S/K$ belongs to $\SAbs = (S/K)^{\perp_1} \cap (S/K)^{\perp_2},$ but not to $(S/K)^{\perp_0};$ the module $S/I$ belongs to
$\SFlat = (S/K)^{\perp_0} \cap (S/K)^{\perp_1},$ but not to $(S/K)^{\perp_2}.$ Let us note that $J$ belongs to $(S/K)^{\perp_0}$ and 
$(S/K)^{\perp_2} = \T_I,$ but not to $(S/K)^{\perp_1}.$ Well, $J$ contains no submodule isomorphic to $S/K$ and $IJ = J,$ so the first two hold. On the other hand, Resolution~(\ref{ML}) tells us that $\Ext^1 (S/K, J) \neq 0.$ Once we have proved Theorem~\ref{Zg spec}, which implies that there are exactly 12 closed subsets of $\Zg (S),$ we will know that this list of definable subcategories of $\Modl S$ is exhaustive. \bigskip

\noindent{\bf e. Three tilting classes in $\Modl S.$}\bigskip

In this section we identify three tilting classes in $\Modl S$ by building a projective resolution of $S/K$ that is a countably generated subcomplex of Resolution~(\ref{ML}).
\begin{equation}  \label{projective}
\vcenter{%
\xymatrix@C=15pt@R=15pt
{%
0 \ar[rr] && S \ar[rrr]^{- \times g} &&& I_{\omega} \ar[drr] \ar[rrrr]^{- \times f}  &&&& I_{\omega} \ar[rrr]^{p_{\omega}} &&& S/K \ar[rr] && 0 \\
         &&                       &&&&&   J_{\omega} \ar[urr] \ar[drr] \\
         &&                       &&& 0 \ar[urr] &&&& 0     
}}
\end{equation}
Here $I_{\omega} \subseteq I$ is a countably generated left $S$-module that is absolutely pure and projective, and $J_{\omega} \subseteq J$ is a countably generated pure projective module. Because $S/K$ is also pure projective, Resolution~(\ref{projective}) is a special absolutely pure coresolution of $S;$ it will be used to study the tilting classes of $\Modl S.$ We will verify that the module $T_1 = I_{\omega} \oplus J_{\omega}$ is a $1$-tilting module and that $T_2 = I_{\omega} \oplus S/K$ is a $2$-tilting module.

\begin{lem}
If $I' \subseteq {_S}I$ is a submodule with $g \in I',$ then $I' = Sg + (I' \cap J).$ If $I'$ is, moreover, a proper submodule of $I,$ then there is an element $x \in J$ such that $I' \subseteq Sg + Sx.$ 
\end{lem}

\begin{proof}
For the first statement, recall that $I = Sg + J,$ so that $I' = I' \cap I = I' \cap (Sg + J) = Sg + (I' \cap J),$ because $Sg \subseteq I'.$ It follows that if $I'$ is a proper submodule of $I,$ then $I' \cap J$ is a proper submodule of $J.$ As ${_S}J$ is uniserial, there is an $x \in J$ such that $Sx \supseteq I' \cap J,$ whence $I' = Sg + (I' \cap J) \subseteq Sg + Sx.$ 
\end{proof}

\begin{thm} \label{countably generated}
There exists a countably generated pure submodule $I_{\omega}$ of ${_S}I$ such that:
\begin{enumerate} 
\item $Sg \subseteq I_{\omega};$ 
\item $I_{\omega}f = J_{\omega},$ where $J_{\omega} = I_{\omega} \cap J;$ and 
\item $J_{\omega}f = J_{\omega}.$
\end{enumerate}
\end{thm}

\begin{proof}
If ${_S}J$ is countably generated, then so is $I = Sg + J,$ and we may let $I_{\omega} = I,$ for then $If = ISf = IK = J = I \cap J,$ and $Jf = JSf = JK = J.$ So assume that $J$ is not countably generated and define, by recursion on $n,$ two sequences $\langle j_n \; | \; n < \omega \rangle$ and $\langle j'_n \; | \; n < \omega \rangle$ of elements of $J$ together with an ascending chain 
$$Sg \subseteq I_0 \subseteq I_1 \subseteq \cdots \subseteq I_n \subseteq \cdots \subseteq I$$
of countably generated pure submodules of $I$ as follows:
\begin{enumerate}
\item $j_0 = 0;$
\item if $j_n$ is defined, let $j'_n \in J$ be some element for which $j'_nf = j_n;$
\item if $j_n$ and $j'_n$ are defined, let $I_n \supseteq Sg + Sj_n + Sj'_n$ be a countably generated pure submodule of $I;$  
\item if $I_n$ is defined, let $j_{n+1} \in J$ be such that $I_n + I_nf \subseteq Sg + Sj_{n+1}.$
\end{enumerate}
The existence of $j'_n \in J$ in ($2$) is ensured by the fact that $Jf = J;$ the existence of $I_n$ in (3) is ensured by the fact that ${_S}I$ is Mittag-Leffler and the property of a Mittag-Leffler module that every finitely generated submodule is contained in a countably generated pure submodule; and the existence of the $j_{n+1}$ in (4) is ensured by the lemma: the submodule $I_n + I_nf$ is countably generated and therefore proper in $I.$ 

Let $I_{\omega} := \sum_n I_n.$ Then $I_{\omega}$ is a countably generated pure submodule of $I$ that contains $Sg$ and it remains to verify that $I_{\omega}f = J_{\omega}.$ The inclusion $I_{\omega}f \subseteq J_{\omega}$ is easy to show. By construction, we have that $I_nf \subseteq I_{n+1},$ so that $I_{\omega}f \subseteq I_{\omega}$ and $I_{\omega}f \subseteq If = ISf = IK = J.$ Thus $I_{\omega}f \subseteq I_{\omega} \cap J = J_{\omega}.$ In particular, $J_{\omega}f \subseteq J_{\omega}.$ To prove the converse inclusion $J_{\omega} \subseteq I_{\omega}f,$ it suffices to verify that $J_{\omega} \subseteq J_{\omega}f,$ which will also yield the equality $J_{\omega}f = J_{\omega}.$ 

Before proving that $J_{\omega} \subseteq J_{\omega}f,$ let us make a few observations regarding the construction of $I_{\omega}.$ First note that the pure submodule $I_0$ of $I$ contains $Sg$ properly. If it weren't so, then $Sg \subseteq I$ would be a pure submodule. As ${_S}I$ is pure in ${_S}S,$ then the embedding $Sg \subseteq S$ would be a pure, hence split, monomorphism, contradicting the fact that ${_S}S$ is uniform. 

The containment $Sg + Sj_n \supseteq Sg$ is therefore proper for every $n > 0.$ This implies that $Sj_n \supseteq Sg \cap J$ - these are submodules of ${_S}J,$ which is uniserial. As the $j_n \in J,$ we get that $(Sg + Sj_n) \cap J = (Sg \cap J) + Sj_n = Sj_n.$
Taking the limit, one obtains $$J_{\omega} = I_{\omega} \cap J = \sum_n (Sg + Sj_n) \cap J = \sum_n Sj_n.$$ 
By ($2$) of the construction of $I_{\omega},$ we have that $Sj_n = Sj'_nf$ and so $J_{\omega} = \sum_n Sj_n = \sum_n Sj'_nf \subseteq J_{\omega}f.$
\end{proof}

As $I_{\omega} \subseteq I$ is a pure submodule, Proposition~\ref{pure and flat} implies that $I_{\omega}$ is absolutely pure and flat. Using the notation $J_{\omega} = I_{\omega} \cap J$ introduced in Theorem~\ref{countably generated}, we have the pushout diagram
$$
\vcenter{%
\xymatrix@C=35pt@R=35pt
{%
& 0 \ar[d] & 0 \ar[d] \\ 
0 \ar[r] & J_{\omega} \ar[r] \ar[d] & I_{\omega} \ar[r]^{p_{\omega}} \ar[d] & S/K \ar[r] \ar@{=}[d] & 0 \\
0 \ar[r] & J \ar[r] \ar[d] & I \ar[r] \ar[d] & S/K \ar[r] & 0 \\
& I/I_{\omega} \ar@{=}[r] \ar[d] & I/I_{\omega} \ar[d] \\
& 0 & 0.
}}
$$
with exact rows and columns. The middle column is a pure exact sequence, which implies that the quotient module $I/I_{\omega}$ is, like $I,$ a flat left $S$-module. The left column is then also a pure exact sequence and we see that the countably generated submodule $J_{\omega} \subseteq J$ is also pure. Now every countably generated Mittag-Leffler module is pure projective, so that both of the modules $I_{\omega}$ and $J_{\omega}$ are pure projective. The module $I_{\omega}$ is even a flat pure projective, and is therefore projective.  Theorem~\ref{countably generated} includes the equality $I_{\omega}f = J_{\omega},$ which allows us to complete Resolution~(\ref{projective}) with the projective resolution of $J_{\omega}$ given by
$$
\vcenter{%
\xymatrix@C=45pt@R=15pt
{%
0 \ar[r] & S \ar[r]^{- \times g} & I_{\omega} \ar[r]^{- \times f} & J_{\omega} \ar[r] & 0
}}
$$
Both of the modules $J_{\omega}$ and $S/K$ are pure projective and therefore belong to ${^{\perp_1}}\SAbs.$ Resolution~(\ref{projective}) of $S/K$ is thus a {\em special} absolutely pure coresolution of $S.$ \bigskip

\begin{cor} \label{tilt1}
The torsion class $\T_I$ is a $1$-tilting class with a tilting module $T_1 = I_{\omega} \oplus J_{\omega}$ that is pure projective, but not classical. 
\end{cor} 

\begin{proof}
Let us verify for $T_1 = I_{\omega} \oplus J_{\omega}$ the three conditions for a $1$-tilting module. (1) Resolution~(\ref{projective}) of $S/K$ shows that $J_{\omega}$ has projective dimension at most $1.$ As $I_{\omega}$ is projective, we obtain that the projective dimension of $T_1$ is at most $1.$ (2) It follows from (1) that 
$$T_1^{\perp} = T_1^{\perp_1} = (I_{\omega} \oplus J_{\omega})^{\perp_1} = J_{\omega}^{\perp_1} = (S/K)^{\perp_2} = \T_I.$$
The equality $J_{\omega}^{\perp_1} = (S/K)^{\perp_2}$ follows from the way that $J_{\omega}$ is situated in Resolution~(\ref{projective}). Neither $I_{\omega}$ nor $J_{\omega}$ admit nonzero morphisms to the simple module $S/I,$ so we have that $T_1 \in \T_I$ and hence $T_1^{(\alpha)} \in \T_I.$ (3) Resolution~(\ref{projective}) contains a coresolution of $S$ in $\Add (I_{\omega} \oplus J_{\omega}).$

To see that $T_1$ is not $\Add$-equivalent to a finitely presented module use Fact~\ref{fp modules} to see that the only finitely presented module that belongs to $\T_I$ is $S/K,$ and it is clear that $T_1$ does not belong to $\Add (S/K).$   
\end{proof} 

If $T_n$ is an $n$-tilting module, then the associated {\em tilting class} is defined to be the subcategory
$$T_n^{\perp_{\infty}} = \{ M \in \Modl S \; | \; \Ext^i (T_n, M) = 0, \; \mbox{for all} \; i > 0  \}.$$
Two $n$-tilting modules $T_n$ and $T_n'$ are {\em tilting equivalent} if $\Add (T_n) = \Add (T_n').$

\begin{cor} \label{tilt2}
The definable subcategory $\SAbs \subseteq \Modl S$ is a $2$-tilting class with a pure projective $2$-tilting module $T_2 = I_{\omega} \oplus S/K$ that is not tilting equivalent to any finitely presented module.
\end{cor}

\begin{proof}
Let us verify for $T_2 = I_{\omega} \oplus S/K$ the three conditions for a $2$-tilting module. (1) Resolution~(\ref{projective}) shows that $S/K$ has projective dimension at most $2.$ As $I_{\omega}$ is projective, the projective dimension of $T_2$ is at most $2.$ (2) It follows from (1) that
$$T_2^{\perp_{\infty}} = (S/K)^{\perp_{\infty}} = (S/K)^{\perp_1} \cap (S/K)^{\perp_2} = \SAbs.$$
As $T_2$ is absolutely pure, it belongs to $\SAbs.$ The coproduct $T_2^{(\alpha)}$ is a pure submodule of the product $T_2^{\alpha},$ and so too belongs to the definable subcategory $\SAbs.$ (3) Resolution~(\ref{projective}) is a coresolution of $S$ in
$\Add (I_{\omega} \oplus S/K).$

To see that $T_2$ is not $\Add$-equivalent to a finitely presented module note, as above, that the only absolutely pure finitely presented indecomposable module is $S/K$ and that $T_2$ does not belong to $\Add (S/K).$  
\end{proof}

The list of $12$ definable subcategories compiled in Table 1 includes only $4$ that are eligible to be a tilting class, because a tilting class must contain all the injective modules and therefore the subcategory $\SAbs.$ The definable subcategories on the list that contain $\SAbs$ are: $\SAbs$ itself, which is a $2$-tilting class; $\T_I = (S/K)^{\perp_2} = K^{\perp_1},$ which is a $1$-tilting class; $\Modl S = S^{\perp_1},$ which is the $0$-tilting class associated to the $0$-tilting module given by the regular representation ${_S}S;$ and the definable subcategory $(S/K)^{\perp_1}.$ \bigskip

\noindent {\bf f. A Coresolution of $S/I.$} \bigskip

In this and the following sections, the flat resolutions of $S/K$ that are constructed are also absolutely pure coresolutions of the flat simple representation $S/I.$ The first two such coresolutions of $S/I$ are obtained from Resolutions~(\ref{ML}) and~(\ref{projective}), respectively, by taking the pushout along the quotient map $\pi_I : S \to S/I.$ Our first such resolution of $S/K$ is preliminary to the last construction of the minimal flat resolution in Theorem~\ref{minimal cores}.
\begin{equation} \label{cores}
\vcenter{%
\xymatrix@C=15pt@R=15pt
{%
0 \ar[rr] && S/I \ar[rrr]^{p} &&& I/Ig \ar[drr] \ar[rrrr]^{- \times f}  &&&& I \ar[rrr]^{p} &&& S/K \ar[rr] && 0 \\
         &&                       &&&&&   J \ar[urr] \ar[drr] \\
         &&                       &&& 0 \ar[urr] &&&& 0     
}}
\end{equation}
To see how this is formed from Resolution~(\ref{ML}), consider the pushout
$$
\vcenter{
\xymatrix@C=35pt@R=35pt
{%
& 0 \ar[d] & 0 \ar[d] \\ 
& I \ar[d] \ar@{=}[r] & I \ar[d]^{- \times g} \\
0 \ar[r] & S \ar[r]^{- \times g} \ar[d]^{\pi_I} & I \ar[r]^{- \times f} \ar[d] & J \ar[r] \ar@{=}[d] & 0 \\
0 \ar[r] & S/I \ar[r]^{s} \ar[d] & I/Ig \ar[r] \ar[d] & J \ar[r] & 0 \\
& 0 & 0
}}
$$
and replace the appearance of the middle row in Resolution~(\ref{ML}) with the bottom row. 

\begin{prop} \label{I/Ig}
The left $S$-module $I/Ig$ is an absolutely pure uniserial module with simple socle $S/I.$
\end{prop}
 
\begin{proof}
Look at the pushout diagram above. The middle column is pure exact, because ${_S}I$ is absolutely pure. This implies that $I/Ig$ is a pure epimorphic quotient of the module $I \in \SAbs,$ which therefore also belongs to $\SAbs.$ The reason why $I/Ig = (Ig + J)/Ig \simeq J/(J \cap Ig)$ is uniserial is because it is a quotient of the uniserial module ${_S}J.$
\end{proof}

We can see now that the definable subcategory $(S/K)^{\perp_1}$ cannot be a tilting class $T^{\perp}$ for any module $T.$ This is because every tilting class is closed under cokernels of monomorphism and we have the short exact sequence
$$
\vcenter{
\xymatrix@1@C=35pt@R=35pt
{%
0 \ar[r] & S/I \ar[r]^{s} & I/Ig \ar[r] & J \ar[r] & 0.
}}
$$
The module $S/I$ belongs to $(S/K)^{\perp_1},$ because the semisimple category $\Add (S/J)$ is closed under extensions, and the absolutely pure module $I/Ig$ also belongs to $S/K^{\perp_1},$ but $J$ clearly does not.\bigskip

\noindent {\bf g. Auslander-Reiten theory in $\Modl S.$}\bigskip

In this section we consider the Auslander-Reiten theory of $\Modl S$ by applying the method of the previous section to Resolution~(\ref{projective}). Namely, we take the pushout along $\pi_I : S \to S/I$ to obtain the following. 
\begin{equation} \label{special cores}
\vcenter{%
\xymatrix@C=15pt@R=15pt
{%
0 \ar[rr] && S/I \ar[rrr]^{s_{\omega}} &&& I_{\omega}/Ig \ar[drr] \ar[rrrr]^{- \times f}  &&&& I_{\omega} \ar[rrr]^{p_{\omega}} &&& S/K \ar[rr] && 0 \\
         &&                       &&&&&   J_{\omega} \ar[urr] \ar[drr] \\
         &&                       &&& 0 \ar[urr] &&&& 0     
}}
\end{equation}
Resolution~(\ref{special cores}) of $S/K$ has the feature that it is also a special absolutely pure coresolution of $S/I,$ because the cosyzygies $J_{\omega}$ and $S/K$ are pure projective and therefore belong to ${^{\perp_1}}(\SAbs).$ We will note that the epimorphism $p_{\omega} : I_{\omega} \to S/K$ is a right almost split morphism in the category $\Add (\modl S)$ of pure projective left $S$-modules, while the monomorphism $i_{\omega} : S/I \to I_{\omega}/Ig$ is a left almost split morphism in the category $\Modl S$ of all left $S$-modules. Neither of these almost split morphisms is minimal. \bigskip

A morphism $r : U \to V$ in some subcategory $\C \subseteq \Modl S$ is {\em right almost split in} $\C$ if it is not a split epimorphism and every morphism $q : W \to V$ in $\C$ that is not a split epimorphism, factors through $r$ as indicated by the dotted arrow
$$
\vcenter{%
\xymatrix@C=45pt@R=35pt
{%
& W \ar[d]^{q} \ar@{.>}[dl] \\
U \ar[r]^{r} & V
}}
$$
A right almost split morphism $r : U \to V$ is {\em minimal} if the only endomorphisms $t : U \to U$ that satisfy $r = rt$ are automorphisms of $V.$ {\em (Minimal) left almost split} morphisms in $\C$ are defined dually.

\begin{prop} \label{right almost split}
The epimorphism $p_{\omega} : I_{\omega} \to S/K$ in the projective resolution~(\ref{projective}) of $S/K$ is a right almost split morphism in the category 
$\Add (\modl S)$ of pure projective left $S$-modules. However, there does not exist a minimal right almost split morphism in $\Add (\modl S)$ with codomain $S/K.$
\end{prop}

\begin{proof}
To prove that the morphism $p_{\omega} : I_{\omega} \to S/K$ is right almost split in $\Add (\modl S),$ it suffices to verify the condition for morphisms with domain $S,$ $K,$ or $S/K.$ As $p_{\omega}$ is an epimorphism, this follows trivially for the projective module $S;$ for the module $K,$ it follows from the observation that $K/J \simeq S/I$ so that there exist no nonzero morphisms $\eta :K \to S/K;$ finally, $S/K$ is simple, so the only nonzero morphism $\eta: S/K \to S/K$ is an isomorphism, and therefore a split epimorphism.

If there existed a minimal right almost split morphism $\eta: U \to S/K$ in $\Add (\modl S),$ then the theory of almost split sequences would imply that $U$ is isomorphic to a direct summand of $I_{\omega}.$ As $I_{\omega}$ is indecomposable, it would imply that $p_{\omega} : I_{\omega} \to S/K$ is the minimal right almost split morphism in $\Add (\modl S)$ with codomain $S/K.$ The theory of almost split sequences would then imply that the short exact sequence  $$
\vcenter{%
\xymatrix@C=45pt@R=15pt
{%
0 \ar[r] & J_{\omega} \ar[r] & I_{\omega} \ar[r]^{p_{\omega}} & S/K \ar[r] & 0
}}
$$
is almost split. But then the endomorphism ring of $J_{\omega}$ would have to be local, which it is not, because every pure projective module with a local endomorphism ring is finitely generated, and $J_{\omega}$ is not.
\end{proof}
 
To build the rest of Resolution~(\ref{special cores}), let us take the pushout of the projective resolution of $J_{\omega}$ found in Resolution~(\ref{projective}) along the quotient map $\pi_I : S \to S/I.$ This leads to the commutative diagram
$$
\vcenter{
\xymatrix@C=35pt@R=35pt
{%
& 0 \ar[d] & 0 \ar[d] \\ 
& I \ar[d] \ar@{=}[r] & I \ar[d]^{- \times g} \\
0 \ar[r] & S \ar[r]^{- \times g} \ar[d]^{\pi_I} & I_{\omega} \ar[r]^{- \times f} \ar[d] & J_{\omega} \ar[r] \ar@{=}[d] & 0 \\
0 \ar[r] & S/I \ar[r]^{s_{\omega}} \ar[d] & I_{\omega}/Ig \ar[r] \ar[d] & J_{\omega} \ar[r] & 0 \\
& 0 & 0.
}}
$$

\begin{prop} \label{P:left asm}
The module $I_{\omega}/Ig$ is an absolutely pure flat uniserial module with simple socle $S/I.$ The morphism $s_{\omega} : S/I \to I_{\omega}/Ig$ is a left almost split morphism in the category $\Modl S$ of all left $S$-modules. However, there exists no minimal left almost split morphism in $\Modl S$ with domain $S/I.$
\end{prop}

\begin{proof}
Because $I$ is an absolutely pure module, the middle column of the diagram above is pure exact. As $I_{\omega}$ is an absolutely pure flat module, it follows that the pure quotient $I_{\omega}/Ig$ is also absolutely pure and flat. The module $I_{\omega}/Ig$ is uniserial, because it is a submodule of the uniserial module $I/Ig.$

To prove that $s_{\omega} : S/I \to I_{\omega}/Ig$ is a left almost split morphism in $\Modl S,$ let $s' : S/I \to M$ be a morphism in
$\Modl S$ that is not a split monomorphism. By the pushout property of $s_{\omega},$ it will be enough to verify that the composition of the vertical arrows in the diagram
$$
\vcenter{%
\xymatrix@C=35pt@R=35pt
{%
S \ar[r]^{- \times g} \ar[d]_{\pi_I} & I_{\omega} \ar@{.>}[ldd] \\
S/I \ar[d]_{s'} \\ M
}}
$$
factors through the horizontal arrow as indicated. If the image of $s'$ is not contained in $IM,$ then the composition
$S/I \stackrel{s'}{\rightarrow} M \to M/IM$ with the quotient map is nonzero. The module $M/IM$ is semisimple, which implies the contradiction that this composition, and hence $s,$ is a split monomorphism. The image of $s'$ is therefore contained in $IM = gM,$ so let $a \in M$ be such that $s'(1 + I) = ga.$ There is a morphism $\eta_a : S \to M$ that sends $1$ to $a.$ It sends $g$ to $ga,$ so that the restriction $\eta_a|_{I_{\omega}}$ to $I_{\omega}$ makes the diagram commute, as required. 

To see that there exists no minimal left almost split morphism in $\Modl S$ with domain $S/I,$ let us argue as above. If such a morphism $q : S/I \to U$ existed, the module $U$ would be a direct summand of the indecomposable module $I_{\omega}/Ig.$ It would follow that $s_{\omega} : S/I \to I_{\omega}/Ig$ is necessarily the minimal left almost split morphism in $\Modl S$ with domain $S/I.$ But then the short exact sequence
$$
\vcenter{%
\xymatrix@1@C=55pt@R=25pt
{%
0 \ar[r] & S/I \ar[r]^{s_{\omega}} & I_{\omega}/Ig \ar[r] & J_{\omega} \ar[r] & 0
}}
$$
is almost split, which implies the contradiction that the endomorphism ring of $J_{\omega}$ is local.
\end{proof}

The fourth author~\cite{JanIM} has proved that if a morphism $r : U \to V$ is a right almost split morphism in the category $\Modl S$ of all modules, then the codomain $V$ is necessarily a finitely presented module with a local endomorphism ring. A result of Auslander then implies that there exists a minimal right almost split morphism  in $\Modl S$ with codomain $V.$ Proposition~\ref{P:left asm} gives a counterexample to the dual statement. \bigskip

\noindent {\bf h.\ The four simple objects of the functor category.}\bigskip

In this section, we finally determine the minimal flat resolution of $S/K.$

\begin{thm} \label{minimal cores}
The minimal flat resolution of $S/K$ is given by the minimal injective resolution of $S/I:$
$$\vcenter{%
\xymatrix@C=10pt@R=10pt
{%
0 \ar[rr] && S/I \ar[rrr]^{i} &&& E(S/I) \ar[drr]_{k} \ar[rrrr]  &&&& E(I) \ar[rrr] &&& S/K \ar[rr]^{0} && 0 \\
         &&                       &&&&&   \PE (J) \ar[urr]_{j} \ar[drr] \\
         &&                       &&& 0 \ar[urr] &&&& 0.     
}}$$
\end{thm}

Notice how the five indecomposable pure injective indecomposable modules of the core Ziegler spectrum $\Zg_0 (S)$ figure in this resolution. The minimal flat resolution of $S/K$ has features dual to Resolution~(\ref{special cores}) of $S/K,$ in that the epimorphism $E(I) \to S/K$ is a right almost split morphism in the category $\Modl S,$ while the monomorphism $i : S/I \to E(S/I)$ is a left almost split morphism in the category $\SPInj$ of pure injective left $S$-modules. In fact, the morphism
$i: S/I \to E(S/I)$ is a {\em minimal} left almost split morphism in $\SPInj$ and we will show how the other labeled arrows in the diagram represent {\em all} the minimal left almost split morphisms in $\SPInj.$ These four morphisms correspond to the four simple objects in the functor category $(\modr S, \Ab),$ and their domains correspond to the four points of the maximal Ziegler spectrum
$\MaxZg (S)$ of $S.$  \bigskip

The objects of the category $(\modr S, \Ab)$ are the additive functors $F: \modr S \to \Ab$ and the morphisms are given by the natural transformations between such functors. The functor category $(\modr S, \Ab)$ is a Grothendieck category and there is a full and faithful canonical functor $\Modl S \to (\modr S, \Ab),$ \linebreak ${_S}M \to - \otimes_S M.$ It is right exact and the higher (left) derived functors of ${_S}M \mapsto - \otimes_S M$ are given by ${_S}M \mapsto \Tor_i^S (-,M).$ For example, the free resolution of $S/K$ yields the complex in $(\modr S, \Ab)$
$$
\vcenter{%
\xymatrix@C=45pt@R=15pt
{%
0 \ar[r] & - \otimes S \ar[r]^{- \otimes (- \times g)} & - \otimes S \ar[r]^{- \otimes (- \times f)} & - \otimes S \ar[r] & 0, 
}}
$$
whose homology is given by $- \otimes_S S/K,$ $\Tor_1^S (-, S/K),$ and $\Tor_2^S (-,S/K).$ A key feature of the functor category is that the finitely presented objects, also known as {\em coherent functors,} form an abelian category $\fp (\modr S, \Ab)$ and that the canonical functor 
$M \to - \otimes_S M$ respects the property of being finitely presented. This means that if $M \in \modl S$ then the functor $- \otimes_S M$ is coherent. It follows that the objects in the complex above are coherent functors, and therefore, that the homology also belongs to $\fp (\modr S, \Ab).$ In this section, we will show how this homology yields three of the four simple objects of 
$\fp (\modr S, \Ab).$ 

The injective objects of $(\modr S, \Ab)$ are the functors isomorphic to those of the form $- \otimes_S N,$ where $N \in \SPInj.$ If $m: M \to \PE(M)$ is the pure injective envelope of $M \in \Modl S,$ then the induced morphism $- \otimes m : - \otimes_S M \to - \otimes_S \PE(M)$ is the injective envelope in $(\modr S, \Ab)$ of $- \otimes_S M.$ For every left $S$-module ${_S}N,$ there exists a natural isomorphism $$\Hom_{(\modr S, \Ab)}[\Tor_i (-,S/K), - \otimes_S N] \simeq \Ext_S^i (S/K,N),$$
for $i = 0,$ $1$ and $2,$ where $\Tor_0 (-,S/K) = -\otimes_S S/K$ is the $0$-th right derived functor of itself, and similarly 
$\Ext^0 (S/K,-)$ denotes $\Hom_S(S/K,-).$ The foregoing considerations allow us to rephrase Proposition~\ref{abs and flat} using functorial terminology.

\begin{prop} \label{isolated}
A left $S$-module $M$ is absolutely pure and flat if and only
$$\Hom_{(\modr S, \Ab)}[\Tor_i (-,S/K), - \otimes_S M] = 0,$$
for $i = 0,$ $1,$ and $2.$
\end{prop}

The points $U$ of the Ziegler spectrum $\Zg (S)$ thus correspond to the indecomposable injective objects $- \otimes_S U$ of the functor category $(\modr S, \Ab).$ Among the points of the Ziegler spectrum are those indecomposable pure injective modules $V$ distinguished by the property that the functor $- \otimes_S V = E(Z)$ is the injective envelope of a simple object $Z$ of 
$(\modr S, \Ab).$ These are the points of the {\em maximal} Ziegler spectrum of $S,$ which is denoted by $\MaxZg (S)$ and endowed with the relative subspace topology inherited from $\Zg (S).$ The following theorem shows that 
$\MaxZg (S) = \{ S/K, \PE (J), S/I, E(S/I) \}$ consists of all of the points of the core Ziegler spectrum, save $E(I).$

\begin{thm} \label{Max Spec}
There are $4$ simple objects of the functor category $(\modr S, \Ab).$ They and their injective envelopes are given as follows:
\begin{enumerate}
\item $E(- \otimes_S S/K) = - \otimes_S S/K;$
\item $E(\Tor_1 (-, S/K)) = - \otimes_S \PE (J);$
\item $E(\Tor_2 (-, S/K)) = - \otimes_S S/I;$ and
\item $E[(- \otimes_S S/I)/\Tor_2 (-, S/K)] = - \otimes_S E(S/I).$
\end{enumerate}
\end{thm}

To prove the theorem, we will treat each case separately. There exists a bijective correspondence~\cite[]{CB} between the simple objects $Z$ of $(\modr S, \Ab)$ and minimal left almost split morphisms $m_Z : V \to M$ in the category $\SPInj$ of pure injective left $S$-modules. If $Z \in (\modr S, \Ab)$ is a simple object, the minimal injective copresentation of $Z$ in $(\modr S, \Ab)$ has the form
$$
\vcenter{%
\xymatrix@1@C=45pt@R=25pt
{%
0 \ar[r] & Z \ar[r] & - \otimes_S V \ar[r]^{- \otimes m_Z}  & - \otimes_S M,
}}
$$
where $m_Z : V \to M$ is a minimal left almost split morphism in the category $\SPInj$ of pure injective left $S$-modules. Thus, a point $V$ in $\Zg (S)$ belongs to $\MaxZg (S)$ if and only if it is the domain of a minimal left almost split morphism in $\SPInj.$ \bigskip

\noindent (1) $- \otimes S/K.$ The easiest minimal left almost split morphism to spot is the zero morphism $0 : S/K \to 0$ with domain $S/K.$ It is obviously left almost split, because $S/K$ is injective and simple; the kernel of the induced zero morphism
$- \otimes 0 : - \otimes S/K \to - \otimes 0$ is the simple coherent functor $- \otimes S/K.$ \bigskip

To find another minimal left almost split morphism, we recall the result of Auslander that every finitely presented nonprojective left $S$-module with a local endomorphism ring lies at the right end of an almost split sequence in $\Modl S.$ This applies to $S/K$ and yields the almost split sequence
\begin{equation} \label{ass}
\vcenter{%
\xymatrix@1@C=55pt@R=25pt
{%
0 \ar[r] &  \tau^{-1}(S/K) \ar[r]^-{j} & Y \ar[r]^{p'}  & S/K \ar[r] & 0
}}
\end{equation}
in $\Modl S.$ Equivalently, $p': Y \to S/K$ is a minimal right almost split morphism and $j: \tau^{-1}(S/K) \to Y$ is a minimal left almost split morphism.
The sequence is not split exact and $S/K$ is pure projective, so that the short exact sequence is not pure exact. This implies~\cite{IvoAR} that the module $\tau^{-1}(S/K)$ is a pure injective indecomposable. 

The middle term $Y$ of the almost split sequence~(\ref{ass}) belongs to the torsion free class $\F_K.$ For, if $S/K$ were a submodule of $Y,$ it would be a direct summand, so it cannot be contained in the kernel of $p'.$ But then the restriction $p'|_{S/K}$ would be nonzero and the epimorphism $p':Y \to S/K$ would split, contradicting the assumption on the sequence~(\ref{ass}). 

The fact that $S/K$ is injective implies that the right almost split map $p' : Y \to S/K$ factors through the injective envelope
$p' : Y \to E(Y) \stackrel{p''}{\to} S/K.$ The morphism $p'' : E(Y) \to S/K$ is not a split epimorphism, so we can deduce from the minimality of $p'$ that $Y$ is a direct summand of $E(Y),$ and therefore that $Y = E(Y)$ and $p'' = p'.$ 

Now every injective module in $\F_K = \Cogen (E(S/I))$ is a direct summand of a product of copies of $E(S/I)$ and is thus flat. This implies that the map $p' : Y \to S/K$ is a special flat precover of $S/K,$ because the kernel $\tau^{-1}(S/K)$ is pure injective and therefore cotorsion. The minimality property of $p'$ implies that it is the flat cover of $S/K.$ By~\cite[Corollary 23]{GH}, the flat cover of a simple module is indecomposable, whenever the ring, such as $S,$ is semilocal. We conclude that the middle term $Y$ of the almost split sequence is an indecomposable flat injective module.
\bigskip

\noindent (2) $\Tor_1 (-, S/K).$ The morphism $j: \tau^{-1}(S/K) \to Y$ is a minimal left almost split morphism in $\Modl S$ and, because it is a morphism in $\SPInj,$ it is a mininal left almost split morphism in $\SPInj.$ The corresponding simple object in the functor category is found by looking at the first few terms of the long exact sequence
$$
\vcenter{%
\xymatrix@1@C=25pt@R=25pt
{%
{\cdots} \ar[r] & 0 \ar[r] & \Tor_1 (-,S/K) \ar[r] & - \otimes_S \tau^{-1}(S/K) \ar[r]^-{- \otimes r} & - \otimes Y \ar[r]^-{- \otimes p'}  & - \otimes S/K \ar[r] & 0.
}}
$$
Because $Y$ is a flat module, $\Tor_1 (-,Y) = 0$ and we see that the simple functor in question is given by $\Tor_1 (-,S/K).$ Furthermore, the indecomposable injective functor $- \otimes \tau^{-1}(S/K)$ is the injective envelope of $\Tor_1 (-,S/K).$ \bigskip

By the definition of an almost split sequence, there is a morphism of short exact sequences
$$
\vcenter{%
\xymatrix@C=35pt@R=25pt
{%
0 \ar[r] & J \ar[d]^{j_0} \ar[r] & I \ar[d]^{i_0} \ar[r]^{p}  & S/K \ar@{=}[d] \ar[r] & 0 \\
0 \ar[r] & \tau^{-1}(S/K) \ar[r]^-{j} & Y \ar[r]^-{p'}  & S/K \ar[r] & 0,
}}
$$
where the top row comes from the beginning of Resolution~(\ref{ML}) of $S/K.$ A morphism
$$
\vcenter{%
\xymatrix@C=35pt@R=25pt
{%
0 \ar[r] & \Tor_1 (-,S/K) \ar[r] \ar@{=}[d] & - \otimes_S J \ar[d]^{- \otimes j_0} \ar[r] & - \otimes I \ar[d]^{- \otimes i_0} 
\ar[r]^{- \otimes p}  & - \otimes S/K \ar@{=}[d] \ar[r] & 0 \\
0 \ar[r] & \Tor_1 (-,S/K) \ar[r] & - \otimes_S \tau^{-1}(S/K) \ar[r]^-{- \otimes j} & - \otimes Y \ar[r]^-{- \otimes p'}  & - \otimes S/K \ar[r] & 0.
}}
$$
of long exact sequences is induced, where $\Tor_1 (-,I) = \Tor_1 (-,Y) = 0.$ We know that the pure injective envelope of ${_S}J$ is an indecomposable pure injective module $\PE (J).$ The injective envelope of $- \otimes_S J$ is therefore also an indecomposable object of $(\modr S, \Ab).$ It follows that $- \otimes_S J$ is a uniform functor, and that the morphism $- \otimes j_0 : - \otimes J \to - \otimes_S \tau^{-1}(S/K)$ is its injective envelope, $\tau^{-1}(S/K) = \PE (J).$ This implies that the morphism $i_0 : I \to Y$ is a monomorphism and, as $Y$ is an indecomposable injective object, that this morphism is the injective envelope of $I.$ We may relabel the objects of the almost split sequence~(\ref{ass}) as shown in the middle row of
$$
\vcenter{%
\xymatrix@C=35pt@R=25pt
{%
& 0 \ar[d] & 0 \ar[d] \\
0 \ar[r] & J \ar[d]^{j_0} \ar[r] & I \ar[d]^{i_0} \ar[r]^{p}  & S/K \ar@{=}[d] \ar[r] & 0 \\
0 \ar[r] & \PE (J) \ar[r]^-{j} \ar[d] & E(I) \ar[r]^-{p'} \ar[d] & S/K \ar[r] & 0 \\
& \Omega^{-1}(I) \ar@{=}[r] \ar[d] & \Omega^{-1}(I) \ar[d] \\
& 0 & 0.
}}
$$
The middle column is pure exact, so that $\Omega^{-1}(I)$ belongs to $\SAbs \cap \SFlat,$ just as $E(I)$ does. The cokernel of the pure injective envelope $j_0 : J \to \PE (J)$ is flat, which implies that it is a special cotorsion preenvelope of $J.$ By the minimality of $j_0,$ it is the cotorsion envelope of $J.$ 

Incidently, the pure epimorphism $\PE (J) \to \Omega^{-1}(I)$ shows that the minimal definable subcategory $\D (\PE (J))$ containing $\PE (J)$ also contains the flat absolutely pure module $\Omega^{-1}(I),$ which implies that $\D (\PE (J))$ contains the minimal definable subcategory $\SAbs \cap \SFlat.$ \bigskip

\noindent (3) $\Tor_2 (-,S/K).$ The minimal left almost split morphism $s_{\omega} : S/I \to I_{\omega}/Ig$ composed with the pure injective envelope yields a minimal left almost split morphism in the category $\SPInj.$ The module $I_{\omega}/Ig$ is an absolutely pure uniserial module, with socle $S/I,$ so that its pure injective envelope is given by its injective envelope $v : I_{\omega}/Ig \to E(I_{\omega}/Ig) = E(S/I)$ as shown in
$$
\vcenter{%
\xymatrix@C=55pt@R=25pt
{%
0 \ar[r] & S/I \ar[r]^{s_{\omega}} \ar@{=}[d] & I_{\omega}/Ig \ar[d]^{v} \\
0 \ar[r] & S/I \ar[r]^{s'} & E(S/I).
}}$$

Any of the flat resolutions of $S/K$ that have been constructed above may be used to calculate the higher derived functors
$\Tor_i (-,S/K).$ This is done by applying the canonical functor $\Modl S \to (\modr S, \Ab)$ and taking homology. Applied to 
Resolution~(\ref{special cores}) of $S/K,$ this principle yields the morphism of exact sequences
$$
\vcenter{%
\xymatrix@C=35pt@R=25pt
{%
0 \ar[r] & \Tor_2 (-,S/K) \ar[r] \ar@{=}[d] & - \otimes_S S/I \ar[r]^{- \otimes s_{\omega}} \ar@{=}[d] & - \otimes_S I_{\omega}/Ig \ar[d]^{- \otimes v} \\
0 \ar[r] & \Tor_2 (-,S/K) \ar[r] & - \otimes_S S/I \ar[r]^{- \otimes s'} & - \otimes_S E(S/I),
}}$$
which proves that $\Tor_2 (-,S/K)$ is a simple functor.\bigskip

To complete the construction of the minimal flat resolution of $S/K,$ consider a morphism from Coresolution~(\ref{cores}) of $S/I$ to the minimal injective coresolution of $S/I.$ It induces a morphism of short exact sequences as shown in
$$
\vcenter{%
\xymatrix@C=35pt@R=25pt
{%
& & 0 \ar[d] & 0 \ar[d] \\
0 \ar[r] & S/I \ar[r]^{s} \ar@{=}[d] & I/Ig \ar[d]^{v_0} \ar[r] & J \ar[r] \ar[d]^{j'} & 0 \\
0 \ar[r] & S/I \ar[r]^{s'} & E(S/I) \ar[r] \ar[d] & \Omega^{-1}(S/I) \ar[r] \ar[d] & 0 \\ 
& & \Omega^{-1}(I/Ig) \ar@{=}[r] \ar[d] & \Omega^{-1}(I/Ig) \ar[d] \\
& & 0 & 0.
}}$$
The module $I/Ig$ is uniserial, so that the morphism $v_0 : I/Ig \to E(S/I)$ is an injective envelope. It is also an absolutely pure module, so that this injective envelope is its pure injective envelope and the middle column is pure exact. This implies that the cosyzygy $\Omega^{-1}(I/Ig)$ is a pure epimorphic image of the flat module $E(S/I)$ and is therefore itself a flat module. Consequently, the injective envelope of $I/Ig$ is also its cotorsion envelope. 

Because $S/I$ and $E(S/I)$ are both  cotorsion modules, so is the quotient module $\Omega^{-1} (S/I).$ The cokernel of the morphism
$j': J \to \Omega^{-1}(S/I)$ is flat, so that $j'$ is a special cotorsion preenvelope of $J.$ It follows that the monomorphism $j' : J \to \Omega^{-1}(S/I)$ factors through the cotorsion envelope $j_0 : J \to \PE (J),$ which yields a factorization of pullbacks
$$
\vcenter{%
\xymatrix@C=55pt@R=25pt
{%
0 \ar[r] & S/I \ar[r]^{s} \ar@{=}[d] & I/Ig \ar[d]^{w_0} \ar[r] & J \ar[r] \ar[d]^{j_0} & 0 \\
0 \ar[r] & S/I \ar[r] \ar@{=}[d] & W \ar[d]^{w_1} \ar[r] & \PE(J) \ar[r] \ar[d]^{j_1} & 0 \\
0 \ar[r] & S/I \ar[r]^{s'} & E(S/I) \ar[r] & \Omega^{-1}(S/I) \ar[r] & 0,
}}$$
where $j_1 : \PE (J) \to \Omega^{-1}(S/I)$ is a monomorphism, because it possesses a retraction. The morphism $w_1 : W \to E(S/I)$ is therefore also a monomorphism. But the module $W$ is an extension of cotorsion modules and is therefore itself cotorsion and the only cotorsion module $I/Ig \subseteq W \subseteq E(I/Ig)$ between $I/Ig$ and its cotorsion envelope is the envelope itself. The morphism
$w_1 : W \to E(S/I)$ is therefore an isomorphism and $j': J \to \Omega^{-1}(S/I)$ is the cotorsion envelope $j_0$ of $J.$ We have the short exact sequence
$$
\vcenter{%
\xymatrix@1@C=55pt
{%
0 \ar[r] & S/I \ar[r] & E(S/I) \ar[r] & \PE(J) \ar[r] & 0.
}}$$
The epimorphism is the flat cover of $\PE (J),$ because $E(S/I)$ is an indecomposable flat module and the kernel $S/I$ is cotorsion. Attaching this short exact sequence to the almost split sequence~(\ref{ass}) yields the resolution shown in Theorem~\ref{minimal cores}, which is now seen to be the minimal injective coresolution of $S/I$ and the minimal flat resolution of $S/K.$ This completes the proof of Theorem~\ref{minimal cores}. \bigskip

Let us summarize. If $U \in \Zg (S),$ then $- \otimes_R U$ is an indecomposable injective object of the functor category $(\modr S, \Ab)$ and we have the following possibilities:
\begin{description}
\item[$S/K$] A point $U \in \Zg (S)$ does not belong to $(S/K)^{\perp_0}$ if and only if
$$\Hom (S/K, U) \simeq \Hom [- \otimes_R S/K, - \otimes_S U] \neq 0$$
if and only if $- \otimes_S U = E(- \otimes_S S/K) = - \otimes S/K$ if and only if $U = S/K.$ The point $S/K$ is therefore isolated, because it is the only point not in the closed subset $\Cl [(S/K)^{\perp_0}].$ On the other hand, it is also closed, 
$\Cl (\Add (S/K)) = \{ S/K \}.$
\item[$\PE(J)$] A point $U \in \Zg (S)$ does not belong to $(S/K)^{\perp_1}$ if and only if
$$\Ext^1 (S/K,U) \simeq \Hom [\Tor_1 (-,S/K), - \otimes U] \neq 0$$
if and only if $- \otimes_S U = E(\Tor_1 (-,S/K)) = - \otimes_S \PE (J)$ if and only if $U = \PE (J).$ The point $\PE (J)$ is therefore isolated, because it is the only point not in the closed subset $\Cl [(S/K)^{\perp_1}].$ 
\item[$S/I$] A point $U \in \Zg (S)$ does not belong to $(S/K)^{\perp_2}$ if and only if
$$\Ext^2 (S/K, U) \simeq \Hom [\Tor_2 (-,S/K), - \otimes U] \neq 0$$
if and only if $- \otimes_S U = E(\Tor_2 (-,S/K))) = - \otimes S/I$ if and only if $U = S/I.$ The point $S/I$ is the only point not in 
$\Cl [(S/K)^{\perp_2}],$ but it is also closed, because $\Cl (\Add (S/I)) = \{ S/I \}.$
\end{description}
As a consequence, any point $U \in \Zg (S)$ that is not equal to $S/K,$ $\PE (J),$ or $S/I$ belongs to 
$(S/K)^{\perp_0} \cap (S/K)^{\perp_1} \cap (S/K)^{\perp_2} = \SAbs \cap \SFlat$ and therefore belongs to minimal closed subset 
$Zg'(S) = \Cl (\SAbs \cap \SFlat).$ Because this closed set contains at least the two points $E(S/I)$ and $E(I),$ the points of
$\Zg'(S)$ are not isolated. This proves Theorem~\ref{Zg spec} and implies that the list of definable subcategories of $\Modl S$ is complete. For example, we can deduce now that $\{ \PE (J) \}^- = \Cl (\T_I \cap \F_K),$ or the following.

\begin{cor} \label{all tilting}
If $S$ is a Dubrovin-Puninski ring, then there exist exactly three tilting classes in $\Modl S,$ given by
\begin{description}
\item[0] $\Modl S = S^{\perp} = \modl S^{\perp_3};$
\item[1] $\T_I = (I_{\omega} \oplus J_{\omega})^{\perp} = K^{\perp} = \modl S^{\perp_2};$ and
\item[2] $\SAbs = (I_{\omega} \oplus S/K)^{\perp} = \modl S^{\perp_1}.$
\end{description}
Moreover, every nonprojective tilting module is a pure projective module that is not tilting equivalent to a finitely generated
(classical) tilting module.
\end{cor}

\noindent (4) $(- \otimes_S S/I)/\Tor_2 (-,S/K).$ If $U \in \Zg (S)$ is not one of the three isolated points $S/K,$ $S/I,$ of $\PE (J),$ then it must be injective. In such a situation, a minimal left almost split morphism with domain $U$ exists only if $U$ is the injective envelope of a simple module. The only possibility is that $E(S/I),$ in which case the minimal left almost morphism is given by the quotient map $k : E(S/I) \to \PE(J)$ and the beginning of the associated long exact sequence is seen to be
$$
\vcenter{%
\xymatrix@1@C=35pt@R=15pt
{%
{\cdots} \ar[r] & 0 \ar[r] & \Tor_2 (-,S/K) \ar[r] & - \otimes_S S/I \ar[r]^{- \otimes i} & - \otimes E(S/I) \ar[r]^{- \otimes k}  & - \otimes \PE(J) \ar[r] & 0.
}}
$$
The kernel of the induced morphism $- \otimes k$ is therefore given by the simple object $(- \otimes_S S/I)/\Tor_2(-,S/K)$ in 
$(\modr S, \Ab).$ Because the coherent object $\Tor_2 (-,S/K)$ is simple, the functor $- \otimes_S S/I$ has length $2$ in
$(\modr S, \Ab).$ \bigskip

This concludes the proof of Theorem~\ref{Max Spec} which gives the classification of the simple objects of $(\modr S, \Ab)$ and yields the following decomposition theorem for for a pure injective $S$-module. In the following, note that both of the simple left $S$-modules $S/K$ and $S/I$ are endosimple and therefore $\Sigma$-pure injective. 

\begin{thm} \label{decomp}
Given a pure injective $S$-module $N,$ there exist cardinals $\alpha,$ $\beta,$ $\gamma,$ and $\delta$ and an injective flat module $N_0$ such that
$$N \simeq (S/K)^{(\alpha)} \oplus \PE [\PE(J)^{(\beta)}] \oplus (S/I)^{(\gamma)} \oplus E[(S/I)^{(\delta)}] \oplus N_0.$$
This decomposition is unique up to isomorphism. 
\end{thm}

\begin{proof}
The decomposition is obtained by considering the socle of the injective object $- \otimes_S N$ in the category $(\modr S, \Ab).$ It will be of the form
$$\soc (- \otimes_S N) = (- \otimes_S S/K)^{(\alpha)} \; \amalg \; \Tor_1 (S/K, -)^{(\beta)} \; \amalg \; \Tor_2 (S/K,-)^{(\gamma)} \; \amalg \; 
[(- \otimes_S S/I)/\Tor_2 (-,S/K)]^{(\delta)},$$ where $\alpha,$ $\beta,$ $\gamma,$ and $\delta$ are cardinals. The injective object 
$- \otimes N$ is then decomposed as
$$- \otimes_S N = E[\soc (- \otimes N)] \; \amalg \; - \otimes N_0,$$
where the injective complement $- \otimes N_0$ has no simple subobject. By Proposition~\ref{isolated}, the module $N_0$ is absolutely pure and flat. Because $N_0$ is pure injective and absolutely pure, it must be injective. 
\end{proof}

The classification of definable subcategories of $\Modl S$ and the decomposition of Theorem~\ref{decomp} yields the following observation.

\begin{cor}
For every definable subcategory $\D \subseteq \Modl S,$ $\D \cap \SPInj = \Prod (\D \cap \MaxZg (S)).$
\end{cor}

\begin{proof}
Suppose that $N \in \D$ is pure injective. If the module $N_0 = 0$ in Theorem~\ref{decomp}, then all the indecomposable summands of $N$ belong to $\D \cap \MaxZg (S)$ and $N$ is a direct summand of their product. If the injective flat module $N_0$ is nonzero, then $\D$ contains $\SAbs \cap \SFlat$ and $E(S/I) \in \D \cap \MaxZg (S).$ But $N_0$ belongs to $\Cogen (E(S/I))$ and the result is clear.   
\end{proof}

\noindent {\bf i. The three cotilting classes.}\bigskip

In this section, we prove a result dual to that of Corollary~\ref{all tilting}. Namely, we show that there exist exactly three cotilting classes in $\Modl S,$ given by the classes of modules of flat dimension $0,$ $1,$ and $2,$ respectively. In order to accomplish this, we use the foregoing analysis to characterize the modules of $\Modl S$ according to their flat dimensions. The first step is to prove that every Dubrovin-Puninski ring is a Xu ring. A ring $S$ is called {\em left Xu} if every left cotorsion $S$-module is pure injective. An equivalent condition on the ring $S$ is that for every module ${_S}M$ the pure cosyzygy $\Omega_p^{-1}(M),$ which is the cokernel of the pure injective envelope of $M,$ 
$$
\vcenter{%
\xymatrix@1@C=45pt@R=35pt
{%
0 \ar[r] & M \ar[r]^{m} & \PE (M) \ar[r]^{n} & \Omega_p^{-1}(M) \ar[r] & 0
}}
$$
is flat. For then the pure injective envelope $\PE (M)$ is the cotorsion envelope of $M.$

\begin{prop} \label{Xu}
If $S$ is a Dubrovin-Puninski ring and ${_S}M$ is a left $S$-module, then the pure cosyzygy $\Omega_p^{-1}(M)$ is flat and absolutely pure. 
\end{prop}

\begin{proof}
We will use the fact that for every module ${_S}M,$ the object $- \otimes_S M$ in the functor category $(\modr S, \Ab)$ is {\em fp-injective.} Let $m:M \to \PE(M)$ be the pure injective envelope of $M$ in $\Modl S.$ The short exact sequence above is pure exact and so induces a short exact sequence
$$
\vcenter{%
\xymatrix@1@C=45pt@R=15pt
{%
0 \ar[r] & - \otimes_S M \ar[r]^{- \otimes m} & - \otimes_S \PE(M) \ar[r]^{- \otimes n}  & - \otimes_S \Omega^{-1}_p (M) \ar[r] & 0
}}
$$
in the functor category. The morphism $- \otimes m : - \otimes_S M \to - \otimes_S \PE(M)$ is the injective envelope of $- \otimes_S M$ in the functor category $(\modr S, \Ab).$ It is an essential extension of $- \otimes_S M;$ if any of the simple coherent objects $- \otimes_S S/K,$ $\Tor_1 (-,S/K),$ or $\Tor_2 (-,S/K)$ admitted a nonzero morphism into $- \otimes \Omega^{-1}_p (M),$ we could take the pullback of the short exact sequence of functors along that monomorphism to obtain a nonsplit extension of the simple coherent functor and the fp-injective object $- \otimes_S M,$ a contradiction. By Proposition~\ref{isolated}, the module $\Omega^{-1}_p (M)$ must be flat and absolutely pure.
\end{proof}

Proposition~\ref{Xu} implies that the complete cotorsion pair generated by the class $\SFlat$ of flat $S$-modules is given by 
$(\SFlat, \SPInj).$ It also implies that the flat syzygy $\Omega^{\flat}(M),$ which is the kernel of the flat cover of $M,$
$$
\vcenter{%
\xymatrix@1@C=45pt@R=35pt
{%
0 \ar[r] & \Omega^{\flat}(M) \ar[r]^{r} & \FC (M) \ar[r]^{s} & M \ar[r] & 0
}}
$$
is a pure injective module. By the minimality property, the pure injective module $\Omega^{\flat}(M)$ cannot have injective direct summands, for they would give rise to direct summands of $\FC (M),$ contained in the kernel. The decomposition of the flat syzygy of
$M,$ according to Theorem~\ref{decomp}, must be of the form
$$\Omega^{\flat}(M) = \PE [\PE (J)^{(\beta)}] \oplus (S/I)^{(\gamma)},$$
where the summand $(S/I)^{(\gamma)}$ is a flat module.

In the following, we will use the notation $t_K$ to denote the torsion radical of the torsion pair $(\Add (S/K), \F_K).$ For any $S$-module $M,$ $t_K (M)$ is the $S/K$-socle of $M$ and, because the module $S/K$ is $\Sigma$-injective, $t_K (M)$ is a direct summand of $M.$  

\begin{lem} \label{flat dim 1}
A left $S$-module $M$ has flat dimension at most $1$ if and only if $t_K(M) = 0.$ 
\end{lem}

\begin{proof}
If $t_K (M) \neq 0,$ then $M$ has a direct summand isomorphic to $S/K,$ whose minimal flat resolution $S/K$ has length $2.$ The flat dimension of $M$ cannot therefore be at most $1.$

If $t_K (M) = 0,$ consider the flat cover of $M$ and let us prove that the cardinal $\beta = 0.$ This will imply that 
$\Omega^{\flat}(M) \in \Add (S/I)$ and yield a minimal flat resolution of length at most $1.$ If $\beta > 0,$ then there is a summand of $\Omega^{\flat}(M)$ isomorphic to $\PE (J),$ which, considered as a submodule of $\FC (M)$ is {\em not} a direct summand. The canonical morphism $\iota : \PE (J) \to \Omega^{\flat}(M)$ induces a morphism
$$
\vcenter{%
\xymatrix@1@C=45pt@R=35pt
{%
0 \ar[r] & \PE(J) \ar[r]^{j} \ar[d]^{\iota} & E(I) \ar[r]^{p'} \ar[d] & S/K \ar[r] \ar[d]^{t} & 0 \\
0 \ar[r] & \Omega^{\flat}(M) \ar[r] & \FC (M) \ar[r] & M \ar[r] & 0
}}
$$
from the almost split sequence with left term $\PE (J).$ By hypothesis, $t : S/K \to M$ is zero, so that this morphism of short exact sequences is null homotopic. This implies that the morphism $\iota$ factors through the injective envelope $E(I).$ Composing $\iota$ with the canonical projection $\pi : \Omega^{\flat}(M) \to \PE (J)$ would then yield a splitting of the almost split sequence, a contradiction.
\end{proof}

Lemma~\ref{flat dim 1} indicates that the torsion free class $\F_K,$ which was earlier identified as $\Cogen (E(S/I)),$ consists of the $S$-modules of flat dimension at most $1.$ Because $S/K$ has flat dimension $2,$ it also implies that every $S$-module has flat dimension at most $2.$ The almost split sequence used in the previous proof may be employed to obtain another characterization of $\F_K,$ for the modules in $\F_K$ are precisely those modules $M$ for which there exists no split epimorphism $q: M \to S/K.$ This implies that every such morphism factors as indicated by the dotted arrow in
$$
\vcenter{%
\xymatrix@C=45pt@R=35pt
{%
&&& M \ar[d]^{u} \ar@{.>}[dl] \\
0 \ar[r] & \PE(J) \ar[r]^{j} & E(I) \ar[r]^{p'} & S/K \ar[r] & 0.
}}
$$
This almost split sequence in the bottom row is also the minimal injective coresolution of $\PE (J),$ so that $M \in \F_K$ if and only if $\Ext^1 (M, \PE (J)) = 0.$ Because of the way it is situated as part of the minimal injective coresolution of $S/I,$ we get that 
\begin{equation} \label{K-tf}
\F_K = {^{\perp_1}}\PE (J) = {^{\perp_2}}S/I.
\end{equation}

A left $S$-module $C_n$ is an $n$-{\em cotilting} module if 
\begin{enumerate}
\item the injective dimension $\id (C_n) \leq n;$ 
\item $\Ext_S^i (C_n^{\aleph}, C_n) = 0,$ for all natural numbers $i > 0$ and cardinals $\aleph;$ and
\item there exists a finite coresolution of the minimal injective cogenerator ${_S}E_0$ of $\Modl S$ in $\Prod (C_n),$ the subcategory of summands of products of copies of $C.$
\end{enumerate} 
If $C_n$ is an $n$-cotilting module, then the associated {\em cotilting class} is defined to be the subcategory
$${^{\perp}}C_n = \{ M \in \Modl S \; | \; \Ext^i (M, C_n) = 0, \; \mbox{for all} \; i > 0  \}.$$
Two $n$-cotilting modules $C_n$ and $C_n'$ are equivalent if $\Prod (C_n) = \Prod (C_n').$ \bigskip

By~\cite{B03} and~\cite{St06}, every cotilting class $\C$ is a definable subcategory of $\Modl S$ that contains the definable subcategory $\SFlat.$ According to our classification of the definable subcategories of $\Modl S,$ there are only four such: $\Modl S,$ $\F_K,$ $S/K^{\perp_1},$ and $\SFlat$ itself. The category $\Modl S$ is the unique $0$-cotilting class associated to the $0$-cotilting module given by the minimal injective cogenerator $E_0 = S/K \oplus E(S/I).$ 

\begin{prop} \label{cotilt1}
The torsion free class $\F_K = {^{\perp_2}}S/I$ is a $1$-cotilting class with cotilting module $C_1 = E(S/I) \oplus \PE(J).$
\end{prop}

\begin{proof}
Let us verify that $C_1$ is a $1$-cotilting module. (1) The almost split sequence whose left term is $\PE (J)$ is an injective coresolution of $\PE (J)$ of length $1$ so that the injective dimension of $\PE(J),$ and therefore that of $C_1,$ is at most $1.$
(2) By Equation~(\ref{K-tf}), $\F_K = {^{\perp_1}}\PE(J) = {^{\perp}}\PE(J) = {^{\perp_1}}C_1,$ and it is easy to see that no product $C_1^{\aleph}$ of copies of $C_1$ has a submodule isomorphic to $S/K.$ (3) The injective indecomposable $E(I)$ belongs to $\Prod (E(S/I)) \subseteq \Prod (C_1)$ so that the minimal flat resolution of $S/K$ given in Theorem~\ref{minimal cores} is in
$\Prod (C_1).$
\end{proof}

\begin{prop} \label{cotilt2}
The definable subcategory $\SFlat \subseteq \Modl S$ is a $2$-cotilting class with cotilting module $C_2 = E(S/I) \oplus S/I.$
\end{prop}

\begin{proof}
(1) The minimal injective coresolution of $S/I$ given in Theorem~\ref{minimal cores} shows that the injective dimension of $S/I,$ and therefore that of $C_2,$ is at most $2.$ (2) The module $C_2$ is pure injective, and therefore cotorsion. It is also flat, so that for any cardinal $\aleph,$ the product $C_2^{\aleph} \subseteq \SFlat \subseteq {^{\perp}}C_2.$ (3) Because $E(I) \in \Prod (E(S/I)),$ the minimal injective coresolution~(\ref{minimal cores}) of $S/I$ is a resolution of $S/K$ in $\Prod (C_2).$  

Let us argue that ${^{\perp}}C_2 = {^{\perp}S/I} = {^{\perp_1}}S/I \cap {^{\perp_2}}S/I = \SFlat.$ Equation~(\ref{K-tf}) implies that ${^{\perp_2}}S/I = \F_K$ is the class of modules of flat dimension at most $1.$ If the flat dimension of a module $M \in {^{\perp}}C_2$ were to equal $1,$ we have seen how the kernel of its flat cover would have to be in $\Add (S/I),$ which would imply that $M$ does not belong to ${^{\perp_1}}S/I.$
\end{proof}

The definable subcategory $S/K^{\perp_1}$ contains the category $\SFlat,$ but is not a cotilting class ${^{\perp}}C$ for any module $C.$ This is because cotilting classes are closed under kernels of epimorphisms and the module $\PE (J),$ which does not belong to $S/K^{\perp_1}$ has injective dimension $1$ and every injective module belongs to $S/K^{\perp_1}.$\bigskip

\noindent {\bf j. The ten preenveloping subcategories of $\SPProj.$} \bigskip

Suppose that $M$ is an $S$-module of flat dimension at most $1.$ We have shown above that this is equivalent to the condition that
the torsion submodule $t_K(M) = 0$ and that it implies that the flat syzygy $\Omega^{\flat}(M) = (S/I)^{(\gamma)}$ is a direct sum of copies of $S/I.$ By the left almost split property of the morphism $s_{\omega}: S/I \to I_{\omega}/Ig,$ this yields a morphism of short exact sequences as depicted by
$$
\vcenter{%
\xymatrix@1@C=55pt@R=35pt
{%
         &                                & 0 \ar[d] & 0 \ar[d] & \\
0 \ar[r] & (S/I)^{(\gamma)} \ar[r]^{s_{\omega}^{(\gamma)}} \ar@{=}[d] & (I_{\omega}/Ig \ar[r])^{(\gamma)} \ar[d]^{h} & J_{\omega}^{(\gamma)} \ar[r] \ar[d] & 0  \\
0 \ar[r] & \Omega^{\flat}(M) \ar[r] & \FC (M) \ar[r] \ar[d] & M \ar[r] \ar[d] & 0  \\
         &                        & F \ar@{=}[r] \ar[d]   & F \ar[d]  \\
				 &                        & 0                     & 0.  
}}
$$
Because $I_{\omega}/Ig$ is essential over $S/I,$ the morphism $h : (I_{\omega}/Ig)^{(\gamma)} \to \FC (M)$ is a monomorphism and, because 
$I_{\omega}/Ig$ is absolutely pure, the middle column is a pure exact sequence. The cokernel $F$ is a pure epimorphic image of the flat module $\FC (M)$ so that it too is a flat module. This implies that the short exact sequence given by the column on the right is also pure exact. This leads to the following important point.

\begin{thm} \label{pure filtration}
Every left $S$-module $M$ possesses a filtration of pure submodules
$$0 \subseteq t_K (M) \subseteq M_2 \subseteq IM \subseteq M,$$
where $t_K (M) \in \Add (S/K),$ $M_2 \simeq J_{\omega}^{(\gamma)},$ the factor $IM/M_2$ is absolutely pure and flat, and
$M/IM \in \Add (S/I).$ 
\end{thm}

\begin{proof}
Consider the module $M' = IM/t_K(M),$ which has flat dimension at most $1,$ because it belongs to $\F_K,$ to obtain a short exact sequence 
$$
\vcenter{%
\xymatrix@1@C=55pt@R=15pt
{%
0 \ar[r] & J_{\omega}^{(\gamma)} \ar[r] & IM \ar[r] & F \ar[r] & 0, 
}}
$$
where $F$ is a flat module that satisfies $IF = F$ and is therefore absolutely pure. Let $M_2$ be the preimage of $J_{\omega}^{(\gamma)}$ in $M.$ Both of the factors $IM/M_2$ and $M/IM$ are flat, so that $M_2$ and $IM$ are pure submodules of $M;$ the submodule $t_K (M)$ is a direct summand.
\end{proof}

The factors of the filtration given in Theorem~\ref{pure filtration} determine the least definable subcategory
$\D (M) \subseteq \Modl S$ that contains $M.$ For we can check that $S/K \in \D (M)$ if and only $t_K (M) \neq 0;$
$\PE (J) \in \D (M)$ if and only if $M_2/t_K(M) \neq 0;$ $S/I \in \D (M)$ if and only if $M/IM \neq 0;$ and $\D (M)$ contains the minimal definable subcategory $\SAbs \cap \SFlat$ if and only if $IM/t_K(M) \neq 0.$

Another consequence of Theorem~\ref{pure filtration} is that every left $S$-module is pure projective-by-flat. Indeed, the submodule
$M_2$ belongs to $\Add (S/K \oplus J_{\omega})$ and the factor $M/M_2$ is flat. This is used in the next result, which is a kind of dual to Theorem~\ref{Xu}.  

\begin{thm} \label{dual Xu}
For every left $S$-module $M,$ there exists a pure exact sequence
$$
\vcenter{%
\xymatrix@1@C=55pt@R=15pt
{%
0 \ar[r] & I' \ar[r] & P \ar[r] & M \ar[r] & 0, 
}}
$$
with $P$ pure projective and $I'$ absolutely pure and flat.
\end{thm} 
  
\begin{proof}
It suffices to prove the theorem for flat modules. For then we can express the module $M$ as a pure projective-by-flat module and take the pullback
$$
\vcenter{%
\xymatrix@1@C=55pt@R=35pt
{%
         &                                 & 0 \ar[d]             & 0 \ar[d] \\
         &                                 & I' \ar[d] \ar@{=}[r] & I' \ar[d] \\
0 \ar[r] & M_2 \ar[r] \ar[d] & M_2 \oplus P \ar[r] \ar[d] & P \ar[r] \ar[d] & 0 \\
0 \ar[r] & M_2 \ar[r] & M \ar[r] \ar[d] & F \ar[r] \ar[d] & 0 \\
         &                          & 0               & 0. 
}}
$$
The bottom row is pure exact with $F$ flat and $M_2$ pure projective. The module $P$ is pure projective, so that the middle row is a split exact sequence. The module $I'$ is absolutely pure and $M_2 \oplus P$ is pure projective, as required.

If $M = F$ is a flat module, then we use the filtration $0 \subseteq IF \subseteq F$ of Theorem~\ref{pure filtration} and work with the factors $IF$ and $F/IF = (S/I)^{(\gamma)}$ separately. The flat module $IF$ belongs to the $1$-tilting class $\T_I = \Gen (I_{\omega}),$ which gives rise to the morphism of short exact sequences in the bottom two rows of
$$
\vcenter{%
\xymatrix@1@C=55pt@R=35pt
{%
         & 0 \ar[d]            & 0 \ar[d]         & 0 \ar[d] \\
0 \ar[r] & I'' \ar[r] \ar[d]   & I' \ar[d] \ar[r] & I^{(\gamma)} \ar[r] \ar[d] & 0 \\
0 \ar[r] & I_{\omega}^{(\delta)} \ar[r] \ar[d] & I_{\omega}^{(\delta)} \oplus S^{(\gamma)} \ar[r] \ar[d] & S^{(\gamma)} \ar[r] \ar[d]^{(\gamma)} & 0 \\
0 \ar[r] & IF \ar[r] \ar[d] & F \ar[r] \ar[d] & (S/I)^{(\gamma)} \ar[r] \ar[d] & 0 \\
         & 0                & 0               & 0. 
}}
$$
The kernel of this morphism is the short exact sequence that appears in the top row. In every column, the short exact sequence is pure exact, because the third term is flat. In the first column, the kernel $I''$ is a pure submodule of
$I_{\omega}^{(\delta)}$ and is therefore absolutely pure and flat; in the third column, the kernel $I^{(\gamma)}$ is likewise absolutely pure and flat. It follows that the kernel $I'$ in the middle column is also absolutely pure and flat, as required.
\end{proof}

The short exact sequence found in Theorem~\ref{dual Xu} implies that if the module $M$ belongs to ${^{\perp}}\SAbs,$ then it must a direct summand of $P$ and therefore pure projective. \bigskip

The second author and Ph.\ Rothmaler~\cite{HR} proved that if $\C$ is a preenveloping subcategory of the category $\Add (\modl S) = \SPProj$ of pure projective modules, then the category $\tilde{\C}$ of objects that arise as pure epimorphic images of coproducts of objects in $\C$ is a definable subcategory of $\Modl S.$ The rule $\C \mapsto \tilde{\C}$ is then a bijective correspondence between the collection of preenveloping subcategories in $\SPProj$ and that of definable subcategories of $\Modl S$ that have enough pure projective modules; the inverse correspondence is given by 
$\D \mapsto \D \cap \SPProj.$

\begin{cor} \label{preenveloping}
Every definable subcategory $\D \subseteq \Modl S,$ save $\Add (S/I)$ and $\Add (S/J),$ has enough pure projective modules.
\end{cor}

\begin{proof}
If a definable subcategory $\D \subseteq \Modl S$ contains $\SAbs \cap \SFlat,$ then it has enough pure projective modules by Theorem~\ref{dual Xu}: every module $M$ in $\D$ admits a pure epimorphism from a pure projective module $P,$ where the kernel $I'$ of the pure epimorphism lies in $\SAbs \cap \SFlat \subseteq \D.$ The pure projective module $P$ therefore belongs to $\D$ too. 

Among the four semisimple definable subcategories $\D \subseteq \Modl S$ that do not contain $\SAbs \cap \SFlat,$ it is readily verified that the classes $\Add (S/I)$ and $\Add (S/J)$ do not have enough pure projective modules, while $\Add (S/K)$ and $0$ do. 
\end{proof}

By the bijective correspondence between definable subcategories of $\Modl S$ with enough pure projective modules and preenveloping subcategories of $\Add (\modl S),$ we conclude that there exist exactly $10$ preenveloping subcategories of $\SPProj.$ They appear in the third column of Table 1. \bigskip

\noindent {\bf k. The six covariantly finite subcategories of $\modl S.$} \bigskip

An additive subcategory $\C_0 \subseteq \modl S$ is called {\em covariantly finite} if it is preenveloping in $\modl S.$ Every covariantly finite subcategory $\C_0 \subseteq \modl S$ gives rise to the preenveloping subcategory $\Add (\C_0)$ of $\SPProj.$
If $P \in \SPProj$ is pure projective, then there is a module $P'$ such that $P \oplus P' = \oplus_i A_i$ is a direct sum of finitely presented. It is then readily verified that the composition $P \to P \oplus P' = \oplus_i A_i \to \oplus_i (A_i)_{\C_0}$ is an
$\Add (\C_0)$-preenvelope of $P.$ The following is a sort of converse.

\begin{prop} \label{cov finite}
Suppose that a preenveloping subcategory $\C \subseteq \SPProj$ of pure projective modules is of the form $\C = \Add (\C_0),$ where $\C_0 \subseteq \modl S$ is an additive subcategory of finitely presented modules. Then $\C_0 = \C \cap \modl S$ is a covariantly finite subcategory of $\modl S.$
\end{prop}

\begin{proof}
First note that $\C_0$ is unique, for if $\Add (\C_0) = \Add (\C'_0)$ and both $\C_0$ and $\C'_0$ are additive subcategories of $\modl S,$ then $\C_0 = \C'_0.$ Now the inclusion
$\C = \Add (\C_0) \subseteq \Add (\C \cap \modl S) \subseteq \C$
implies that $\C_0 = \C \cap \modl S.$ To see that $\C_0$ is covariantly finite consider the $\C$-preenvelope $A \to A_{\C}$ of a finitely presented module $A.$ Because $A_{\C} \in \Add (\C_0),$ there is a module $A'$ such that 
$A_{\C} \oplus A' = \oplus_i B_i,$ with $B_i \in \C_0.$ The composition $A \to A_{\C} \to \oplus_i B_i$ with the structural injection is another $\C$-preenvelope of $A.$ Because $A$ is finitely presented this $\C$-preenvelope factors through a finite direct sum of the $B_i,$ which belongs to $\C_0.$ This morphism serves as a $\C_0$-preenvelope of $A.$
\end{proof}

In the correspondence $\C \mapsto \tilde{\C}$ between the preenveloping categories of $\SPProj$ and the definable subcategories of $\Modl S$ with enough pure projective modules, the preenveloping subcategories of the form $\Add (\C_0),$ with $\C_0$ covariantly finite in $\modl S,$ correspond to the definable subcategories $\D \subseteq \Modl S$ that {\em have enough finitely presented modules,} in the sense that every module $D \in \D$ admits a pure epimorphism $\oplus_i A_i \to D,$ where every $A_i$ is a finitely presented $S$-module that belongs to $\D.$

Among the ten preenveloping subcategories of $\SPProj$ that appear in the third column of Table 1, six are readily identified to be of the form $\Add (\C_0),$ where $\C_0 \subseteq \modl S$ is an additive subcategory of finitely presented modules. Beside $0$ and $\Add (S/K) = \Add (\add (S/K)),$ there are four that contain the module $S.$ By Fact~\ref{fp modules}, they are determined by which of the two finitely presented indecomposable modules $K$ and $S/K$ they contain:
\begin{enumerate}
\item $\Add (S) = \SFlat \cap \SPProj = \SProj = \Add (\Sproj);$
\item $\Add (S \oplus K) = \F_K \cap \SPProj = \Add (\add (S \oplus K));$
\item $\Add (S \oplus S/K) = S/K^{\perp_1} \cap \SPProj  = \Add (\add (S \oplus S/K));$ and
\item $\SPProj = \Modl S \cap \modl S = \Add (\modl S).$
\end{enumerate}
By Proposition~\ref{cov finite}, these considerations yield a list of $6$ covariantly finite subcategories of $\modl S:$ $0,$
$\add (S/K),$ $\add (S),$ $\add (S \oplus K),$ $\add (S \oplus S/K),$ and $\modl S.$ To be sure that there are no others, it suffices to check that none of the other $4$ preenveloping subcategories of $\SPProj$ are not of the form 
$\C = \Add (\C \cap \modl S).$ Let us use Fact~\ref{fp modules} to show that in these four cases 
$\C \cap \modl S \subseteq \tilde{\C} \cap \modl S$ is either $0$ or $\add (S/K):$ 
\begin{itemize}
\item $\Add (I_{\omega}) \cap \modl S \subseteq (\SAbs \cap \SFlat) \cap \modl S = 0;$
\item $\Add (I_{\omega} \oplus S/K) \cap \modl S \subseteq \SAbs \cap \modl S = \add (S/K);$
\item $\Add (I_{\omega} \oplus J_{\omega}) \cap \modl S \subseteq [(S/K)^{\perp_1} \cap (S/K)^{\perp_2}] \cap \modl S = 0;$ and
\item $\Add (I_{\omega} \oplus J_{\omega} \oplus S/K) \cap \modl S \subseteq \T_I \cap \modl S = \add (S/K).$ 
\end{itemize}

If $\C = \Add (\C_0)$ is a preenveloping subcategory of $\SPProj$ associated to a covariantly finite subcategory, then a result of Lenzing implies that every object in the definable subcategory $\tilde{\C}$ is a direct limit of objects in $\C_0.$ We can express this as ${\displaystyle \tilde{\C} = \lim_{\to} (\C_0)}.$ Crawley-Boevey~\cite[\S 4]{CB94} and H.\ Krause~\cite[Proposition 3.11]{HK} showed that if a definable subcategory is of the form ${\dis \lim_{\to} (\C_0)},$ for some additive subcategory of $\modl S,$ then $\C_0$ is covariantly finite. This follows from the considerations above, but was first proved without using $\SPProj$ as the ambient category.
The four covariantly finite subcategories $\C_0$ of $\modl S$ that contain $S$ give rise in this way to the four definable subcategories ${\dis \lim_{\to} (\C_0)}$ of $\Modl S$ that contain $\SFlat.$ These four definable subcategories generate the following cotorsion pairs:
\begin{enumerate}
\item $(\SFlat, \SPInj) = (\SFlat, \Modl S \cap \SPInj);$
\item $(\F_K, \Prod (\PE (J) \times S/K \times E(S/I))) = (\F_K, \T_I \cap \SPInj);$
\item $(S/K^{\perp_1}, \Prod (S/I \times S/K \times E(S/I))) = (S/K^{\perp_1}, S/K^{\perp_1} \cap \SPInj);$ and
\item $(\Modl S, \SInj) = (\Modl S, \SAbs \cap \SPInj).$
\end{enumerate}
That $\SFlat^{\perp_1} = \SPInj$ follows from the fact that $S$ is a Xu ring. Because the definable categories on the left of each pair contain $\SFlat,$ the categories on the right are contained in $\SPInj.$ The categories on the right contain the category $\SInj$ of injective $S$-modules and are closed under products and direct summands. Thus they are completely determined by which of the two pure injective noninjective indecomposable modules $S/I$ and $\PE (J)$ they contain. Looking at the third cotorsion pair on this list, we see that $S/K^{\perp_1} = {^{\perp_1}}S/I.$

There are four cotorsion pairs in $\Modl S$ that are generated by a subcategory of $\modl S:$ 
\begin{enumerate}
\item $(\SProj, \Modl S) = (\SFlat \cap \SPProj, \Modl S);$
\item $(\Add (S \oplus K), \T_I) = (\F_K \cap \SPProj, \T_I);$
\item $(\Add (S \oplus S/K), S/K^{\perp_1}) = (S/K^{\perp_1} \cap \SPProj, S/K^{\perp_1});$ and
\item $(\Add (\modl S), \SAbs) = (\SPProj, \SAbs).$
\end{enumerate}
In each case, the category on the left is contained in $\modl S,$ so that the category on the right is a definable subcategory that contains $\SAbs.$ Proposition~\ref{dual Xu} implies that ${^{\perp_1}}\SAbs = \SPProj$ so that the category on the left is necessarily a subcategory of $\modl S.$ The right category in the second cotorsion pair is given by $K^{\perp_1} = S/K^{\perp_2} = \T_I.$ 

Among these two lists of cotorsion pairs, the only ones that are not associated to a tilting or cotilting class are the two that appear as part of the cotorsion triple $(S/K^{\perp_1} \cap \SPProj, S/K^{\perp_1}, S/K^{\perp_1} \cap \SPInj).$

\bibliographystyle{plain}
\bibliography{references}

\begin{thebibliography}{10}

\bibitem{AIR}
T.~Adachi, O.~Iyama, and I.~Reiten.
\newblock $\tau$-tilting theory.
\newblock {\em Compos. Math}, 150(3):415--452, 2014.

\bibitem{AHr}
A.~Angeleri~H{\"u}gel and M.~Hrbek.
\newblock Silting modules over commutative rings.
\newblock {\em International Mathematics Research Notices}, page rnw147, 2016.

\bibitem{AMV}
L.~Angeleri~H{\"u}gel, F.~Marks, and J.~Vit{\'o}ria.
\newblock Silting modules.
\newblock {\em Int. Math. Res. Notices}, 4:1251--1284, 2016.

\bibitem{AS}
M.~Auslander and S.O. Smal{\o}.
\newblock Almost split sequences in subcategories.
\newblock {\em Journal of Algebra}, 69(2):426--454, 1981.

\bibitem{B03}
S.~Bazzoni.
\newblock Cotilting modules are pure-injective.
\newblock {\em Proc. Amer. Math. Soc.}, 131(12):3665--3672 (electronic), 2003.

\bibitem{BH}
S.~Bazzoni and D.~Herbera.
\newblock One dimensional tilting modules are of finite type.
\newblock {\em Algebr. Represent. Theory}, 11(1):43--61, 2008.

\bibitem{BBD}
A.A. Be{\u\i}linson, J.~Bernstein, and P.~Deligne.
\newblock Faisceaux pervers.
\newblock In {\em Analysis and topology on singular spaces, {I} ({L}uminy,
  1981)}, volume 100 of {\em Ast\'erisque}, pages 5--171. Soc. Math. France,
  Paris, 1982.

\bibitem{BBT}
C.~Bessenrodt, H.H. Brungs, and G.~T{\"o}rner.
\newblock {\em Right chain rings}.
\newblock Fachber., Univ. Duisburg Gesamthochschule, 1986.

\bibitem{BD}
H.~Brungs and N.~Dubrovin.
\newblock A classification and examples of rank one chain domains.
\newblock {\em Transactions of the American Mathematical Society},
  355(7):2733--2753, 2003.

\bibitem{CMT}
R.~Colpi, F.~Mantese, and A.~Tonolo.
\newblock When the heart of a faithful torsion pair is a module category.
\newblock {\em Journal of Pure and Applied Algebra}, 215(12):2923 -- 2936,
  2011.

\bibitem{CT}
R.~Colpi and J.~Trlifaj.
\newblock Tilting modules and tilting torsion theories.
\newblock {\em J. Algebra}, 178(2):614--634, 1995.

\bibitem{CB94}
W.W. Crawley-Boevey.
\newblock Locally finitely presented additive categories.
\newblock {\em Comm. Algebra}, 22(5):1641--1674, 1994.

\bibitem{CB}
W.W. Crawley-Boevey.
\newblock Infinite dimensional modules in the representation theory of finite
  dimensional algebras.
\newblock In {\em Algebras and Modules I, Canadian Math. Soc. Conf. Proc},
  volume~23, pages 29--54, 1998.

\bibitem{DP}
N.~Dubrovin and G.~Puninski.
\newblock Classifying projective modules over some semilocal rings.
\newblock {\em Journal of Algebra and Its Applications}, 06(05):839--865, 2007.

\bibitem{EJ}
E.E. Enochs and O.M.G. Jenda.
\newblock {\em Relative homological algebra}, volume~30 of {\em de Gruyter
  Expositions in Mathematics}.
\newblock Walter de Gruyter \& Co., Berlin, 2000.

\bibitem{FS}
L.~Fuchs and L.~Salce.
\newblock {\em Modules over non-{N}oetherian domains}, volume~84 of {\em
  Mathematical Surveys and Monographs}.
\newblock American Mathematical Society, Providence, RI, 2001.

\bibitem{FuSh}
K.R. Fuller and W.A. Shutters.
\newblock Projective modules over non-commutative semilocal rings.
\newblock {\em T\^ohoku Math. J. (2)}, 27(3):303--311, 1975.

\bibitem{Glaz89}
S.~Glaz.
\newblock {\em Commutative coherent rings}, volume 1371 of {\em Lecture Notes
  in Mathematics}.
\newblock Springer-Verlag, Berlin, 1989.

\bibitem{GT12}
R.~G{\"o}bel and J.~Trlifaj.
\newblock {\em Approximations and endomorphism algebras of modules. {V}olume
  1}, volume~41 of {\em de Gruyter Expositions in Mathematics}.
\newblock Walter de Gruyter GmbH \& Co. KG, Berlin, extended edition, 2012.
\newblock Approximations.

\bibitem{GT}
R{\"u}diger G{\"o}bel and Jan Trlifaj.
\newblock {\em Approximations and endomorphism algebras of modules}, volume~41
  of {\em de Gruyter Expositions in Mathematics}.
\newblock Walter de Gruyter GmbH \& Co. KG, Berlin, 2006.

\bibitem{G}
K.R. Goodearl.
\newblock {\em Von Neumann regular rings}.
\newblock Krieger Pub Co, 1991.

\bibitem{Groth}
A.~Grothendieck.
\newblock \'{E}l\'ements de g\'eom\'etrie alg\'ebrique. {IV}. \'{E}tude locale
  des sch\'emas et des morphismes de sch\'emas {IV}.
\newblock {\em Inst. Hautes \'Etudes Sci. Publ. Math.}, (32):361, 1967.

\bibitem{Gru}
L.~Gruson.
\newblock Dimension homologique des modules plats sur an anneau commutatif
  noeth\'erien.
\newblock In {\em Symposia {M}athematica, {V}ol. {XI} ({C}onvegno di {A}lgebra
  {C}ommutativa, {INDAM}, {R}ome, 1971)}, pages 243--254. Academic Press,
  London, 1973.

\bibitem{GH}
P.A. Guil~Asensio and I.~Herzog.
\newblock Indecomposable flat cotorsion modules.
\newblock {\em Journal of the London Mathematical Society}, 76(3):797--811,
  2007.

\bibitem{HRS}
D.~Happel, I.~Reiten, and O.S. Smal{\o}.
\newblock Tilting in abelian categories and quasitilted algebras.
\newblock {\em Mem. Amer. Math. Soc.}, 120(575):viii+ 88, 1996.

\bibitem{HP}
D.~Herbera and P.~P{\v{r}}{\'{\i}}hoda.
\newblock Big projective modules over noetherian semilocal rings.
\newblock {\em J. Reine Angew. Math.}, 648:111--148, 2010.

\bibitem{IvoAR}
I.~Herzog.
\newblock The auslander-reiten translate.
\newblock In {\em Abelian Groups and Noncommutative Rings: A Collection of
  Papers in Memory of Robert B. Warfield, Jr.}, volume 130 of {\em Contemporary
  Mathematics}. 1992.

\bibitem{HR}
I.~Herzog and Ph. Rothmaler.
\newblock Pure projective approximations.
\newblock {\em Math. Proc. Cambridge Philos. Soc.}, 146(1):83--94, 2009.

\bibitem{HeRo}
I.~Herzog and Ph. Rothmaler.
\newblock When cotorsion modules are pure injective.
\newblock {\em Journal of Mathematical Logic}, 09(01):63--102, 2009.

\bibitem{HoJo}
H.~Holm and P.~J{\o}rgensen.
\newblock Covers, precovers, and purity.
\newblock {\em Illinois J. Math.}, 52(2):691--703, 2008.

\bibitem{K2}
I.~Kaplansky.
\newblock Elementary divisors and modules.
\newblock {\em Trans. Amer. Math. Soc.}, 66:464--491, 1949.

\bibitem{Ka}
I.~Kaplansky.
\newblock Projective modules.
\newblock {\em Ann. of Math (2)}, 68:372--377, 1958.

\bibitem{HK}
H.~Krause.
\newblock {\em The spectrum of a module category}, volume 707.
\newblock American Mathematical Soc., 2001.

\bibitem{Lenz83}
H.~Lenzing.
\newblock Homological transfer from finitely presented to infinite modules.
\newblock volume 1006 of {\em Lecture Notes in Math.}, pages 734--761.
  Springer, 1983.

\bibitem{MR}
J.C. McConnell and J.C. Robson.
\newblock {\em Noncommutative {N}oetherian rings}, volume~30 of {\em Graduate
  Studies in Mathematics}.
\newblock American Mathematical Society, Providence, RI, revised edition, 2001.
\newblock With the cooperation of L. W. Small.

\bibitem{Na}
M.~Nagata.
\newblock {\em Local rings}.
\newblock Interscience Tracts in Pure and Applied Mathematics, No. 13.
  Interscience Publishers a division of John Wiley \& Sons\, New York-London,
  1962.

\bibitem{O}
B.L. Osofsky.
\newblock {\em Homological dimensions of modules}.
\newblock Number~12. American Mathematical Soc., 1973.

\bibitem{PS}
C.E. Parra and M.~Saor\'{i}n.
\newblock Direct limits in the heart of a t-structure: The case of a torsion
  pair.
\newblock {\em Journal of Pure and Applied Algebra}, 219(9):4117 -- 4143, 2015.

\bibitem{PS16}
C.E. Parra and M.~Saor\'{i}n.
\newblock Addendum to “direct limits in the heart of a t-structure: The case
  of a torsion pair” [j. pure appl. algebra 219 (9) (2015) 4117–4143].
\newblock {\em Journal of Pure and Applied Algebra}, 220(6):2467 -- 2469, 2016.

\bibitem{PT}
D.~Posp{\'{\i}}{\v{s}}il and J.~Trlifaj.
\newblock Tilting for commutative noetherian rings of global dimension two.
\newblock 2013.

\bibitem{PSL}
M.~Prest.
\newblock {\em Purity, spectra and localisation}, volume 121.
\newblock Cambridge University Press, 2009.

\bibitem{Pun}
G.~Puninski.
\newblock {\em Serial Rings}.
\newblock Springer Netherlands, 2001.

\bibitem{RS}
J.~Rada and M.~Saor\'{i}n.
\newblock Rings characterized by (pre)envelopes and (pre)covers of their
  modules.
\newblock {\em Communications in Algebra}, 26(3):899--912, 1998.

\bibitem{RG}
M.~Raynaud and L.~Gruson.
\newblock Crit\`eres de platitude et de projectivit\'e. {T}echniques de
  ``platification'' d'un module.
\newblock {\em Invent. Math.}, 13:1--89, 1971.

\bibitem{Sid}
M.F. Siddoway.
\newblock On endomorphism rings of modules over {H}enselian rings.
\newblock {\em Comm. Algebra}, 18(5):1323--1335, 1990.

\bibitem{St06}
J.~{\v{S}}{\v{t}}ov{\'{\i}}{\v{c}}ek.
\newblock All {$n$}-cotilting modules are pure-injective.
\newblock {\em Proc. Amer. Math. Soc.}, 134(7):1891--1897 (electronic), 2006.

\bibitem{Tr07}
J.~Trlifaj.
\newblock Infinite dimensional tilting modules and cotorsion pairs.
\newblock In {\em Handbook of tilting theory}, volume 332 of {\em London Math.
  Soc. Lecture Note Ser.}, pages 279--321. Cambridge Univ. Press, Cambridge,
  2007.

\bibitem{JanIM}
J.~\v{S}aroch.
\newblock On the non-existence of right almost split maps.
\newblock {\em Invent. Math.}, 2016.

\bibitem{W}
J.~Wei.
\newblock Large support $\tau$-tilting modules.
\newblock {\em preprint}, 2014.

\bibitem{Wh}
J.M. Whitehead.
\newblock Projective modules and their trace ideals.
\newblock {\em Comm. Algebra}, 8(19):1873--1901, 1980.

\bibitem{Zg}
M.~Ziegler.
\newblock Model theory of modules.
\newblock {\em Ann. Pure Appl. Logic}, 26(2):149--213, 1984.

\end{thebibliography}
\end{document}